\documentclass[reqno]{amsart}

\usepackage{mathrsfs}
\usepackage{enumitem}
\usepackage{amssymb}
\usepackage{enumitem}
\usepackage{thm-restate}
\usepackage{tikz}
\usetikzlibrary{shapes.geometric}
\usepackage{undertilde}
\usepackage[backref]{hyperref}

\theoremstyle{definition}
\newtheorem{thm}{Theorem}[section]
\newtheorem{cor}[thm]{Corollary}
\newtheorem{prop}[thm]{Proposition}
\newtheorem{lem}[thm]{Lemma}

\newtheorem{defn}[thm]{Definition}
\newtheorem{ex}[thm]{Example}

\newtheorem{Q}[thm]{Question}
\newtheorem*{claim*}{Claim}

\numberwithin{equation}{section}

\newcommand{\ddownarrow}{{\hspace{-.1mm}\downarrow}}

\begin{document}










\title[Randomness, Disintegrations, and Convergence to the Truth]
{Algorithmic Randomness, Effective Disintegrations, and Rates of Convergence to the Truth}
\author{Simon M. Huttegger}

\address{Department of Logic and Philosophy of Science \\ 5100 Social Science Plaza \\ University of California, Irvine \\ Irvine, CA 92697-5100, U.S.A.}

\email{shuttegg@uci.edu}

\urladdr{http://faculty.sites.uci.edu/shuttegg/}

\author{Sean Walsh}

\address{Department of Philosophy \\ University of California, Los Angeles \\ 390 Portola Plaza, Dodd Hall 321 \\ Los Angeles, CA 90095-1451}

\email{walsh@ucla.edu}

\urladdr{http://philosophy.ucla.edu/person/sean-walsh/}

\author{Francesca Zaffora Blando}

\address{Department of Philosophy\\ Carnegie Mellon University \\Baker Hall 161\\
5000 Forbes Avenue\\ Pittsburgh, PA 15213}

\email{fzaffora@andrew.cmu.edu}

\thanks{Many thanks to Jeremy Avigad, Peter Cholak, Johanna Franklin, Alexander Kastner, Josiah Lopez-Wild, Christopher Porter, Michael Rescorla, and Jason Rute for discussion and feedback.}


\date{}

\begin{abstract}
\makeatletter\phantomsection\def\@currentlabel{(abstract)}\makeatother
L\'evy's Upward Theorem says that the conditional expectation of an integrable random variable converges with probability one to its true value with increasing information. In this paper, we use methods from effective probability theory to characterise the probability one set along which convergence to the truth occurs, and the rate at which the convergence occurs. We work within the setting of computable probability measures defined on computable Polish spaces and introduce a new general theory of effective disintegrations. We use this machinery to prove our main results, which (1)~identify the points along which certain classes of effective random variables converge to the truth in terms of certain classes of algorithmically random points, and which further~(2) identify when computable rates of convergence exist. Our convergence results significantly generalize earlier results within a unifying novel abstract framework, and there are no precursors of our results on computable rates of convergence. Finally, we make a case for the importance of our work for the foundations of Bayesian probability theory.
\end{abstract}

\maketitle

\tableofcontents

\section{Introduction}\label{sec:intro}

Measure-theoretic probability was developed in the early 20th Century in response to pressing problems in statistical physics, astronomy, and pure mathematics, and today it is used throughout the mathematical sciences.\footnote{For a historical survey, see \cite{Plato1994}.} What proved to be an especially significant conceptual progress was the ability to say that certain properties are true with probability one. Early examples include Borel's strong law of large numbers, irrational rotations of the unit interval, Birkhoff's ergodic theorem, and Poincar\'e's recurrence theorem. It is often unclear, however, what these sets are. That is to say, measure-theoretic results only assert the existence of certain sets of probability one but fail to characterise the points that are elements of those sets. This was pointed out as early as 1916 by Weyl, who insisted that a deeper understanding of the sets involved in zero-one laws was necessary in order to interpret the results of measure-theoretic probability.\footnote{\cite{Weyl1916}.}

The theory of algorithmic randomness involves a fine-grained classification of different measure one sets, with the primary exemplars being the Martin-L\"of random points, the Schnorr random points, and the Kurtz random points.\footnote{The original papers of Martin-L\"of, Schnorr, and Kurtz are: \cite{Martin-Lof1966aa}, \cite{Martin-Lof1970ab}, \cite{Schnorr1971aa}, \cite{Schnorr1977ab}, \cite{Kurtz1981-eb}. There are now several comprehensive references on algorithmic randomness, including \cite{Li1997aa}, \cite{Nies2008aa}, \cite{Downey2010aa}, \cite{Shen2017-ze}.} Originally this was done for the uniform ``fair coin'' measure on Cantor space (the space of infinite sequences of 0's and 1's) and the famous results pertained to algorithmic incompressibility and the Turing degrees.\footnote{For instance, the Levin-Schnorr characterisation of Martin-L\"of randomness in terms of initial segment complexity, and the Ku\v{c}era-G\'acs proof that every Turing degree is below the degree of a Martin-L\"of random. See, e.g., \cite[Theorem 6.3.10 p. 239, Theorem 8.3.2 p. 326]{Downey2010aa} for statement and references.} However, the theory has been recently developed for a more general class of computable probability measures on computable spaces, by authors such as G\'acs, Hoyrup and Rojas, Reimann, Rute, and Miyabe.\footnote{\cite{Gacs2005-xu}, \cite{Hoyrup2008-oj}, \cite{Hoyrup2009-pl}, \cite{Reimann2010-dt}, \cite{Rute2012aa}, \cite{Miyabe2014-ih}, \cite{Hoyrup2021-qj} (listed in rough chronological order).} A related recent trend has been showing that effectivized versions of classical theorems on almost sure convergence prove convergence exactly on various classes of algorithmically random points.\footnote{\cite{Brattka2015-ka}, \cite{Miyabe2016-jo}. The latter is, in part, a survey and contains many further references.} This arguably contributes to the deeper understanding along the lines suggested by Weyl. Further, this recent trend suggests reconceiving of the various notions of algorithmic randomness less as on a par with rival conceptual analyses of a pre-theoretic phenomenon, \emph{\`a la} the Church-Turing thesis, and more as delineations of extensionally and conceptually distinct kinds of probability one events.\footnote{This point is due to \cite{Porter2016-au}.} Or, if one puts the point in terms of the corresponding null sets, the various notions demarcate different types of impossibility that occur throughout measure-theoretic mathematics and its many applications.

Our main theorems (Theorems~\ref{thm:newnewlevy}, \ref{thm:newnewlevyrates}, \ref{thm:drlevy}, \ref{thm:drlevy:drates}, \ref{thm:doobmax}) contribute to this recent literature by characterising, in terms of algorithmic randomness, the points at which L\'evy's Upward Theorem holds for various classes of effective random variables, as well as providing information about the rates of convergence to the truth. 

Let us recall the classical statement of L\'evy's Theorem.\footnote{\cite[p. 134]{Williams1991aa}, \cite[\S{41} pp. 128 ff]{Levy1938-ya}.} Suppose $(X,\mathscr{F}, \nu)$ is a probability triple. Let $\mathscr{F}_1, \mathscr{F}_2, \hdots$ be an increasing sequence of sub-$\sigma$-algebras of $\mathscr{F}$ whose union generates $\mathscr{F}$. Then \emph{L\'evy's Upward Theorem} states that one has $\mathbb{E}_{\nu}[f \mid \mathscr{F}_n]\rightarrow f$ both $\nu$-a.s. and in $L_1(\nu)$, for any $\mathscr{F}$-measurable function $f$ in $L_1(\nu)$. In this, $\mathbb{E}_{\nu}[f \mid \mathscr{G}]$ denotes the conditional expectation of $f$ relative to~$\mathscr{G}$, which, recall, is defined as the $\nu$-a.s. unique $\mathscr{G}$-measurable function $g$ such that $\int_A g \; d\nu = \int_A f\; d\nu$ for all events $A$ in the sub-$\sigma$-algebra $\mathscr{G}$ of $\mathscr{F}$.

The convergence in L\'{e}vy's Upward Theorem is one of the cornerstones of Bayesian epistemology.\footnote{\cite[pp. 144 ff]{Earman1992aa}, \cite[pp. 28-29]{Howson2006aa}.} The random variable $f$ can be thought of as a quantity that a Bayesian agent, whose degrees of belief are captured by the underlying probability measure~$\nu$, is trying to estimate by repeatedly performing an experiment. The quantity $\mathbb{E}_{\nu}[f\mid\mathscr{F}_n]$ can be seen as encoding the agent's opinions regarding the value of $f$ after having observed the outcomes of the first $n$ experiments. L\'{e}vy's Upward Theorem then implies that, with probability one, the Bayesian agent's opinions regarding the value of $f$ will converge to $f$'s true value in the limit.

To be able to characterise the $\nu$-measure one set on which $\mathbb{E}_{\nu}[f \mid \mathscr{F}_n]\rightarrow f$, one needs to choose versions of $\mathbb{E}_{\nu}[f \mid \mathscr{F}_n]$ and $f$. It seems natural to focus attention on classes of effective random variables $f$ defined relative to spaces~$X$ and probability measures~$\nu$ which are themselves computable. For, computability seems like a natural constraint to place on our Bayesian agent, and many of the examples of probability measures and random variables that occur in practice and applications are computable. However, the Bayesian perspective recommends few other general constraints on what is eligible to be a credence or a~prior. Hence it is important to develop the theory for a maximally broad class of computable spaces and probability measures.\footnote{The distinctive status of the computability hypothesis which we are suggesting raises a host of interesting and complex conceptual questions, ranging from the nature of cognition to the character of inductive inference. We put these issues aside here.}

\subsection{Effective probability and algorithmic randomness}\label{subsec:core}

The computable Polish spaces with computable probability measures are such an appropriately general class of spaces and measures. In this brief section we collect together the few definitions we need about their theory. The reader already familiar with these concepts can easily skip to the next section (\S\ref{sec:disintegration}).

A \emph{Polish space} is a topological space which is separable and completely metrizable. All the paradigmatic spaces such as the reals and their products and their closed and open subspaces are Polish, and similarly for Cantor space. Descriptive set theory takes as its subject matter the Borel and projective subsets of Polish spaces.\footnote{\cite{Kechris1995}, \cite{Moschovakis2009-vf}.} When topological considerations are salient or when one needs i.i.d. sequences with prescribed distributions, it is often assumed in contemporary probability theory that the sample space is a Polish space or a Borel subset thereof.\footnote{A Borel subset of a Polish space together with its Borel subsets is known as a \emph{standard Borel space} in descriptive set theory (cf. \S\cite[Definition 12.5, Corollary 13.4]{Kechris1995}). For representative examples of standard Borel spaces within probability, see e.g. \cite[p. 51]{Durrett2010-lj}, \cite[p. 7, pp. 561 ff]{Kallenberg2002-gw}. For a classic probability text that foregrounds standard Borel spaces, see \cite[Chapter 1]{Parthasarathy1967-ad}.}

A \emph{computable Polish space $X$} is a Polish space with a distinguished countable dense set $x_0, x_1, \ldots$ and a distinguished complete compatible metric $d$ such that the distance $d(x_i,x_j)$ between any two elements of the countable dense set is a computable real, uniformly in $i,j\geq 0$.\footnote{A standard reference for computable Polish spaces is \cite[Chapter 3]{Moschovakis2009-vf}. A comparison to the Weihrauch approach to computable analysis is given in \cite{Gregoriades2017-xr}. One can view the treatment of metric spaces in \cite{Simpson2009aa} as an axiomatization of Polish spaces and reals which are computable in an oracle. Finally, the study of computable Polish spaces in and of themselves is distinct from effective descriptive set theory, which usually refers to techniques for proving results about all Borel sets by first proving it for lightface Borel sets in Baire space (the most famous example of this is the Glimm-Effros dichotomy (cf. \cite[Chapter 6]{Gao2008-ap})).} (The enumeration of the distinguished countable dense set can contain repetitions, and will need to do so in finite spaces.)

In a computable Polish space, an open set $U$ is \emph{c.e. open} if there is a computable function $n(\cdot)$ which enumerates a subsequence $x_{n(i)}$ of the countable dense set and a computable sequence $r_i$ of rational radii such that $U=\bigcup_i B_d(x_{n(i)},r_i)$. In this, $B_d(x,r)$ denotes the open ball with centre $x$ and radius $r$ relative to metric $d$ (when the metric $d$ is clear from context, we just write $B(x,r)$). The name ``c.e. open'' is chosen since the natural numbers are a computable Polish space with the discrete metric, and the c.e. opens in this space are precisely the computably enumerable sets of natural numbers, one of the canonical objects of the contemporary theory of computation.\footnote{See \cite{soare2016turing}, a standard reference.} Further, many of the elementary methods of studying c.e. sets (e.g., universal enumerations) extend to c.e. open sets. The complements of c.e. open sets are called \emph{effectively closed} sets (cf. \S\ref{sec:closedcont}).

In a computable Polish space, we say that a sequence $x_n\rightarrow x$ at \emph{geometric rate~$b$} if $d(x,x_n)\leq b^{-n}$ for all $n\geq 0$. We say that a sequence $x_n\rightarrow x$ \emph{fast} if $x_n\rightarrow x$ at geometric rate $b=2$. We then say that a point $x$ is \emph{computable} if there is a computable function $n(\cdot)$ which enumerates a subsequence $x_{n(i)}$ of the countable dense set such that $x_{n(i)}\rightarrow x$ fast. This subsequence is called a \emph{witness} to the computability of $x$. In Cantor space with its usual metric, the computable points are precisely the computable subsets of natural numbers. 

In the real numbers, the countable dense set is the rationals, and the above definition of computable points is precisely how Turing defined computable real numbers at the outset of the theory of computation nearly a century ago.\footnote{\cite{Turing1937-cm}.} An equivalent formalisation of computable reals is by Dedekind cuts. A real $x$ is \emph{left-c.e.} (resp. \emph{right-c.e.}) if its left Dedekind cut  $\{q\in \mathbb{Q}: q<x\}$ in the rationals is a c.e. set (resp. if its right Dedekind cut $\{q\in \mathbb{Q}: x<q\}$ in the rationals is a c.e. set). One can show that a real is computable iff it is both left-c.e. and right-c.e., and uniformly so. (An example of a left-c.e. real that is not computable is $\sum_n 2^{-f(n)}$, where $f:\mathbb{N}\rightarrow \mathbb{N}$ is an injective computable function with non-computable range.)\footnote{These and other effective aspects of real numbers are treated extensively in e.g. \cite{Rettinger2021-rm}, \cite[Chapter 5]{Downey2010aa}.}

These preliminaries in place, one can then quickly define the required core notions from algorithmic randomness and effective probability. These are all needed in order to formally state our main theorems, but one might restrict oneself to (\ref{defn:core:1})-(\ref{defn:core:10}) on a first pass and come back to the others as needed. 
\begin{defn}\label{defn:core} (Core notions)
\begin{enumerate}[leftmargin=*]
\item \label{defn:core:1} A function $f:X\rightarrow (-\infty,\infty]$ is  \emph{lower semi-computable} (abbreviated \emph{lsc}) if for all rational $q$, the set $f^{-1}(q,\infty]$ is uniformly c.e. open.

\item \label{defn:core:1a} A function $f:X\rightarrow [-\infty,\infty)$ is  \emph{upper semi-computable} (abbreviated \emph{usc}) if for all rational $q$, the set $f^{-1}[-\infty,q)$ is uniformly c.e. open.

\item \label{defn:core:2} A probability measure $\nu$ is \emph{computable} if $\nu(U)$ is uniformly left-c.e. as $U$ ranges over c.e. opens.
   
\item  \label{defn:core:6}   Given a computable probability measure $\nu$ and a computable real $p\geq 1$, a function $f:X\rightarrow [0,\infty]$ is \emph{an $L_p(\nu)$ Schnorr test} if it is lsc and if $\|f\|_p$ is a computable real, where this denotes the $p$-norm $\|f\|_p=(\int \left|f\right|^p \; d\nu)^{\frac{1}{p}}$.
    
 \item \label{defn:core:7}   Given a computable probability measure $\nu$ and a computable real $p\geq 1$, a function $f:X\rightarrow [0,\infty]$ is \emph{an $L_p(\nu)$ Martin-L\"of test} if it is lsc and $\|f\|_p<\infty$.
    
\item  \label{defn:core:8}    A point $x$ in $X$ is \emph{Kurtz random relative to $\nu$} (abbreviated $\mathsf{KR}^{\nu}(X)$) if $x$ is in every c.e. open $U$ with $\nu(U)=1$.
    
 \item \label{defn:core:9}    A point $x$ in $X$ is \emph{Schnorr random relative to $\nu$} (abbreviated $\mathsf{SR}^{\nu}(X)$) if $f(x)<\infty$ for any $L_1(\nu)$ Schnorr test $f$ (equivalently, for any $L_p(\nu)$ Schnorr test, for $p\geq 1$ computable).

 \item  \label{defn:core:10}   A point $x$ in $X$ is \emph{Martin-L\"of random relative to $\nu$} (abbreviated $\mathsf{MLR}^{\nu}(X)$) if $f(x)<\infty$ for any $L_1(\nu)$ Martin-L\"of test $f$ (equivalently, for any $L_p(\nu)$ Martin-L\"of test, for $p\geq 1$ computable).

\item \label{defn:core:3.1}  A \emph{computable basis} $\mathscr{B}$ for $X$ is a computable sequence of c.e. opens such that every c.e. open can be effectively written as a union of elements in $\mathscr{B}$.

\item \label{defn:core:3}  If $\nu$ is a computable probability measure, then a \emph{$\nu$-computable basis} $\mathscr{B}$ for $X$ is a computable basis such that (i) finite unions from $\mathscr{B}$ uniformly have $\nu$-computable measure, and (ii) each c.e. open in $\mathscr{B}$ is uniformly paired with an effectively closed superset of the same $\nu$-measure. If $\nu$ is clear from context, we simply say \emph{measure computable basis} instead of $\nu$-computable basis.

 \item \label{defn:core:4}  A sub-$\sigma$-algebra $\mathscr{F}$ of the Borel sets on $X$ is \emph{$\nu$-effective} if it is generated by a computable sequence of events  $\{A_{m}: m\geq 0\}$ from the algebra $\mathscr{A}$ generated by a $\nu$-computable basis $\mathscr{B}$.\footnote{When working with $\mathscr{A}$, we assume that we are working with the codes for Boolean combinations of elements of $\mathscr{B}$, and only by extension with the sets that they define. This is because there are some spaces where Boolean algebra structure on the quotient is not computable.} We say that $\mathscr{F}$ is \emph{generated} by $\{A_{m}: m\geq 0\}$.
  
\item  \label{defn:core:5}  A \emph{full $\nu$-effective filtration} $\mathscr{F}_n$  (resp. \emph{almost-full $\nu$-effective filtration}) is an increasing sequence $\mathscr{F}_n$ of uniformly $\nu$-effective sub-$\sigma$-algebras generated by a uniformly computable sequence $\{A_{n,m}: m\geq 0\}$ from the algebra $\mathscr{A}$ generated by a $\nu$-computable basis $\mathscr{B}$, which is further equipped with a uniform procedure for going from a c.e. open $U$ to a computable sequence $A_{n_i, m_i}$ such that $U=\bigcup_i A_{n_i, m_i}$ (resp.  $U=\bigcup_i A_{n_i, m_i}$ on $\mathsf{KR}^{\nu}(X)$).

\item \label{defn:core:15} If $x$ is a point of the computable Polish space $X$ and $Y$ is a subset of Baire space (the space of all functions from natural numbers to natural numbers), then \emph{$x$ weakly computes an element of $Y$} if, for every sequence $x_{n(i)}$ from the countable dense set of $X$ such that $x_{n(i)}\rightarrow x$ fast, there is $y$ in $Y$ which is Turing reducible to the function $i\mapsto n(i)$. If $Y=\{y\}$, then we just say that \emph{$x$ weakly computes $y$}.\footnote{One can extend Turing reducibility from a relation between sets of natural numbers to a relation between closed subsets of Baire space. In this setting, the notion of weak computation is called Muchnik reducibility, and is contrasted to a strong uniform notion called Medvedev reducibility. See \cite{Simpson2005-pa}, \cite{Hinman2012-gz} for introduction and references, although this theory is usually focused on effectively closed sets. Given a point $x$, the set of functions $i\mapsto n(i)$ such that $x_{n(i)}\rightarrow x$ at a fixed rate, in the sense of (\ref{defn:core:16}), is a closed subset of Baire space.}$^{,\;}$
\footnote{If $X$ is Cantor space or the reals, then for each point $x$ of the space, there is a sequence $i\mapsto n(i)$ of least Turing degree such that $x_{n(i)}\rightarrow x$ fast. In these settings, computational properties of the point of the space usually refer to those of this sequence. However, there are spaces for which there are points with no sequence of least Turing degree. See Miller \cite{Miller2004-cf}.}

\item \label{defn:core:15.5} If $x$ is a point of the computable Polish space $X$, and $\mathcal{C}$ is any collection of Turing degrees (equivalence classes of elements of Baire space under Turing reducibility), then we say that \emph{$x$ is in $\mathcal{C}$} if there is some $i\mapsto n(i)$ whose Turing degree is in $\mathcal{C}$ such that $x_{n(i)}\rightarrow x$ fast, where $x_j$ again enumerates the countable dense set. In the case where $\mathcal{C}$ just consists of the computable degree, note that $x$ is in $\mathcal{C}$ iff $x$ is computable as a point of $X$.

\item \label{defn:core:16} Suppose $y_n, y$ are elements in a metric space $Y$ such that $y_n\rightarrow y$. Then a \emph{rate of convergence} for $y_n\rightarrow y$ is a function $m:\mathbb{Q}^{>0}\rightarrow \mathbb{N}$ such that for all rational $\epsilon>0$ and all $n\geq m(\epsilon)$ one has $d(y_n,y)<\epsilon$.\footnote{If $y_n\rightarrow y$ at geometric rate $b>1$ in a computable Polish space, then one defines a rate in the sense of (\ref{defn:core:16}) by setting $m(\epsilon)=n$ for the least $n$ such that $b^{-n}<\epsilon$.} Often in practice we use the case where $Y$ is the reals and $y_n=f_n(x)$ and $y=f(x)$, where $f_n,f$ are real-valued functions. A synonym for \emph{rate} is \emph{modulus}, and so we often use the $m$ variable for rates.
\end{enumerate}
\end{defn}

For the algorithmic randomness notions in (\ref{defn:core:8})-(\ref{defn:core:10}), we just write $\mathsf{KR}^{\nu}$ instead of $\mathsf{KR}^{\nu}(X)$ when $X$ is clear from context; and similarly for $\mathsf{SR}^{\nu}$ and $\mathsf{MLR}^{\nu}$. For $\sigma$-algebras $\mathscr{F}$, it is always understood that they are sub-$\sigma$-algebras of the Borel $\sigma$-algebra, and when $\nu$ is clear from context we just say \emph{effective} instead of \emph{$\nu$-effective}. 

Algorithmic randomness is often formulated in terms of effective null sets, called  \emph{sequential tests}. But the definitions given above in terms of \emph{integral tests} are easier to work with in our setting and are known to be equivalent to the sequential definitions, by theorems of Levin and Miyabe.\footnote{\cite{Levin1976}, \cite[Theorem 3.5]{Miyabe2013}.} 

Before turning to disintegrations, it is helpful to introduce notational conventions regarding versions of integrable functions vs. equivalence classes thereof. In the following definition, the $\sigma$-algebra on $[-\infty,\infty]$ is simply $\{B\cup C: B\subseteq \mathbb{R} \mbox{ Borel}, C\subseteq \{-\infty,\infty\}\}$, and similarly for $[0,\infty]$.

\begin{defn}\label{defn:versions} (Conventions on functions defined pointwise vs. functions defined up to $\nu$-a.s. equivalence)

Suppose $X$ is a Polish space and suppose that $\nu$ is a finite non-negative measure on the Borel events of $X$. Then we define:

$\mathbb{L}_p(\nu)$ is the set of pointwise defined Borel measurable functions $f:X\rightarrow [-\infty,\infty]$ such that $\|f\|_p<\infty$.

 $\mathbb{L}^+_p(\nu)$ is the set of pointwise defined Borel measurable functions $f:X\rightarrow [0,\infty]$ such that $\|f\|_p<\infty$.

 $L_p(\nu)$ is the set of equivalence classes of elements of $\mathbb{L}_p(\nu)$ under $\nu$-a.s. equivalence. That is, $L_p(\nu)$ is the classical Banach space with norm $\|\cdot\|_p$.

 $L^+_p(\nu)$ is the set of equivalence classes of elements of $\mathbb{L}^+_p(\nu)$ under $\nu$-a.s. equivalence. That is, $L^+_p(\nu)$ is a positive cone in the Banach space $L_p(\nu)$.

\end{defn}

Then $\mathbb{L}_p(\nu)$ projects onto $L_p(\nu)$ by sending a function to its equivalence class, and likewise $\mathbb{L}^+_p(\nu)$ projects onto $L^+_p(\nu)$. Note that $L_p(\nu)$ Schnorr tests and Martin-L\"of tests from Definition~\ref{defn:core}(\ref{defn:core:6})-(\ref{defn:core:7}) are elements of $\mathbb{L}_p^+(\nu)$, and they are elements of $L_p^+(\nu)$ only after passing to the equivalence class.

\subsection{Classical and effective disintegrations}\label{sec:disintegration}

We use disintegrations for the versions $\mathbb{E}_{\nu}[f\mid \mathscr{F}_n]$ of conditional expectation. While, classically, conditional expectation is defined only $\nu$-a.s., in order to characterise algorithmic randomness notions in terms of L\'evy's Upward Theorem, we need to select specific versions of conditional expectation. The concept of a disintegration provides a very general way of making such selections. It is due to Rohlin\footnote{\cite{Rohlin1949-ly}, \cite{Rohlin1962-rn}, \cite{Rokhlin1967-op}.} and is routinely used today in ergodic theory and optimal transport,\footnote{It is often used in the proof of the Ergodic Decomposition Theorem and the Gluing Lemma. See \cite[154]{Einsiedler2010-yb}, \cite[182]{Santambrogio2015-li}.} and it is closely related to conditional probability distributions.\footnote{\cite{Chang1997-ip}, \cite[\S{5.3}]{Pollard2002-ih}.} 

Suppose that $X$ is a Polish space, $\nu$ is a probability measure on the Borel sets of $X$ and $\mathscr{F}$ is a countably generated sub-$\sigma$-algebra of the Borel $\sigma$-algebra. Define the equivalence relation $\sim_{\mathscr{F}}$ on $X$ by $x\sim_{\mathscr{F}} x^{\prime}$ iff, for all $A$ in $\mathscr{F}$, one has $x$ in $A$ iff $x^{\prime}$ in $A$, and let $[x]_{\mathscr{F}}$ be the corresponding equivalence class.\footnote{Since we are focused on L\'evy's Upward Theorem, we are focusing on countably generated sub-$\sigma$-algebras $\mathscr{F}$ of the Borel $\sigma$-algebra. Note that this has the consequence that the relation $\sim_{\mathscr{F}}$ is a smooth Borel equivalence relation (cf. \cite[\S{5.4}]{Gao2008-ap}). More complicated Borel equivalence relations occur naturally in nearby topics. For instance, Rute examines L\'evy's Downward Theorem (cf. \cite[Theorem 11.2]{Rute2012aa}, \cite[Theorem 14.4]{Williams1991aa}), which in Cantor space results naturally in the sub-$\sigma$-algebra of $E_0$-invariant events, where $E_0$ is the Borel equivalence relation featuring in the Glimm-Effros dichotomy (cf. \cite[Definition 6.1.1]{Gao2008-ap}).}

Let $\mathcal{M}^+(X)$ be the Polish space of non-negative Borel measures on $X$ (cf. \S\ref{sec:efdst:measure}). For Borel measurable $\rho: X\rightarrow \mathcal{M}^+(X)$, whose action is written as $x\mapsto \rho_x$, we define the partial map  
\begin{equation}\label{eqn:disintegrate}
\mathbb{E}_{\nu}[\cdot \mid \mathscr{F}](\cdot): \mathbb{L}_1(\nu)\times X\dashrightarrow [-\infty,\infty] \hspace{3mm} \mbox{ by } \hspace{3mm} \mathbb{E}_{\nu}[f \mid \mathscr{F}](x) = \int f(v) \; d\rho_{x}(v)
\end{equation}
This map is a version, that is, it is partially defined on all pairs $(f,x)$. Note that it is totally defined on $\mathbb{L}_1^+(\nu)\times X$, with range $[0,\infty]$,\footnote{Indeed, it is totally defined on all pairs ($f,x)$ where $f$ is non-negative Borel measurable. But for our purpose of defining a version of $\mathbb{E}_{\nu}[\cdot \mid \mathscr{F}]$ we only need to pay attention to when $f$ is in $\mathbb{L}_1(\nu)$.}  and it is totally defined and finite on all simple functions. It is further helpful to keep in mind that whether $ \mathbb{E}_{\nu}[f \mid \mathscr{F}](x)$ is finite depends on whether the element $f$ of $\mathbb{L}_1(\nu)$ is additionally in $\mathbb{L}_1(\rho_x)$: that is, it is integrability with respect to $\rho_x$ rather than $\nu$ which is at issue.

For a Polish space $X$, one says that a map $\rho:X\rightarrow \mathcal{M}^+(X)$ is \emph{the disintegration of $\mathscr{F}$ with respect to $\nu$} if both the following happen:\footnote{We are following the treatment of Einsiedler-Ward \cite[135]{Einsiedler2010-yb}. Since the main examples of disintegrations involve products (cf. Appendicies~\ref{sec:app:classical:disintegrations}-\ref{sec:app:effective:disintegrations}), often alternative definitions of disintegrations involve maps that axiomatize the role that the projection operators play in the paradigmatic examples. For an example of definitions along these lines, see \cite{Chang1997-ip}, \cite[\S{5.3}]{Pollard2002-ih}.}
\begin{itemize}[leftmargin=*]
\item for all $f$ in $\mathbb{L}_1(\nu)$, one has that $\mathbb{E}_{\nu}[f \mid \mathscr{F}]$ is a version of the conditional expectation of $f$ with respect to $\mathscr{F}$ and $\nu$.\footnote{Hence it is in $\mathbb{L}_1(\nu)$, and thus it is defined and finite for $\nu$-a.s. many $x$ from $X$.}
\item For $\nu$-a.s. many $x$ from $X$, one has $\rho_{x}(X)=1$ and $\rho_{x}([x]_{\mathscr{F}})=1$.
\end{itemize}
A disintegration of $\mathscr{F}$ with respect to $\nu$ exists for any countably generated sub-$\sigma$-algebra $\mathscr{F}$ of the Borel $\sigma$-algebra on $X$.\footnote{\cite[135]{Einsiedler2010-yb}. Indeed, a little more is true: one can replace $X$ by one of its Borel subsets. Further, one can relax the assumption that  $\mathscr{F}$ is countably generated, provided that one does not insist on $\rho_{x}([x]_{\mathscr{F}})=1$.} Further, it is possible to be more agnostic about the codomain of $\rho$ outside the $\nu$-measure~one set on which it outputs probability measures. See Appendix~\ref{sec:app:classical:disintegrations} for two classical examples of disintegrations.

Here is our key definition of effective disintegration:
\begin{defn}\label{defn:eff:disin} (Effective disintegrations).

Let $X$ be a computable Polish space. Let $\nu$ be a computable probability measure on $X$. Let $\mathscr{F}$ be a $\nu$-effective $\sigma$-algebra. Let $\mathsf{XR}^{\nu}$ be a $\nu$-measure one subset of $\mathsf{KR}^{\nu}(X)$. Then the map $\rho: X\rightarrow \mathcal{M}^+(X)$ is \emph{an $\mathsf{XR}^{\nu}$ disintegration of $\mathscr{F}$ with respect to $\nu$} if each of the following happen:
\begin{enumerate}[leftmargin=*]
\item \label{defn:eff:disin:1} For all $f$ in $\mathbb{L}_1(\nu)$, one has that $\mathbb{E}_{\nu}[f \mid \mathscr{F}]$ is a version of the conditional expectation of $f$ with respect to $\mathscr{F}$ and $\nu$.
\item \label{defn:eff:disin:2} For all $x$ in $\mathsf{XR}^{\nu}$ , one has that $\rho_{x}(X)=1$ and $\rho_{x}([x]_{\mathscr{F}}\cap \mathsf{XR}^{\nu})=1$.
\item \label{defn:eff:disin:3} For c.e. open $U$, the map $x\mapsto \rho_x(U)$ is uniformly lsc from $X$ to $[0,\infty)$.
\end{enumerate} 

We further define:

A \emph{Kurtz disintegration} is simply a $\mathsf{KR}^{\nu}$ disintegration.

A \emph{Schnorr disintegration} is a map which is both a $\mathsf{KR}^{\nu}$ disintegration and a $\mathsf{SR}^{\nu}$ disintegration.

 A \emph{Martin-L\"of disintegration} is a map which is a $\mathsf{KR}^{\nu}$ disintegration and a $\mathsf{SR}^{\nu}$ disintegration and a $\mathsf{MLR}^{\nu}$ disintegration.

\end{defn}

Technically, it appears possible to, e.g., be a $\mathsf{SR}^{\nu}$ disintegration but not a $\mathsf{KR}^{\nu}$ disintegration. This is due to the universal quantifier over $\mathsf{XR}^{\nu}$ at the outset of~(\ref{defn:eff:disin:2}). But this possibility does not appear to occur naturally among examples.

Due to space constraints, we have opted to focus on theory in the body of the text, and have put a brief discussion of the many interesting examples of effective disintegrations in Appendix~\ref{sec:app:effective:disintegrations}.
 
Finally, we can define: 
\begin{defn}\label{defn:dr}
Let $\mathscr{F}_n$ be an almost-full effective filtration, equipped uniformly with Kurtz disintegrations $\rho^{(n)}$. A point $x$ in $X$ is said to be \emph{density random with respect to $\rho$}, abbreviated $\mathsf{DR}^{\nu}_{\rho}(X)$, if $x$ is in $\mathsf{MLR}^{\nu}(X)$ and $\lim_n \rho_x^{(n)}(U)=\delta_x(U)$ for every c.e. open $U$.
\end{defn}
In this, $\delta_x$ is the Dirac measure centred at $x$. With the limit written as such, by the Portmanteau Theorem one sees that it is a strengthening of the weak convergence of the measures $\rho_x^{(n)}\rightarrow \delta_x$. Since we use the disintegration to define the conditional expectation, as in equation (\ref{eqn:disintegrate}) above, the limit in Definition~\ref{defn:dr} can be written equivalently as $\lim_n \mathbb{E}_{\nu}[I_U \mid \mathscr{F}_n](x)=I_U(x)$. With respect to the canonical filtration of length $n$-strings on Cantor space and its natural disintegration (cf. Example~\ref{ex:refinedpartitionconcrete}), density randomness has been a focal topic in recent literature on algorithmic randomness.\footnote{\cite{Bienvenu2014-cs}, \cite{Miyabe2016-jo}, \cite{Khan2016-qv}. In this setting with $\nu$ being the uniform measure, it is known that $\mathsf{DR}^{\nu}_{\rho}$ is a proper subset of $\mathsf{MLR}^{\nu}$. One example which shows this properness is an element $\omega$ of $\mathsf{MLR}^{\nu}$ such that $\{\omega^{\prime}: \omega^{\prime}<_{lex} \omega\}$ is c.e. open, where $<_{lex}$ is the lexicographic order.} Definition~\ref{defn:dr} is our suggestion for how to generalise this to the setting of arbitrary effective disintegrations.

\subsection{Statement of main results}\label{subsec:statementmain}

Our first main theorem is the following:

\begin{restatable}{thm}{thmmain}\label{thm:newnewlevy} (Effective Upward L\'evy Theorem for Schnorr Randomness).
Suppose that $X$ is a computable Polish space and $\nu$ is a computable probability measure. Suppose that $\mathscr{F}_n$ is an almost-full effective filtration, equipped uniformly with Kurtz disintegrations.

If $p\geq 1$ is computable, then the following four items are equivalent for $x$ in $X$: 
\begin{enumerate}[leftmargin=*]
  \item \label{levy:thm:1} $x$ is in $\mathsf{SR}^{\nu}(X)$.
  \item \label{levy:thm:2} $x$ is in $\mathsf{KR}^{\nu}(X)$ and $\lim_n \mathbb{E}_{\nu}[f\mid \mathscr{F}_n](x)=f(x)$ for all $L_p(\nu)$ Schnorr tests~$f$.
  \item \label{levy:thm:2.5} $x$ is in $\mathsf{KR}^{\nu}(X)$ and $\lim_n \mathbb{E}_{\nu}[f\mid \mathscr{F}_n](x)$ exists for all $L_p(\nu)$ Schnorr tests~$f$ and $\lim_n \mathbb{E}_{\nu}[I_U\mid \mathscr{F}_n](x)=I_U(x)$ for all c.e. opens $U$ with $\nu(U)$ computable.
  \item \label{levy:thm:3} $x$ is in $\mathsf{KR}^{\nu}(X)$ and $\lim_n \mathbb{E}_{\nu}[f\mid \mathscr{F}_n](x)$ exists for all $L_p(\nu)$ Schnorr tests~$f$.
\end{enumerate}

\end{restatable}

In condition (\ref{levy:thm:2}), $\lim_n \mathbb{E}_{\nu}[f\mid \mathscr{F}_n](x)=f(x)$ means that the limit of $\mathbb{E}_{\nu}[f\mid \mathscr{F}_n](x)$ exists and is finite and equal to $f(x)$. Likewise in (\ref{levy:thm:2.5})-(\ref{levy:thm:3}), the existence of the limit means that it is finite. Note that (\ref{levy:thm:2.5}) implies that $\mathsf{SR}^{\nu}$ already proves the analogue of density randomness where we restrict to c.e. open $U$ with $0<\nu(U)<1$ computable. 

For rates of convergence, we have:
\begin{restatable}{thm}{thmmainrates}\label{thm:newnewlevyrates} (Rates for Upward L\'evy Theorem for Schnorr Randomness). For all $X,\nu,\mathscr{F}_n$ as in Theorem~\ref{thm:newnewlevy}, one has:
\begin{enumerate}[leftmargin=*]
\item \label{levy:extra:5}  For all $x$ in $\mathsf{SR}^{\nu}(X)$ and all computable $p\geq 1$ and all $L_p(\nu)$ Schnorr tests $f$ one has that $x$ weakly computes a rate of convergence for $\mathbb{E}_{\nu}[f\mid \mathscr{F}_n](x)\rightarrow f(x)$.
\item \label{levy:extra:6}  For all $x$ in $\mathsf{SR}^{\nu}(X)$ of computably dominated degree and all computable $p\geq 1$ and all $L_p(\nu)$ Schnorr tests $f$ one has that there is a computable rate for the convergence $\mathbb{E}_{\nu}[f\mid \mathscr{F}_n](x)\rightarrow f(x)$.
\end{enumerate}
\end{restatable}

The notion in Theorem~\ref{thm:newnewlevyrates}(\ref{levy:extra:6}) is a classical notion from the theory of computation: a Turing degree is \emph{computably dominated} if any function from natural numbers to natural numbers that is computable from the degree is dominated by a computable function, in the sense that the computable function is eventually above it.\footnote{\cite[124]{soare2016turing}, \cite[27]{Nies2008aa}. A more traditional name for this concept is ``of hyperimmune-free degree.'' This more traditional name comes from an equivalent definition that emerged in the context of Post's Problem (cf. \cite[133 ff]{soare2016turing}).} For many but not all computable Polish spaces $X$ and $\nu$ in $\mathcal{P}(X)$ computable,  there are non-atoms in $\mathsf{MLR}^{\nu}$ (and hence in $\mathsf{SR}^{\nu}$) of computably dominated degree. This is a consequence of the existence of universal tests for $\mathsf{MLR}^{\nu}$ and the Computably Dominated Basis Theorem (cf. discussion at Proposition~\ref{prop:existencecompdom}, Example~\ref{ex:existencecompdom}, Question~\
\ref{q:existencecompdom}).

It is unknown to us whether Theorem~\ref{thm:newnewlevyrates}(\ref{levy:extra:6}) can be improved, in the sense of an affirmative answer to the following question:
\begin{restatable}{Q}{qrates}\label{q:rates}
For all~$x$ in $\mathsf{SR}^{\nu}$ that are not of computably dominated degree and all computable $p\geq 1$ and for all $L_p(\nu)$ Schnorr tests $f$  is there a computable rate for the convergence $\mathbb{E}_{\nu}[f\mid \mathscr{F}_n](x)\rightarrow f(x)$?

The following is the simplest concrete version of the question (cf. Example~\ref{ex:refinedpartitionconcrete}):

If $\nu$ is uniform measure on Cantor space, and $\omega$ in $\mathsf{SR}^{\nu}$ is not of computably dominated degree, and if $U$ is c.e. open with  $0<\nu(U)<1$ computable and $\omega$ not in $U$, then does the convergence $\nu(U\mid [\omega\upharpoonright n])\rightarrow 0$ have a computable rate?\footnote{For such $U$, the set $\{\omega: \nu(U\mid [\omega\upharpoonright n])\rightarrow I_U(\omega)\}$ can be rather complex. In particular, Carotenuto-Nies \cite{Carotenuto2016-rt} show that it is $\Pi^0_3$-complete when $U$ is dense. It is not clear to us whether this complexity is located among the $\mathsf{SR}^{\nu}$'s or the non-computably dominated  $\mathsf{SR}^{\nu}$'s, or whether it is reflected in their rates of convergence.}
\end{restatable}
\noindent Under uniform measure on Cantor space, the points which are not of computably dominated degree have measure one, a result due to Martin.\footnote{Martin's paper \cite{martin1967measure} is unpublished, but his proof has subsequently appeared in other sources, such as \cite[Theorem 8.21.1 p. 381]{Downey2010aa}, \cite[Theorem 1.2]{Dobrinen2004-oj}.} One way to negatively resolve the question would be to show that the non-computable-domination in Martin's proof (or a variation on it) could be witnessed by a rate of convergence associated to an $L_p(\nu)$ Schnorr test, or perhaps even to an indicator function of a c.e. open $U$ with $0<\nu(U)<1$ computable.

We prove Theorems~\ref{thm:newnewlevy}-\ref{thm:newnewlevyrates} in \S\ref{sec:proofmain}. Theorem~\ref{thm:newnewlevy} extends and unifies prior work by Pathak, Rojas, and Simpson, and of Rute (see discussion in \S\ref{sec:relation:work:previous} below), while Theorem-\ref{thm:newnewlevyrates} is entirely new.

Our next theorem pertains to convergence along Martin-L\"of tests: 
\begin{restatable}{thm}{thmmaindr}\label{thm:drlevy} (Effective~Upward L\'evy Theorem for Density Randomness~$p>1$).

Suppose that $X$ is a computable Polish space and $\nu$ in $\mathcal{P}(X)$ is computable. Suppose that $\mathscr{F}_n$ is an almost-full effective filtration, equipped uniformly with Kurtz disintegrations~$\rho^{(n)}$.

If $p> 1$ is computable, then the following three items are equivalent for $x$ in $X$: 
\begin{enumerate}[leftmargin=*]
  \item \label{dr:thm:1} $x$ is in $\mathsf{DR}_{\rho}^{\nu}(X)$.
  \item \label{dr:thm:2} $x$ is in $\mathsf{KR}^{\nu}(X)$ and $\lim_n \mathbb{E}_{\nu}[f\mid \mathscr{F}_n](x)=f(x)$ for all $L_p(\nu)$ Martin-L\"of tests~$f$.
  \item \label{dr:thm:3}  $x$ is in $\mathsf{KR}^{\nu}(X)$  and $\lim_n \mathbb{E}_{\nu}[f\mid \mathscr{F}_n](x)$ exists for all $L_p(\nu)$ Martin-L\"of tests~$f$ and $\lim_n \mathbb{E}_{\nu}[I_U\mid \mathscr{F}_n](x)=I_U(x)$ for every c.e. open $U$.
\end{enumerate}

\end{restatable}

In contrast to Theorem~\ref{thm:newnewlevyrates}(\ref{levy:extra:6}), one has the following, whose proof is a traditional diagonalization argument deploying the halting set:
\begin{restatable}{thm}{thmmaindrrates}\label{thm:drlevy:drates} 
(Rates for Upward L\'evy Theorem for Density Randomness).

There are $X,\nu, \mathscr{F}_n, \rho^{(n)}$ as in Theorem~\ref{thm:drlevy} which have the property that for every computable $p>1$ and every~$x$ in $\mathsf{DR}_{\rho}^{\nu}(X)$ there is $L_p(\nu)$ Martin-L\"of test $f$ such that the convergence $\mathbb{E}_{\nu}[f\mid \mathscr{F}_n](x)\rightarrow f(x)$ has no computable rate.
\end{restatable}
\noindent Hence, once we shift the tests from Schnorr tests to Martin-L\"of tests, we never have points which possess computable rates for all tests. In one sense, Question~\ref{q:rates} is asking whether there is some way to emulate a halting-set-like construction among the non-computably dominated $\mathsf{SR}^{\nu}$'s.

We prove Theorems~\ref{thm:drlevy}-\ref{thm:drlevy:drates} in \S\ref{sec:thm:dnesity:proof}. Theorem~\ref{thm:drlevy} extends work from a paper of Miyabe, Nies, and Zhang, which we discuss in the next section, while Theorem~\ref{thm:drlevy:drates} is entirely new.

Given Theorem~\ref{thm:newnewlevy} and Theorem~\ref{thm:drlevy}, it is natural to try to understand whether there is a convergence to the truth characterisation of $\mathsf{MLR}^{\nu}$, or at least some nearby superset of it (by contrast, $\mathsf{DR}^{\nu}_{\rho}$ is a subset of $\mathsf{MLR}^{\nu}$). We thus isolate a class of Martin-L\"of tests $f$ which have approximations $f_s$ such that $f_s\rightarrow f$ in $L_p(\nu)$ at an exponential rate, but not a rate that can necessarily be computed. Hence we define the following, where clauses (\ref{defn:maximaldoob:1})-(\ref{defn:maximaldoob:3}) mimic the canonical approximations of $L_p(\nu)$ Schnorr tests (cf. Proposition~\ref{prop:exptoflsc}, Lemma~\ref{ex:approximationbyschnorr2}), and where clause (\ref{defn:maximaldoob:4}) pertains to exponential rates:
\begin{defn}\label{defn:maximaldoob}
Suppose that $p\geq 1$ is computable. 

A \emph{$L_p(\nu)$ maximal Doob test} $f:X\rightarrow [0,\infty]$ is an lsc function in $L_p(\nu)$ such that there is a uniformly computable sequence $f_s$ of $L_p(\nu)$ Schnorr tests satisfying 
\begin{enumerate}[leftmargin=*]
\item \label{defn:maximaldoob:1} $0\leq f_s\leq f_{s+1}$ on $\mathsf{KR}^{\nu}$ and $f=\sup_s f_s$ on $\mathsf{KR}^{\nu}$.
\item \label{defn:maximaldoob:2} $f-f_s$ is equal on $\mathsf{KR}^{\nu}$ to a non-negative lsc function.
\item \label{defn:maximaldoob:3} $f_t-f_s$ for $t>s$ is equal on $\mathsf{KR}^{\nu}$ to an $L_p(\nu)$ Schnorr test, uniformly in $t>s$.
\item \label{defn:maximaldoob:4} For all $k\geq 0$, $\sum_s \|f-f_s\|_p\cdot (s+1)^k <\infty$.\footnote{By taking $k=0$, we have $f_s\rightarrow f$ in $L_p(\nu)$, and so $f$ is an $L_p(\nu)$ Martin-L\"of test, and hence in conjunction with (\ref{defn:maximaldoob:2}) we have that $f-f_s$ is equal on $\mathsf{KR}^{\nu}$ to an $L_p(\nu)$ Martin-L\"of test.}
\end{enumerate}

A point $x$ is \emph{$p$-maximal Doob random relative to $\nu$}, abbreviated $\mathsf{MDR}^{\nu,p}(X)$, if $f(x)<\infty$ for all $L_p(\nu)$ maximal Doob tests.
\end{defn}

Our theorem on this is the following:
\begin{restatable}{thm}{thmdoobmax}\label{thm:doobmax} (Effective~Upward L\'evy Theorem for Maximal Doob Randomness,~$p>1$).

Suppose that $X$ is a computable Polish space and $\nu$ in $\mathcal{P}(X)$ is computable. Suppose that $\mathscr{F}_n$ is an almost-full effective filtration, equipped uniformly with Kurtz disintegrations.

If $p> 1$ is computable, then the following three items are equivalent for $x$ in $X$: 
\begin{enumerate}[leftmargin=*]
    \item \label{thm:doobmax:1} $x$ in $\mathsf{MDR}^{\nu, p}$.
    \item \label{thm:doobmax:2} $x$ is in $\mathsf{KR}^{\nu}$ and $f(x)=\lim_n \mathbb{E}_{\nu}[f\mid \mathscr{F}_n](x)$ for all $L_p(\nu)$ maximal Doob tests $f$.
\item \label{thm:doobmax:3} $x$ is in $\mathsf{KR}^{\nu}$ and $\lim_n \mathbb{E}_{\nu}[f\mid \mathscr{F}_n](x)$ exists for all $L_p(\nu)$ maximal Doob tests $f$.
\end{enumerate}
\end{restatable}

We prove Theorem~\ref{thm:doobmax} in \S\ref{sec:mdr}. The name ``Maximal Doob'' in Theorem~\ref{thm:doobmax} and Definition~\ref{defn:maximaldoob} comes from the role played by Doob's Maximal Inequality (cf. Lemma~\ref{lem:doob}(\ref{lem:doob:1})) in the proofs in \S\ref{sec:mdr}. In Proposition~\ref{prop:doob:basecase}, we note that $\mathsf{MLR}^{\nu}\subseteq \mathsf{MDR}^{\nu,p}\subseteq \mathsf{SR}^{\nu}$. But we do not know the answer to the following question:
\begin{Q}\label{q:doob}
Are the inclusions $\mathsf{MLR}^{\nu}\subseteq \mathsf{MDR}^{\nu,p}\subseteq \mathsf{SR}^{\nu}$ proper? 
\end{Q}
\noindent We suspect that $\mathsf{MDR}^{\nu,p}$ is a proper subset of $\mathsf{SR}^{\nu}$, and that one could show this by establishing the analogue of Theorem~\ref{thm:drlevy:drates}.

We add that we do not know the answer to the following:
\begin{Q}
Do Theorem~\ref{thm:drlevy} and Theorem~\ref{thm:doobmax} hold for $p=1$? 
\end{Q}
\noindent The proof of the former uses H\"older at one place (cf. equation~(\ref{eqn:iamholder})), and the latter uses Doob's Maximal Inequality (cf. Lemma~\ref{lem:doob}(\ref{lem:doob:1})).

\subsection{Relation to previous work}\label{sec:relation:work:previous}

Theorem~\ref{thm:newnewlevy} generalises the result of Pathak, Rojas, and Simpson, who show it for the specific case of $p=1$ and $X=[0,1]^k$, $\nu$ being the $k$-fold product of Lebesgue measure on $[0,1]$ with itself, and with $\mathscr{F}_n$ being given by dyadic partitions.\footnote{\cite{Pathak2014aa}. They state their result not in terms of $L_p(\nu)$ Schnorr tests, but in terms of computable points of $L_p(\nu)$. See \S\ref{sec:app1:lpcomputable}.} Under this guise, L\'evy's Upward Theorem just is the Lebesgue Differentiation Theorem. Their proof goes through Tarski's decidability results on the first-order theory of the reals, and so seems in certain key steps specific to the reals with Lebesgue measure.\footnote{Such as at \cite[Lemma 3.3 p. 339]{Pathak2014aa}.} However, see Proposition~\ref{prop:schnorrtest} below, which generalises rather directly from their setting to the general setting.

In conjunction with the properties of effective disintegrations (cf. \S\ref{sec:ed-prop}), one can derive the equivalence of (\ref{levy:thm:1})-(\ref{levy:thm:2}) in Theorem~\ref{thm:newnewlevy} from results of Rute.\footnote{In particular, for the (\ref{levy:thm:1}) to (\ref{levy:thm:2}) direction of Theorem~\ref{thm:newnewlevy}, see Rute's ``Effective Levy 0/1 law'' \cite[Theorem 6.3 p. 31]{Rute2012aa}. Our Proposition~\ref{prop:compcont} and Proposition~\ref{prop:contcomp} implies that if $f$ is an $L_1(\nu)$ Schnorr test, then $\mathbb{E}_{\nu}[f\mid \mathscr{F}_n]$ is a computable point of $L_1(\nu)$, and so Rute's Theorem 6.3 applies, once one internalises how to translate back and forth between $L_p(\nu)$ Schnorr tests and computable points of $L_p(\nu)$ (cf. \S\ref{sec:app1:lpcomputable}). For the (\ref{levy:thm:2}) to (\ref{levy:thm:1}) direction of Theorem~\ref{thm:newnewlevy}, see Rute's \cite[Example 12.1 p. 31]{Rute2012aa}.} As for Schnorr randomness, our work then expands on Rute's primarily by finding a large class of versions of conditional expectations to which his results apply, and by identifying the information on rates of convergence in Theorem~\ref{thm:newnewlevyrates}. More generally, the theory we develop is organised around the elementary concept of an integral Schnorr test, and so we hope might be of value to others by virtue of being accessible.\footnote{In particular, we can avoid appeal to Rute's theory of a.e. convergence, which is an alternative way to organise effective convergence in $L_0(\nu)$ (cf. \S\ref{subsection:convergence:measure}). See \cite[Proposition 3.15 p. 15]{Rute2012aa} and his Convergence Lemma \cite[Lemma 3.19 p. 17]{Rute2012aa}.}

Further, we are able to strengthen what is, in our view, one of the more foundationally significant parts of Rute's work. He notes that traditionally ``algorithmic randomness is more concerned with success than convergence'' and that ``only computable randomness has a well-known characterisation in terms of martingale convergence instead of martingale success.''\footnote{\cite[p. 7]{Rute2012aa}. He is referring to what is called a ``folklore'' characterisation of computable randomness on Cantor space with the uniform measure in \cite[Theorem 7.1.3 p. 270]{Downey2010aa}.} In Cantor space with the uniform measure, Rute has a characterisation of Schnorr randomness in terms of convergence of $L_2(\nu)$ martingales.\footnote{See items (1), (4) in his Example 1.5, immediately below the preceding quotation.} We have been able to generalise this to all computable measures on computable Polish spaces: see Theorem~\ref{thm:margingaleconvergencel2}. This proof follows Rute's $L_2(\nu)$ Hilbert space proof in broad outline. It seems to us that keeping track of the maximal function, which we can then use in DCT arguments, has been helpful here. 

In the setting of Cantor space with the uniform measure and the natural filtration of length $n$-strings and the natural disintegration (cf. Example~\ref{ex:refinedpartitionconcrete}), our Theorem~\ref{thm:drlevy} was already known for $p=1$ and hence all computable $p\geq 1$. This result is in a paper of Miyabe, Nies, and Zhang, where it is attributed to the Madison group of Andrews, Cai, Diamondstone, Lempp and Miller.\footnote{\cite[Theorem 3.3 p. 312]{Miyabe2016-jo}. Their proof is a little more general, in that it just concerns martingale convergence rather than martingales associated to random variables. Their proof, in the Cantor space setting, also can be modified to give not only convergence but convergence to the truth for random variables.} Their argument goes through an auxiliary test notion of Madison test. While we only have it for computable $p>1$, our proof of Theorem~\ref{thm:drlevy} goes through first principles about density randomness and effective disintegrations. It is not presently clear to us whether the Cantor space proof using Madison tests can be generalised to arbitrary computable probability measures on computable Polish spaces equipped with effective disintegrations.\footnote{As a final remark about the previous literature, we should mention that L\'evy's Upward Theorem has also been studied in the context of Shafer and Vovk's game-theoretic probability (\cite{Shafer2009-ac}, \cite[Chapter 8]{Shafer2019-qp}). Their approach conceives of martingales primarily as game-theoretic strategies, and does not treat computational matters explicitly. By contrast, here we are focusing on the martingales $\mathbb{E}_{\nu}[f\mid \mathscr{F}_n]$ and on effective properties of them conceived of as sequences of random variables. Discerning the relation between our approach and their approach would involve, as a first step, carefully going through their approach and ascertaining the exact levels of effectivity needed to secure their results, and secondly translating back and forth between the strategy and random variable paradigms.}

We hope our efforts brings this prior important work on algorithmic randomness to the attention of a broader audience. Bayesianism is an important and increasingly dominant framework in a variety of disciplines, and this prior work and our work hopefully makes vivid the way in which computability theory and algorithmic randomness bears directly on the question of when and how fast Bayesian inductive methods converge to the truth. As its proof makes clear, the non-computable rate of convergence to the truth in Theorem~\ref{thm:drlevy:drates} is another presentation of the halting set, and so this proof gives the theory of computation a central role in a limitative theorem of inductive inference, similar to its central role in the great limitative theorems of deductive inference like the Incompleteness Theorems. Further, the existence of Schnorr random worlds that are computably dominated, as in Theorem~\ref{thm:newnewlevyrates}(\ref{levy:extra:6}), shows that a central kind of randomness is entirely compatible with there being effective ways of determining how close we are to the truth. This optimistic inductive possibility is not one that would be visible in absence of recent work in algorithmic randomness. 

Internal to the discussion about Bayesianism within philosophy, authors such as Belot have voiced the concern that the classical theory only tells us that worlds at which we fail to converge to the truth have probability zero, but otherwise tells us little about when and where the failure happens.\footnote{\cite{Belot2013aa}.} From the perspective of Theorems~\ref{thm:newnewlevy}, \ref{thm:drlevy}, \ref{thm:doobmax}, the probability zero event of non-convergence is not arbitrary, so long as one is insisting on convergence along a broad enough class of effective random variables. Namely, the sequences along which convergence to the truth fails for some element of this class are exactly those that are not random with respect to the underlying computable prior probability measure. In other words, those sequences can be determined by effective means to be atypical from the agent's point of view.

Finally, we should emphasise that ours is not the only perspective on conditional expectations and its effectivity that one could adopt. In focusing on disintegrations, we are presupposing a framework where pointwise there is a single ``formula'' for the conditional expectation, namely the one displayed in equation~(\ref{eqn:disintegrate}) (and again see Appendicies~\ref{sec:app:classical:disintegrations}-\ref{sec:app:effective:disintegrations} for examples). Likewise, the effectivity constraints in Definition~\ref{defn:eff:disin} have the consequence that the conditional expectation operator is a continuous computable function (cf. Proposition~\ref{prop:compcont}), and so sends computable points to computable points (cf. Proposition~\ref{prop:contcomp}). Both of these presuppositions constrain the applicability of our framework. For instance, Rao points out that conditional expectations are used throughout econometrics, but there one often uses the Dynkin-Doob Lemma as definitional of the conditional expectation,\footnote{\cite[376]{Rao2014-le}. For an example, see the presentation of conditional expectation in \cite[Chapter 7]{Florens2007-pn}.} and there is no more hope of having a single formula come out of it than there is of having all variables expressible in linear terms of one another. Likewise, conditional expectations and martingales can be used to prove theorems like the Radon-Nikodym Theorem,\footnote{\cite[p. 145-146]{Williams1991aa}.} which is ``computably false'' in that there are computable absolutely continuous probability measures with no computable Radon-Nikodym derivative.\footnote{\cite[p. 396]{Simpson2009aa}, \cite{Yu1990}, \cite{Hoyrup2011-rm}.} That, of course, is not to say that these are not of interest or that determining how non-effective they are is not of interest, but just to say they will not be available in a framework like ours where we restrict to computable continuous conditional expectation operators.\footnote{The paper Ackerman et. al. \cite{Ackerman2017-xf} is an important recent paper studying how non-effective, in general, it is to have disintegrations. Our Definition~\ref{defn:eff:disin}, by contrast, restricts attention to those disintegrations that are highly effective. This will not be all of them, and we do not claim that it would be all of the interesting ones.}

\subsection{Outline of paper}

The paper is organised as follows. In \S\ref{sec:large:cps}, we begin with a brief discussion of some aspects of effectively closed sets and computable continuous functions and lsc functions which we need for our proof, and then we go over relevant aspects of the three computable Polish spaces which are central for effective probability theory:
\begin{itemize}[leftmargin=*]
\item The computable Polish space  $\mathcal{M}^+(X)$  of non-negative finite Borel measures on $X$ and its computable Polish subspace $\mathcal{P}(X)$ of probability measures.
\item For each computable $\nu$ in $\mathcal{M}^+(X)$ and each computable $p\geq 1$, the computable Polish space $L_p(\nu)$.
\item For each computable $\nu$ in $\mathcal{M}^+(X)$ the computable Polish space $L_0(\nu)$ of Borel measurable functions which are finite $\nu$-a.s. and whose topology is given by convergence in measure. 
\end{itemize}
The space $L_0(\nu)$ is needed since when $f$ is in $L_p(\nu)$, the maximal function $f^{\ast}=\sup_n \mathbb{E}_{\nu}[f\mid \mathscr{F}_n]$ is in $L_0(\nu)$, and is guaranteed to be in $L_p(\nu)$ iff $p>1$. In addition to being needed in the proofs of the main theorems, the material in \S\ref{sec:large:cps} also can serve to contextualise many components of Definition~\ref{defn:core}. For instance, we mention in \S\ref{sec:efdst:measure} a result of Hoyrup-Rojas that an element of $\mathcal{P}(X)$ is computable in the sense of Definition~\ref{defn:core}(\ref{defn:core:2}) iff it is computable as an element of the Polish space $\mathcal{P}(X)$. Likewise, in \S\ref{sec:Lp} we mention a result saying that an $L_p(\nu)$ Schnorr test is simply a non-negative lsc function whose equivalence class is a computable element of $L_p(\nu)$ (cf. Proposition~\ref{prop:exptoflsc}). Finally, towards the close of \S\ref{subsection:convergence:measure}, we define an $L_0(\nu)$ Schnorr test and prove a new characterisation of $\mathsf{SR}^{\nu}$ in terms of these tests (cf. Definition~\ref{defn:l0st}, Proposition~\ref{prop:l0schnorrtest}).

In \S\ref{sec:twolemsonsr} we present two lemmas on Schnorr randomness. The second of these, called the Self-location Lemma (\ref{lem:sefllocate}) is a distinctive feature of Schnorr randomness (\emph{vis-\`a-vis} the other algorithmic randomness notions), and is central to our proof of Theorem~\ref{thm:newnewlevyrates}. In \S\ref{sec:recover} we present some results on recovering the pointwise values of effective random variables on $\mathsf{SR}^{\nu}$. In \S\ref{sec:maximal:classical}, we review various classical features of the maximal function which we shall need later. In \S\ref{sec:abstract} we present an abstract treatment of Theorem~\ref{thm:newnewlevy} in terms of various effective constraints that a version of the conditional expectation may satisfy. 
In \S\ref{sec:ed-prop} we develop the fundamental properties of effective disintegrations. In \S\ref{sec:proofmain}, we prove Theorems~\ref{thm:newnewlevy}-\ref{thm:newnewlevyrates}, and in \S\ref{sec:thm:dnesity:proof} we prove Theorems~\ref{thm:drlevy}-\ref{thm:drlevy:drates} and in \S\ref{sec:mdr} we prove Theorem~\ref{thm:doobmax}. In \S\ref{sec:app1:lpcomputable}, we show how Miyabe's translation method allows us to recast Theorem~\ref{thm:newnewlevy} in terms of computable points of $L_p(\nu)$. In \S\ref{sec:ed-martin}, we develop the theory of martingales in $L_2(\nu)$ and prove the aforementioned generalisation of Rute's result characterising Schnorr randomness in terms of martingale convergence. In Appendix~\ref{sec:app:classical:disintegrations} we briefly exposit two classical examples of disintegrations, and in Appendix~\ref{sec:app:effective:disintegrations} we present several examples of effective disintegrations.

In a sequel to this paper, we present a similar analysis of the Blackwell-Dubins Theorem,\footnote{\cite{Blackwell1962-ux}} which is also a ``convergence to the truth'' result, but wherein the pair ``agent and world'' is replaced with a pair of agents whose credences are variously absolutely continuous with respect to one another.

\section{Computable Polish spaces for effective probability theory}\label{sec:large:cps}

\subsection{Effectively closed, computable continuous, and lsc}\label{sec:closedcont}

In this section, we briefly describe two further concepts from the theory of computable Polish spaces: namely effectively closed subsets and computable continuous functions, and we close by mentioning a few brief aspects of lsc functions.

Before we do that, we mention one elementary proposition on computable Polish spaces which is worth having in hand (for e.g. the Self-location Lemma~\ref{lem:sefllocate}):
\begin{prop}\label{prop:closure:reals:general:0.4}
Let $X$ be a computable Polish space with metric $d$ and countable dense set~$x_0, x_1, \ldots$. Suppose the map $i\mapsto n(i)$ is such that $x_{n(i)}\rightarrow x$ fast. Then
\begin{enumerate}[leftmargin=*]
    \item \label{prop:closure:reals:general:0.4:shadow} The set $\{(j,q)\in \mathbb{N}\times \mathbb{Q}^{>0}: x\in B(x_j,q)\}$ is c.e. in graph of~$i\mapsto n(i)$.
    \item \label{prop:closure:reals:general:0.4:shadow:2}  The point $x$ is computable iff the set $\{(j,q)\in \mathbb{N}\times \mathbb{Q}^{>0}: x\in B(x_j,q)\}$ is c.e.
\end{enumerate}
\end{prop}
\begin{proof}
For (\ref{prop:closure:reals:general:0.4:shadow}), since the distances between the points of the countable dense set is uniformly computable, they are also uniformly right-c.e. Hence, it suffices to note that $d(x_j,x)<q$ iff there is $i\geq 0$ with $d(x_j,x_{n(i)})<q-2^{-i}$.

For (\ref{prop:closure:reals:general:0.4:shadow:2}), if $x$ is computable, then we can choose $i\mapsto n(i)$ computable, and then are done by (\ref{prop:closure:reals:general:0.4:shadow}). Conversely, if the set is c.e., given $i\geq 0$, enumerate it until one finds a pair $(j,q)$ with $q\leq 2^{-i}$, and set $n(i)=j$.
\end{proof}

As mentioned in \S\ref{sec:intro}, the complement of a c.e. open set is called an \emph{effectively closed set}. In Cantor space and Baire space, the effectively closed sets can be represented as paths through computable trees.\footnote{\cite[p. 41]{Cenzer1999-hy}.} The following provides a simple example on the real line of a classically closed set which is not effectively closed:
\begin{ex}\label{ex:theuniversalcounter}
Suppose that $c<d$ and $c$ is right-c.e. and $d$ is left-c.e but neither $c,d$ are computable. Then $[c,d]$ is a computable Polish space. Further, $[c,d]$ is a classically closed subset of the reals which is not an effectively closed subset of the reals.
\end{ex}
\noindent It is a computable Polish space since its countable dense set $(c,d)\cap \mathbb{Q}$ is c.e. since it is the intersection of the the right Dedekind cut of $c$ and the left Dedekind cut of $d$. If $[c,d]$ were effectively closed in the reals then $U=(-\infty,c)\cup (d,\infty)$ would be c.e. open in the reals. Then by choosing a rational $r$ in $(c,d)$, one has that $\{q\in \mathbb{Q}: q<c\}=\{q\in U: q<r\}$ is c.e., contrary to hypothesis.

Effectively closed subsets of a computable Polish space need not themselves have the structure of a computable Polish space, since one in addition needs to produce an enumeration of a countable dense set where the distance between the points is uniformly computable.\footnote{By contrast, classically, the Polish subspaces of a Polish space are precisely the $G_{\delta}$ subsets. See \cite[p. 17]{Kechris1995}.} Hence, we define: a \emph{computable Polish subspace} $Y$ of $X$ is given by a an effectively closed subset $Y$ of $X$ and a countable sequence of points $y_0, y_1, \ldots$ which are uniformly computable points of $X$ and which are dense in $Y$. One can check that the c.e. opens relative to $Y$ are just the c.e. opens of the space $X$ intersected with $Y$, and further any effectively closed subset of $Y$ is also an effectively closed subset of $X$.  Similarly, one can check that a computable point of $Y$ is just a computable point of $X$ which happens to be in $Y$. As a simple example of a Polish subspace which is not a computable Polish subspace, one has:
\begin{ex}\label{ex:theuniversalcounter2}
Suppose that $a<b$ and $a$ is left-c.e. and $b$ is right-c.e but neither $a,b$ are computable. Then the closed interval $[a,b]$ is an effectively closed subset of the reals which is not a computable Polish subspace of the reals. 
\end{ex}
\noindent It is effectively closed since $a$ being left-c.e. and $b$ being right-c.e. implies that $(-\infty,a)$ and $(b,\infty)$ are c.e. open. And if $[a,b]$ were a computable Polish subspace of the reals, and if $y_0, y_1, \ldots$ were a sequence of uniformly computable reals dense in $[a,b]$, then for a rational $q$ we would have $a<q$ iff there is $i$ such that $y_i<q$, which is a c.e. condition and so $a$ would be right-c.e. and thus computable.

An effectively closed set $C$ is \emph{computably compact} if there is a partial computable procedure which, when given an index for a computable sequence of c.e. opens $U_0, U_1, \ldots$ in $X$ which covers $C$, returns a natural number $n\geq 0$ such that $U_0, \ldots, U_n$ covers $C$. We further say that $C$ is \emph{strongly computably compact} if there is a partial computable procedure which, when given an index for a computable sequence of c.e. opens $U_0, U_1, \ldots$ in $X$ halts iff this is a cover of $C$, and when it halts returns  a natural number $n\geq 0$ such that $U_0, \ldots, U_n$ covers $C$. If $X$ itself is strongly computably compact, then so are all of its effectively closed sets. If $c<d$ is computable, then $[c,d]$ is strongly computably compact. Likewise, Cantor space is strongly computably compact, and if $f:\mathbb{N}\rightarrow \mathbb{N}$ is computable then the computably bounded set $\{\omega\in \mathbb{N}^{\mathbb{N}}: \forall \; n \; \omega(n)\leq f(n)\}$ is a computable Polish subspace of Baire space which is strongly computably compact. Example~\ref{ex:theuniversalcounter} is an example of a compact computable Polish space which is not computably compact, since if it were then we could compute the endpoints using maxs and mins of the centres of finite coverings with fast decreasing radii. For another example of a compact computable Polish space which is not computably compact, one can take the paths through a computable subtree of Baire space which is not computably bounded.\footnote{See \cite[Example 2.1.5 p. 59]{Cenzer2017-nm}.}

If $X,Y$ are two computable Polish spaces, then a function $f:X\rightarrow Y$ is \emph{computable continuous} if inverse images of c.e. opens are uniformly c.e. open.\footnote{See Moschovakis \cite[110]{Moschovakis2009-vf}. Simpson \cite[Exercise II.6.9 p. 88]{Simpson2009aa} notes that it is equivalent to his preferred definition at \cite[85]{Simpson2009aa}.} 
The following characterisation usefully parameterises each continuous computable function by a single c.e. set, where it is assumed for the sake of simplicity that both countable dense sets are identified with the natural numbers:\footnote{This is from \cite[1169]{Harrison-Trainor2020-mt}. It can be seen as a simplification of Simpson's definition in \cite[85]{Simpson2009aa}.}
\begin{itemize}[leftmargin=*]
\item A function $f:X\rightarrow Y$ is computable continuous iff there is a c.e. set $I\subseteq \mathbb{N}\times \mathbb{Q}^{>0}\times \mathbb{N}\times \mathbb{Q}^{>0}$ such that both (i)~if $(i,p,j,q)$ is in $I$ then $B(i,p)\subseteq f^{-1}(B(j,q))$ and (ii)~for all $x$ in $X$ and all $\epsilon>0$ there is $(i,p,j,q)$ in $I$ with $x$ in $B(i,p)$ and $q<\epsilon$.
\end{itemize}
Computable continuous maps are also computable continuous when restricted to computable Polish subspaces. The computable continuous maps preserve computability of points:
\begin{prop}\label{prop:contcomp}
If $f:X\rightarrow Y$ is computable continuous and $x$ in $X$ is computable, then $f(x)$ in $Y$ is computable.
\end{prop}
\begin{proof}
Suppose that $x$ is computable. Then $\{(i,q)\in \mathbb{N}\times \mathbb{Q}^{>0}: d(i,x)<q\}$ is c.e. by Proposition~\ref{prop:closure:reals:general:0.4}
(\ref{prop:closure:reals:general:0.4:shadow}). For each $n\geq 0$, by~(ii) above,  search in $I$ for a tuple $(i_n,p_n,j_n,q_n)$ with $x$ in $B(i_n,p_n)$ and $q_n<2^{-n}$. Then by (i) above, $f(x)$ is in $f(B(i_n, p_n))\subseteq B(j_n, q_n)$ and so $j_n\rightarrow f(x)$ fast.
\end{proof}
\noindent This proposition is important because many arguments for the computability of points in effective analysis and probability can be seen as the result of applying computable continuous functions to computable points.

There is a partial converse to the previous proposition in the uniformly continuous setting. A \emph{computable modulus of uniform continuity} for a uniformly continuous function $f:X\rightarrow Y$ is  a computable function $m:\mathbb{Q}^{>0}\rightarrow \mathbb{Q}^{>0}$ such that $d(x,x^{\prime})<m(\epsilon)$ implies $d(f(x), f(x^{\prime}))<\epsilon$ for all $\epsilon$ in $\mathbb{Q}^{>0}$. For instance, if $c>0$ is rational, then a $c$-Lipschitz function is just a function with linear modulus of uniform continuity $m(\epsilon)=\frac{c}{2} \cdot \epsilon$. The partial converse to Proposition~\ref{prop:contcomp} is the following:
\begin{prop}\label{prop:suffforcont}
Suppose $X,Y$ are computable Polish spaces and that $f:X\rightarrow Y$ has a computable modulus of uniform continuity. Suppose that the image of the countable dense set in $X$ under $f$ is uniformly computable in $Y$. Then $f$ is computable continuous.
\end{prop}
\begin{proof}
Suppose $x_n$ is the countable dense set in $X$. Suppose that $m:\mathbb{Q}^{>0}\rightarrow \mathbb{Q}^{>0}$ is the computable modulus of uniform continuity. Suppose that $y_{n,i}\rightarrow f(x_n)$ fast, where $y_{n,i}$ is a uniformly computable sequence from the countable dense set in $Y$. Then define the c.e. set $I=\{(x_n, m(2^{-i}), y_{n,i}, 2^{-i+1}): n,i\geq 0\}$. First we show that $B(x_n, m(2^{-i}))\subseteq f^{-1}(B(y_{n,i}, 2^{-i+1}))$. For, suppose that $x$ is in $B(x_n, m(2^{-i}))$. Then $d(x, x_n)<m(2^{-i})$. Then $d(f(x), f(x_n))<2^{-i}$. Further since $y_{n,i}\rightarrow f(x_n)$ fast, we have $d(f(x_n), y_{n,i})\leq 2^{-i}$, from which we obtain $d(f(x), y_{n,i})<2^{-i+1}$ by triangle inequality. Second suppose that $x$ is in $X$ and $\epsilon>0$. Let $i\geq 0$ be such that $2^{-i+1}<\epsilon$. Since $x_n$ is an enumeration of the countable dense set, there is $x_n$ such that $d(x,x_n)<m(2^{-i})$. Then $x$ is in $B(x_n, m(2^{-i}))$, and the tuple $(x_n, m(2^{-i}), y_{n,i}, 2^{-i+1})$ is in $I$ and $2^{-i+1}<\epsilon$.
\end{proof}
\noindent This proposition is widely applicable in our context since many operators in functional analysis are uniformly continuous, and since one often in practice has good control over what happens with the countable dense set (see the proofs of Proposition~\ref{prop:exptoflsc}
and Proposition~\ref{prop:pushforwards} for representative examples).

Finally, recall the notion of core notion of lsc from Definition~\ref{defn:core}(\ref{defn:core:1}), which is the effectivization of the classical notion of lower semi-continuous. This class of functions has some paradigmatic examples and useful closure conditions which we briefly enumerate without proof:
\begin{prop}\label{prop:lscclosure}

Constant functions that are left-c.e. reals are lsc. Indicator functions of c.e. opens are lsc.

Lsc functions are closed under addition, maxs and mins. Non-negative lsc functions are closed under multiplication.

Sups of uniformly lsc functions are lsc. Infinite sums of non-negative lsc functions are lsc. Compositions of lsc functions with computable continuous functions are lsc.
\end{prop}

Since $f$ is lsc iff $-f$ is usc, one can use this proposition to obtain examples and closure conditions for usc functions as well.

The analogue of Proposition~\ref{prop:contcomp} for lsc functions is that they send computable points to left-c.e. reals.

\subsection{The space of probability measures}\label{sec:efdst:measure}

If $X$ is a Polish space, then the space of real-valued finite signed Borel measures on $X$ is written as $\mathcal{M}(X)$. Recall that the weak$^{\ast}$-topology on $\mathcal{M}(X)$ is the smallest topology such that all the linear maps $\nu\mapsto \int_X f \; d\nu$ are continuous, where $f$ ranges over bounded continuous functions on the space.\footnote{Or, equivalently, as $f$ ranges over all bounded uniformly continuous functions on the space (\cite[110]{Kechris1995}).} Unless the space $X$ is finite, the weak$^{\ast}$-topology on $\mathcal{M}(X)$ is not metrizable.\footnote{\cite[17, 102]{Bogachev2018-dv}} However, when $X$ is a Polish space, the space $\mathcal{P}(X)$ of all probability Borel measures on $X$ with the weak$^{\ast}$-topology is a Polish space, as is the space $\mathcal{M}^+(X)$ of all finite non-negative Borel measures on $X$.\footnote{\cite[\S{17.E pp. 109 ff}]{Kechris1995}.} Convergence in $\mathcal{P}(X)$ is characterised by the Portmanteau Theorem.\footnote{\cite[Theorem 17.20 p. 111]{Kechris1995}, \cite[Theorem 2.1 p. 16]{Billingsley2013-sj}.}

A natural countable dense set on the spaces $\mathcal{P}(X)$ and $\mathcal{M}^+(X)$ are the finite averages of Dirac measures associated to points from the countable dense set on $X$, with rational values for the weights. These spaces can be completely metrized by the  metric of Prohorov. However, when working on $\mathcal{P}(X)$, it is often more useful to work with the Wasserstein metric, and further when the metric on $X$ is unbounded it is more convenient to work with the Kantorovich-Rubinshtein metric on $\mathcal{P}(X)$:\footnote{\cite[104,111]{Bogachev2018-dv}.}
\begin{equation*}
d_{KR}(\nu, \mu)   = \sup \{ \left|\mathbb{E}_{\nu} f-\mathbb{E}_{\mu} f\right| \; : f \mbox{ is $1$-Lipschitz} \; \& \; \|f\|_{\infty}\leq 1\} 
\end{equation*}
In this, $\|f\|_{\infty}=\sup_{x\in X} \left|f(x)\right|$. Hoyrup and Rojas prove that $\mathcal{P}(X)$ with the metrics of Prohorov or Wasserstein are computable Polish spaces, and their proof extends naturally to the Kantorovich-Rubinshtein metric. Likewise, their proof shows that $\mathcal{M}^+(X)$ is a computable Polish space and has $\mathcal{P}(X)$ as a computable Polish subspace.\footnote{\cite[49]{Hoyrup2008-oj}, \cite[838]{Hoyrup2009-pl}.} Further, Hoyrup and Rojas characterise the computable points in $\mathcal{P}(X)$ as follows, and their proof extends naturally to $\mathcal{M}^+(X)$:\footnote{\cite[52]{Hoyrup2008-oj}, \cite[839]{Hoyrup2009-pl}.}
\begin{prop}\label{prop:hoyruprojas}
A point $\nu$ in $\mathcal{M}^+(X)$ is computable iff $\nu(X)$ is computable and $\nu(U)$ is uniformly left-c.e. for c.e. opens $U$ in $X$.
\end{prop}
\noindent This proposition helps motivate Definition~\ref{defn:core}(\ref{defn:core:2}).

To illustrate the utility of this proposition, consider $[c,d]$ from Example~\ref{ex:theuniversalcounter}. This proposition implies that Lebesgue measure $m$ on $[c,d]$ is not a computable point of $\mathcal{M}^+(X)$ since $m([c,d])=d\mbox{-}c$ is left-c.e. but not computable. Likewise, $\frac{1}{d-c} m$ is not a computable point of $\mathcal{P}(X)$ since one can choose rationals $q,\epsilon$ such that $(q-\epsilon, q+\epsilon)\subseteq (c,d)$, and then one has $\frac{1}{d-c} m((q-\epsilon, q+\epsilon)) = \frac{2\epsilon}{d-c}$ is right-c.e. but not computable.

Recall the notion of a comptuable basis and measure computable basis from Definition~\ref{defn:core}(\ref{defn:core:3.1})-(\ref{defn:core:3}). Hoyrup and Rojas use an effective version of the Baire Category Theorem to prove every $\nu$ is a computable point of $\mathcal{M}^+(X)$ has a $\nu$-computable basis. Moreover, the basis can be taken to be open balls $B(i, r_j)$ with centres $i$ from the countable dense set and with radii given by a dense computable sequence $r_j$ of non-zero reals, with the closed balls $B[i,r_j]$ being the corresponding effectively closed supersets.\footnote{See \cite[Corollary 5.2.1 p. 844]{Hoyrup2009-pl}, \cite[Theorem 2.2.1.2 p. 60]{Hoyrup2008-oj}. Rute also employs this result of Hoyrup and Rojas, \cite[pp. 13-14]{Rute2012aa}, although he leaves out from the definition of a measure computable basis the pairing of each basis $U$ element with an effectively closed superset $C$ of the same measure. Hoyrup and Rojas include this pairing, but further require that $U\cup (X\setminus C)$ is dense (cf. \cite[Definition 5.1.2 p. 842]{Hoyrup2009-pl}, \cite[Definition 2.2.1.2 p. 58]{Hoyrup2008-oj}).}

The following proposition summarises some basic properties of $\nu$-computable bases:
\begin{prop}\label{prop:mucomputablebasis}\label{prop:union:mucompbasis}
Suppose that $\nu$ is a computable point of $\mathcal{M}^+(X)$.
\begin{enumerate}[leftmargin=*]
    \item \label{prop:mucomputablebasis:1} Elements of the algebra generated by a $\nu$-computable basis uniformly have $\nu$-computable measure. Indeed, this holds for all holds for all sequences $B_0, B_1, \ldots$ of events such that finite unions of them have uniformly $\nu$-computable measure.
    \item \label{prop:union:mucompbasis:2.2} If a computable sequence of c.e. opens with uniformly computable $\nu$-measure is added to a $\nu$-computable basis, then finite unions from the resulting sequence have uniformly computable $\nu$-measure, as do elements from the algebra generated by the resulting sequence. Indeed, this holds for all computable bases such that finite unions of them have uniformly $\nu$-computable measure.
    \item \label{prop:union:mucompbasis:2.3} If a computable sequence of c.e. opens with uniformly computable $\nu$-measure and with uniformly effectively closed supersets of the same $\nu$-measure is added to a $\nu$-computable basis, then the result is a $\nu$-computable basis.
    \item \label{prop:union:mucompbasis:2} The $\nu$-computable bases are closed under effective union.
\end{enumerate}
\end{prop}

\begin{proof}
For (\ref{prop:mucomputablebasis:1}), suppose that $B_0, B_1, \ldots$ is a sequence of events such that finite unions of them have uniformly $\nu$-computable measure.

First note that finite intersections $B_{i_0}\cap \cdots \cap B_{i_n}$ have uniformly computable $\nu$-measure: this is an induction on $n\geq 1$, and for the induction step, use inclusion-exclusion $\nu(B_{i_0}\cap \cdots \cap  B_{i_{n}}) =-\nu(B_{i_0}\cup \cdots \cup  B_{i_{n}})+\sum_{\emptyset \neq J\subsetneq \{0, \ldots, n\}} (-1)^{\left|J\right|-1} \nu(\bigcap_{j\in J} B_{i_j})$.

Likewise, finite unions $A_1 \cup \cdots \cup A_m$ of finite intersections $A_j=\bigcap_{k=1}^{\ell_j} B_{i_{j,k}}$ of members of $B_i$ have uniformly computable $\nu$-measure: this is an induction on $m\geq 1$, and for the induction step, use distribution $\nu(A_1\cup \cdots \cup A_{m+1})=\nu(A_1\cup \cdots \cup A_m)+\nu(A_{m+1})-\nu((A_1\cap A_{m+1})\cup \cdots \cup (A_m\cap A_{m+1}))$.

Finally, note that finite intersections of members of $B_i$ and complements of members of $B_i$ have uniformly computable $\nu$-measure. This follows from the previous steps, the elementary identity $\nu(C\setminus D)=\nu(C)-\nu(C\cap D)$ and distribution as follows when $n>0$: $\nu(B_{i_0}\cap \cdots \cap B_{i_{n-1}}\cap (X\setminus B_{i_n})\cap \cdots \cap (X\setminus B_{i_{n+m-1}})) = \nu(B_{i_0}\cap \cdots \cap B_{i_{n-1}})-\nu(B_{i_0}\cap \cdots \cap B_{i_{n-1}}\cap (B_{i_n}\cup \cdots \cup B_{i_{n+m-1}})) =\nu(B_{i_0}\cap \cdots \cap B_{i_{n-1}})-\nu(\bigcup_{j=n}^{n+m-1} B_{i_0}\cap \cdots \cap B_{i_{n-1}}\cap B_{i_j})$, which is uniformly computable by the two previous paragraphs. If $n=0$ then note that $\nu((X\setminus B_{i_n})\cap \cdots \cap (X\setminus B_{i_{n+m-1}}))=\nu(X\cap (X\setminus B_{i_n})\cap \cdots \cap (X\setminus B_{i_{n+m-1}}))$, and since $\nu(X)$ is computable, we can argue as in the case of $n=1$ with $X$ playing the role of $B_{i_0}$.

For (\ref{prop:union:mucompbasis:2.2}), suppose that $U_0, U_1, \ldots$ is a computable sequence of c.e. opens such that $\nu(U_0), \nu(U_1), \ldots$ is uniformly computable. Suppose that $B_0, B_1, \ldots$ is a computable basis such that finite unions of them have uniformly $\nu$-computable measure. We must show that finite unions from $B_0, B_1, \ldots, U_0, U_1, \ldots$ have uniformly $\nu$-computable measure. It suffices to consider the case where $U_0, U_1, \ldots$ just consists of a single c.e. open $U_i$, since by induction and (\ref{prop:mucomputablebasis:1}) we may assume that the $U_1, \ldots, U_{i-1}$ are already among the $B_0, B_1, \ldots$. Since $B_0, B_1, \ldots$ is a computable basis, write $U_i=\bigcup_j B_{m(j)}$, where $m$ is a computable function. Then $\nu(B_1\cup \cdots \cup B_n \cup U_i)=\nu(\bigcup_j B_1\cup \cdots \cup B_n \cup B_{m(j)}) =\lim_k \nu(\bigcup_{j<k} B_1\cup \cdots \cup B_n\cup B_{m(j)})$. Since this limit is increasing, and $\nu(\bigcup_{j<k} B_1\cup \cdots \cup B_n\cup B_{m(j)})$ is uniformly computable, we have that $\nu(B_1\cup \cdots \cup B_n \cup U_i)$ is left-c.e. Similarly, $\nu((B_1\cup \cdots \cup B_n) \cap U_i) = \nu(\bigcup_j (B_1\cup \cdots \cup B_n)\cap B_{m(j)})=\lim_k \nu(\bigcup_{j<k} (B_1\cup \cdots \cup B_n)\cap B_{m(j)})$. Since this limit is increasing, and $\nu(\bigcup_{j<k} (B_1\cup \cdots \cup B_n)\cap B_{m(j)})$ is uniformly computable by (\ref{prop:mucomputablebasis:1}), we have that $\nu((B_1\cup \cdots \cup B_n) \cap U_i)$ is left-c.e. Then $\nu(B_1\cup \cdots \cup B_n \cup U_i)=\nu(B_1\cup \cdots \cup B_n)+\nu(U_i)-\nu((B_1\cup \cdots \cup B_n) \cap U_i)$ is also right-c.e. and hence computable.

Finally, (\ref{prop:union:mucompbasis:2.3}) follows from (\ref{prop:union:mucompbasis:2.2}) and the definition of a $\nu$-computable basis; and (\ref{prop:union:mucompbasis:2}) follows directly from the uniformity in the proof of (\ref{prop:union:mucompbasis:2.3}).
\end{proof}

Many of the canonical computable bases are measure computable bases:

\begin{ex}\label{ex:cantorish}
If a computable basis on $X$ consists of sets which are also uniformly effectively closed, then the basis is $\nu$-computable for any computable point $\nu$ of $\mathcal{M}^+(X)$. This point applies to the canonical computable basis of clopens on Baire space or Cantor space.
\end{ex}

\begin{ex} If a computable basis on $X$ consists of c.e. open sets $U$ such that $\overline{U}$ is uniformly effectively closed with $\overline{U}\setminus U$ is finite, then the basis is $\nu$-computable for any computable atomless $\nu$ in $\mathcal{M}^+(X)$. This point applies to the canonical atomless measures on $[a,b]$ for $a<b$ computable. 
\end{ex}

Here is an example of a computable basis that is not a measure computable basis:
\begin{ex}
Let $f:\mathbb{N}\rightarrow \mathbb{N}\setminus \{0\}$ be an injective function whose range is c.e. but not computable, so that $b=\sum_i 2^{-f(i)}<1$ is left-c.e. but not computable. Let $q_i=1-2^{-(i+1)}$, which converges upwards to one, starting from $\frac{1}{2}$. Define a computable point $\nu$ of $\mathcal{P}([0,1])$ by $\nu = (\sum_i 2^{-f(i)}\cdot \delta_{q_i}) + (1-b)\cdot \delta_{1}$. A computable basis for $[0,1]$ is given by $(p,q)\cap [0,1]$ where $p<q$ are rationals. But this is not a $\nu$-computable basis since $\nu(0,1) = b$ is left-c.e. but not computable.
\end{ex}

To illustrate the utility of measure computable bases, consider the following approximation method. In this proof, we use the standard notation $W_e$ for the $e$-th c.e. set, and we use $W_{e,s}$ for the points in $W_e$ which get enumerated in by stage~$s$ in the canonical enumeration.\footnote{\cite[pp. 17-18, 47]{soare2016turing}.}
\begin{prop}\label{prop:hoyruprojas:compbasis:cor}
Suppose $\nu$ is a computable point of $\mathcal{M}^+(X)$. 

From a rational $\epsilon>0$ and an index for a c.e. open $U$ with $\nu(U)$ computable, one can uniformly compute an index for an effectively closed set $C\subseteq U$ and an index for $\nu(C)$ as a computable real such that $\nu(U\setminus C)<\epsilon$. 
\end{prop}
\begin{proof}
We work with the $\nu$-computable basis $B(i, r_j)$ as above (discussed immediately before Proposition~\ref{prop:mucomputablebasis}). Let $\epsilon>0$ rational be given. Suppose $U$ is c.e. open with $\nu(U)$ computable. Let $U=\bigcup_k B(i_{f(k)}, r_{f(k)}))$ where $f$ is a computable function. For each $m\geq 0$ let $U_m=\bigcup_{k<m} B(i_{f(k)}, r_{f(k)})$. Note that $U_m$ has $\nu$-computable measure, uniformly in $m\geq 0$. Using this and the computability of $\nu(U)$, compute $m\geq 0$ such that $\nu(U)-\nu(U_{m})<\frac{\epsilon}{2}$. For each $k<m$, the set $W_{g(k)}= \{j: 0<r_j<r_{f(k)}\}$ is c.e. and dense in the open interval $(0, r_{f(k)})$ and so $B(i_{f(k)}, r_{f(k)}) = \bigcup_{j\in W_{g(k)}} B(i_{f(k)}, r_j) = \bigcup_{j\in W_{g(k)}} B[i_{f(k)}, r_j]$. Compute $s\geq 0$ such that $\nu(U_{m})-\nu(\bigcup_{k<m} \bigcup_{j\in W_{g(k),s}} B(i_{f(k)}, r_j))<\frac{\epsilon}{2}$. Then $C= \bigcup_{k<m} \bigcup_{j\in W_{g(k),s}} B[i_{f(k)}, r_j]$ is a finite union of effectively closed sets and so effectively closed; and further $\nu(C)$ is a computable real since it is a finite union of elements from the $\nu$-computable basis. Further $C\subseteq U$ and $\nu(U)-\nu(C)\leq \nu(U)-\nu(U_{m})+\nu(U_{m})-\nu(C)<\epsilon$.
\end{proof}

The following is an important property of the interaction of $\mathsf{KR}^{\nu}$ with $\nu$-computable bases:
\begin{prop}\label{prop:krplusbasis}
Each element $A$ of the algebra generated by a $\nu$-computable basis is uniformly identical on $\mathsf{KR}^{\nu}$ to a c.e. open $U$, which is effectively paired with an effectively closed superset $C$ of $U$ of the same $\nu$-measure. 
\end{prop}
Note that since $C\setminus U$ is an effectively closed $\nu$-null set, $C=U$ on $\mathsf{KR}^{\nu}$.
\begin{proof}
Suppose that $B_0, B_1, \ldots$ is a $\nu$-computable basis with corresponding effectively closed set $C_i\supseteq B_i$ of the same $\nu$-measure. Again, since $C_i\setminus B_i$ is an effectively closed $\nu$-null set, we have that $C_i=B_i$ on  $\mathsf{KR}^{\nu}$. Then $X\setminus C_i$ is c.e. open with effectively closed superset $X\setminus B_i$ which with it agrees on $\mathsf{KR}^{\nu}$.

Suppose that $A$ is an element of the algebra generated by the $\nu$-computable basis $B_0, B_1, \ldots$. Then $A$ can be written as the finite union of finite intersections of the $B_0, B_1, \ldots$  and their relative complements  $X\setminus B_0, X\setminus B_1, \ldots$. This is indexed by a finite list of pairs of strings $\sigma_1, \tau_1, \ldots, \sigma_n, \tau_n$ such that
\begin{equation}\label{prop:krplusbasis:0}
A = \bigcup_{i=1}^n \bigg( \bigcap_{j<\left|\sigma_i\right|} B_{\sigma_i(j)}  \cap  \bigcap_{j<\left|\tau_i\right|} X\setminus B_{\tau_i(j)} \bigg)
\end{equation}
Then form c.e. open $V$ by replacing the effectively closed $X\setminus B_{\tau_i(j)}$ with the c.e. open $X\setminus C_{\tau_i(j)}$, and similarly form effectively closed $D$ by replacing c.e. open $B_{\sigma_i(j)}$ with effectively closed $C_{\sigma_i(j)}$, as follows:
\begin{equation}\label{prop:krplusbasis:1}
V = \bigcup_{i=1}^n \bigg( \bigcap_{j<\left|\sigma_i\right|} B_{\sigma_i(j)}  \cap  \bigcap_{j<\left|\tau_i\right|} X\setminus C_{\tau_i(j)}\bigg), \hspace{3mm} D = \bigcup_{i=1}^n \bigg(\bigcap_{j<\left|\sigma_i\right|} C_{\sigma_i(j)}  \cap  \bigcap_{j<\left|\tau_i\right|} X\setminus B_{\tau_i(j)} \bigg)
\end{equation}
Then $A,V,D$ are equal on $\mathsf{KR}^{\nu}$ and hence have the same $\nu$-measure, and further $V$ is c.e. open and $D\supseteq V$ is effectively closed. 

\end{proof}

The previous proposition places topological constraints on the sets in $\nu$-computable bases, at least when the measure has full support (that is, there are no open $\nu$-null sets):
\begin{prop}\label{prop:closureissues}
Suppose that $\nu$ is a computable point of $\mathcal{P}(X)$.
\begin{enumerate}[leftmargin=*]
    \item \label{prop:closureissues:1} If $\nu$ has full support and $\mathsf{XR}^{\nu}$ is a $\nu$-measure one set and the c.e. open $U$ is equal to effectively closed $C$ on $\mathsf{XR}^{\nu}$, then $\nu(\overline{U})=\nu(U)$. 
    \item \label{prop:closureissues:2} If $\nu$ has full support then no element $U$ of a $\nu$-computable basis can satisfy $\nu(\overline{U})>\nu(U)$.
\end{enumerate}
\end{prop} 
In this, we use $\overline{\;\cdot\;}$ for topological closure.
\begin{proof}
For (\ref{prop:closureissues:1}), the c.e. open $U\setminus C$ is a subset of the $\nu$-null $X\setminus \mathsf{XR}^{\nu}$. Since $\nu$ has full support, we must have that $U\setminus C$ is empty, so that $U\subseteq C$ and $\overline{U}\subseteq C$. Since $U,C$ have same $\nu$-measure, the same must then be true of $U,\overline{U}$. For (\ref{prop:closureissues:2}), this follows from (\ref{prop:closureissues:1}) and the previous proposition.
\end{proof}

By contrast, Proposition~\ref{prop:mucomputablebasis}(\ref{prop:union:mucompbasis:2.3}) implies any c.e. open $U$ with $\nu(U)$ computable and $\overline{U}$ effectively closed and $\nu(\overline{U})=\nu(U)$ can be added to any $\nu$-computable basis to form a larger $\nu$-computable basis. 

For a simple example of c.e. open as in Proposition~\ref{prop:closureissues}(\ref{prop:closureissues:2}), one has the following:\footnote{This example is a minor modification of an example from a proof in \cite[p. 58]{Cenzer1999-hy}.}
\begin{ex}\label{ex:ceopencenzer}
Consider Cantor space with the uniform measure. Let $0=c_0<c_1<c_2<\cdots$ be a computable sequence of natural numbers such that $\sum_n 2^{-(c_{n+1}-c_n)}<\infty$ (resp. is computable). Let $I$ be any computable set. For all $n\geq 0$, consider the following clopen:
\begin{equation*}
U_n = \{\omega: \forall \; i\in [c_n, c_{n+1}) \; \big(\omega(i)=1 \leftrightarrow  (n,i)\in I\big)\}
\end{equation*}
Since $U_n$ makes decisions on $c_{n+1}-c_n$ many bits, its measure is $2^{-(c_{n+1}-c_n)}$. Then $U=\bigcup_n U_n$ is a c.e. open. And $0<\nu(U)=\sum_n \nu(U_n\setminus \bigcup_{m<n} U_m) \leq \sum_n \nu(U_n)<\infty$ (resp. is computable by the Comparison Test and the fact that $U_n\setminus \bigcup_{m<n} U_m$ is clopen, cf. Example~\ref{ex:cantorish} and Proposition~\ref{prop:mucomputablebasis}(
\ref{prop:mucomputablebasis:1})). Further, the set $U$ is dense and so its closure is the entire space. 
\end{ex}

For a similar example on the unit interval with Lebesgue measure, one can use the complements of positive measure Cantor sets.


\subsection{The space of integrable functions}\label{sec:Lp}

For $\nu$ a computable point of $\mathcal{M}^+(X)$ and $p\geq 1$ computable, there is a natural Polish space structure on $L_p(\nu)$ (cf. Definition~\ref{defn:versions}). For, one can take as the countable dense set the simple functions $\sum_{i=1}^n q_i \cdot I_{A_i}$, where $A_i$ come from the algebra of sets generated by a $\nu$-computable basis. If $f,g$ are two such functions, then so is $f-g$, and hence it suffices to show that if $h=\sum_{i=1}^n q_i \cdot I_{A_i}$ is such a simple function then $\|h\|_p$ is computable. Since $A_i$ comes from an algebra, we can assume that the $A_i$ are pairwise disjoint, which implies $\left|\sum_{i=1}^n q_i \cdot I_{A_i}\right|^p=\sum_{i=1}^n \left| q_i\right|^p \cdot I_{A_i}$ everywhere. Then one has $\|h\|_p = \big(\sum_{i=1}^n \left|q_i\right|^p \nu(A_i)\big)^{\frac{1}{p}}$, which is computable by Proposition~\ref{prop:mucomputablebasis}(\ref{prop:mucomputablebasis:1}). Note that the countable dense set is in $\mathbb{L}_p(\nu)$, that is, it is defined everywhere rather than merely $\nu$-a.s. (cf. Definition~\ref{defn:versions}). But when we pass to their equivalence classes, they become elements of $L_p(\nu)$, and they are a countable dense set in $L_p(\nu)$.

We do not record the choice of the $\nu$-computable basis in the notation for the computable Polish space $L_p(\nu)$. This is for two reasons. First, the $\nu$-computable bases are closed under effective union Proposition~\ref{prop:mucomputablebasis}(\ref{prop:union:mucompbasis:2}). Hence one can typically just assume that one is working with the union of whichever of them are salient in a given context. Second, one can check that any two $\nu$-computable bases result in computably homeomorphic presentations of $L_p(\nu)$.

Many of the natural continuous functions on the computable Polish space $L_p(\nu)$ are computable continuous, such as: addition, subtraction, multiplication by computable scalar, absolute value, maximum, minimum, positive part, and negative part. 

By considering the continuous computable function $\Phi(f)=\left|f\right|-f$, one sees that $L_p^+(\nu)=\Phi^{-1}(\{0\})$, and so $L_p^+(\nu)$ is an effectively closed subset of $L_p(\nu)$ (cf. Definition~\ref{defn:versions}). Further, it is a computable Polish subspace, since the equivalence classes of the non-negative elements of the countable dense set of $L_p(\nu)$ are dense in $L_p^+(\nu)$.

Since we are working with a finite computable measure $\nu$ from $\mathcal{M}^+(X)$, if $p\leq q$, then the identity map is a computable continuous map from $L_q(\nu)$ into $L_p(\nu)$ and satisfies $\|f\|_p\leq \|f\|_q$ for all $f$ from $L_q(\nu)$. We refer to this as \emph{the computable embedding of $L_q(\nu)$ into $L_p(\nu)$}.

In working with $L_p(\nu)$ for $p>1$ it is useful to remember the following inequalities:
\begin{equation}
u,v\geq 0: \hspace{5mm} u^p+v^p \leq (u+v)^p \hspace{3mm} \mbox{ \; \& \; } \hspace{3mm} (u+v)^{\frac{1}{p}}\leq u^{\frac{1}{p}}+v^{\frac{1}{p}}
\end{equation}
By letting $u=x-y$ and $v=y$, one obtains the following inequalities:
\begin{equation}\label{eqn:estimate3}
0\leq y\leq x: \hspace{5mm} (x-y)^p\leq x^p-y^p \hspace{3mm} \mbox{ \; \& \; } \hspace{3mm}  x^{\frac{1}{p}}- y^{\frac{1}{p}} \leq (x-y)^{\frac{1}{p}}
\end{equation}


The following proposition gives a canonical approximation of lsc functions which are bounded from below, and indicates that for the non-negative ones, being a computable point of $L_p(\nu)$ is solely a matter of the computability of the norm. The first part is due to Miyabe for $p=1$.\footnote{\cite[Lemma 4.6]{Miyabe2013-fd}.} This kind of approximation is a mainstay of working with lsc functions, and different approximations tend to be appropriate for different purposes.\footnote{See \cite[Definition 1.7.4 p. 35]{Li1997aa}.} We will need a variation on this approximation in Proposition \ref{prop:dis:4}.
\begin{prop}\label{prop:exptoflsc}
From a rational $q$ and a lsc function $f:X\rightarrow [q,\infty]$, one can compute an index for a computable sequence of functions $f_s:X\rightarrow [q,\infty)$ from the countable dense set of $L_1(\nu)$ such that $f_s\leq f_{s+1}$ everywhere and $f=\sup_s f_s$ everywhere.

Further, if $p\geq 1$ is computable, then a non-negative lsc function $f:X\rightarrow [0,\infty]$ in $L_p(\nu)$ is a $L_p(\nu)$ Schnorr test (cf. Definition~\ref{defn:core}(\ref{defn:core:6})) iff it is a computable point of $L_p(\nu)$, and in this case the witness is a computable subsequence of $f_s$.

Finally, if $p\geq 1$ is computable, then any non-negative lsc function $f:X\rightarrow [0,\infty]$ in $L_p(\nu)$ is a $L_p(\nu)$ Martin-L\"of test (cf. Definition~\ref{defn:core}(\ref{defn:core:7})), and $f_s\rightarrow f$ in $L_p(\nu)$.
\end{prop}

\begin{proof}
Let $B_0, B_1, \ldots$ be a $\nu$-computable basis. Enumerate $\mathbb{Q}\cap [q,\infty)$ as $q_0, q_1, \ldots$. For each $n\geq 0$, one has that $f^{-1}(q_n,\infty]$ is uniformly c.e. open. Hence, there is a computable function $g$ such that $f^{-1}(q_n, \infty) = \bigcup_{i\in W_{g(n)}} B_i$. Then define
\begin{equation}\label{prop:exptoflsc:fsexpress}
f_s(x) = \max\{q, q_n: n\leq s, i\in W_{g(n),s}, x\in B_i\}
\end{equation}
This is an element of the countable dense set of $L_1(\nu)$ since we just enumerate $\bigcup_{n\leq s} W_{g(n),s}$ as $i_0, \ldots, i_{k(s)}$ and for each non-empty subset $K$ of $\{i_0, \ldots, i_{k(s)}\}$ we consider the element $B_K=\bigcap_{i_j\in K} B_i \cap \bigcap_{i_j\notin K} X\setminus B_i$ of the algebra generated by the $\nu$-computable basis, and we let $q_K = \max\{q_n: n\leq s, i_j\in K, i_j\in W_{g(n), s}\}$, so that we have $f_s = \sum_{\emptyset \neq K\subseteq \{i_0, \ldots, i_{k(s)}\}} q_K\cdot I_{B_K}$. Further, at the initial stages $s$ (if any) where $\bigcup_{n\leq s} W_{g(n),s}$ is empty, we set $f_s=q\cdot I_X$.

Further from (\ref{prop:exptoflsc:fsexpress}) one sees that $f_s\leq f_{s+1}$ since the sum over which we taking the maximum grows in $s$. Further, one has $f_s\leq f$ everywhere since if we had $f_s(x)>f(x)$, then $f_s(x)=q_n$ for some $n\leq s$ with $i\in W_{g(n),s}$ and $x$ in $B_i$. But then $B_i\subseteq f^{-1}(q_n,\infty]$, and so $f(x)>q_n$. Finally, one has $\sup_s f_s=f$ everywhere, since if not we would have $\sup_s f_s(x) < q_n < f(x)$ for some $x$ and some $n$ and hence $x$ would be in $f^{-1}(q_n,\infty]= \bigcup_{i\in W_{g(n)}} B_i$ and so $x$ would be in $B_i$ for some $i$ in $W_{g(n)}$ and hence there would be $s$ such that $i$ is in $W_{g(n),s}$ and hence by definition in (\ref{prop:exptoflsc:fsexpress}) one would have that $f_s(x)\geq q_n$.

Suppose $p\geq 1$ is computable and $f:X\rightarrow [0,\infty]$ is lsc and in $L_p(\nu)$. If $f$ is a computable point of $L_p(\nu)$, then since the norm is computable continuous (using Proposition~\ref{prop:suffforcont}), we have that $\|f\|_p$ is computable (using Proposition~\ref{prop:contcomp}). Conversely, suppose that $f$ is an $L_p(\nu)$ Schnorr test, so that $\|f\|_p$ is computable. Then by taking $p$-th roots, we have $\int f^p \; d\nu$ is computable. Since $\int f_s^p \; d\nu$ converges upwards to $\int f^p \; d\nu$ and we can compute both, we can compute a $s(n)$ such that $\int f^p -f_{s(n)}^p \; d\nu<2^{-np}$ for all $n\geq 0$. Then using the estimate $(f-f_{s(n)})^p \leq f^p -f_{s(n)}^p$ from (\ref{eqn:estimate3}), we have that $\int (f-f_{s(n)})^p \; d\nu <2^{-np}$, and so by taking $p$-th roots again we have $\|f-f_{s(n)}\|_p<2^{-n}$.

Similarly, for the last point, since $\int f_s^p \; d\nu$ converges upwards to $\int f^p \; d\nu$, we can use the estimate from (\ref{eqn:estimate3}) to argue that for all $\epsilon>0$ there is $s_0\geq 0$ such that for all $s\geq s_0$ one has $\int (f-f_s)^p \; d\nu\leq \int f^p \; d\nu-\int f_s^p \; d\nu<\epsilon^p$, and so $\|f-f_s\|_p<\epsilon$.
\end{proof}

The following records the ``universal test'' for $\mathsf{MLR}^{\nu}$. For integral tests, it is due G\'acs and Hoyrup-Rojas in the case $p=1$.\footnote{G\'acs \cite[102, Corollary 3.3]{Gacs2005-xu}, Hoyrup-Rojas \cite[845-6]{Hoyrup2009-pl}. Further, the version stated here is simplified in that it is only stated for a single measure, whereas these authors state a version where the lsc functions have domain $\mathcal{P}(X)\times X$. Hoyrup-Rojas improve on G\'acs by removing any assumption about the computability of the Boolean algebra structure on the algebra generated by the canonical computable basis.} 
\begin{prop}\label{prop:universal}
Suppose $\nu$ is a computable point of $\mathcal{P}(X)$ and $p\geq 1$ is computable. 

Then there is an $L_p(\nu)$ Martin-L\"of test $f$ with $\|f\|_p\leq 1$ such that for all $L_p(\nu)$ Martin-L\"of tests $g$ with $\|g\|_p \leq 1$ there is constant $c>0$ such that $g\leq c\cdot f$ everywhere. 

Hence $\mathsf{MLR}^{\nu} = \bigcup_n f^{-1}[0,n]$, an increasing sequence of effectively closed sets.
\end{prop}
\begin{proof} (Sketch)
Enumerate the $L_p(\nu)$ Martin-L\"of tests with $p$-norm $\leq 1$ as $h_0, h_1, \ldots$. Do this by enumerating approximations to them (as in Proposition~\ref{prop:exptoflsc}) which have $p$-norm $<1$. Then set $f=\sum_e 2^{-e} \cdot h_e$.
\end{proof}

The previous proposition has the following useful consequence regarding computable domination, which recall features in Theorem~\ref{thm:newnewlevy}(\ref{levy:extra:6}):
\begin{prop}\label{prop:existencecompdom}
Suppose that $X$ is computably compact and $\nu$ is a computable point of $\mathcal{P}(X)$. Then there are points in $\mathsf{MLR}^{\nu}$ of computably dominated degree.
\end{prop}
The main idea of the proof is to build a computably compact space of fast Cauchy sequences above $X$, and to apply there the Computably Dominated Basis Theorem.\footnote{\cite[Theorem 3.7 p. 54]{Cenzer1999-hy}, \cite[Theorem 9.5.1 p. 179]{soare2016turing}.} One can of course thematize the space of fast Cauchy sequences more than we are doing in this short paper, and in part what we are doing in the below proof is doing the construction out ``by hand'' in the computably compact case.
\begin{proof}
Without loss of generality, we identify the countable dense set with the natural numbers. By effective Baire Category Theorem, choose a strictly decreasing computable sequence of positive reals $\eta_s<2^{-(s+1)}$ such that $\{\eta_s :s\geq 0\}\cap \{\frac{1}{2}\cdot d(i,j): i,j\geq 0\}$ are disjoint. We define a non-decreasing computable sequence $n_s$ of natural numbers as follows. Suppose that we have already defined things up to stage~$s$. To define at stage~$s$, we consider the open cover $B(0, \eta_s), B(1, \eta_s), \ldots$ and use computable compactness to compute an $n_s\geq n_{s-1}$ such that $B(0, \eta_s), \ldots, B(n_s, \eta_s)$ covers $X$. Define the following computable trees:
\begin{align*}
T_0 &  = \{ \sigma \in \mathbb{N}^{<\mathbb{N}}:  \forall \; t<\left|\sigma\right| \; \exists \; i\leq n_t \; \sigma(t)=i\} \\
T & = \{\sigma\in T_0: \forall \; t<\left|\sigma\right| \; \forall \; r\in [t,\left|\sigma\right|) \; d(\sigma(t), \sigma(r))\leq 2\eta_t\}
\end{align*}
The tree $T$ is computable since $\{2\cdot \eta_s :s\geq 0\}\cap \{d(i,j): i,j\geq 0\}$ are disjoint. Further $T$ has no dead ends since we can just extend by repeating the last entry (since $n_s\geq n_{s-1}$). Let $C=[T]$, which is then a computable Polish space with countable dense set given by extending any node $\sigma$ in $T$ by means of repeating its last entry indefinitely. Since the function $t\mapsto n_t$ is computable, one has that $C$ is strongly computably compact.

The map $\pi:C\rightarrow X$ given by sending $\omega$ to $\lim_i \omega(i)$ in $X$ is well-defined. For, since $\eta_s<2^{-(s+1)}$, every $\omega$ in $C$ is a Cauchy sequence.

By definition of $T$, note that $d(\pi(\omega), \omega(t))\leq 2\eta_t$ for all $t\geq 0$. For let $\epsilon>0$. Since $d(\pi(\omega), \omega(r))\rightarrow 0$, choose $r>t$ such that $d(\pi(\omega), \omega(r))<\epsilon$. Then $d(\pi(\omega), \omega(t))\leq d(\pi(\omega), \omega(r))+d(\omega(t), \omega(r))\leq \epsilon+2\eta_t$. 

Note that any $\omega$ in $C$ is a sequence from the countable dense set of $X$ which converges fast to $\pi(\omega)$. This is because $2\eta_t<2^{-t}$.

Further, $\pi:C\rightarrow X$ is surjective: if $x$ in $X$ is given, then for each $j$ choose $\omega(j) \leq n_j$ such that $x$ is in $B(\omega(j),\eta_j)$. Then $\omega$ is in $C$ since for all $k\geq j$ one has $d(\omega(j), \omega(k)) \leq d(\omega(j), x)+d(x,\omega(k))\leq \eta_j+\eta_k\leq 2\eta_j$. 

Then by Proposition~\ref{prop:suffforcont}, the map $\pi:C\rightarrow X$ is computable continuous since it has a computable modulus of uniform continuity. For, if rational $\epsilon>0$ is given, compute least $\ell\geq 0$ such that $4\cdot \eta_{\ell}<\epsilon$. Suppose that  $\omega, \omega^{\prime}$ are in $C$ with $\omega, \omega^{\prime}$ agreeing $\leq \ell$. Then $d(\pi(\omega), \pi(\omega^{\prime}))\leq d(\pi(\omega), \omega(\ell))+d(\omega(\ell), \omega^{\prime}(\ell))+d(\omega^{\prime}(\ell), \pi(\omega^{\prime}))\leq 2\cdot \eta_{\ell}+0+2\cdot \eta_{\ell}<\epsilon$.

By the previous proposition, choose a non-empty effectively closed subset $D$ of $X$ which consists only of $\mathsf{MLR}^{\nu}$'s. Then $\pi^{-1}(D)\subseteq C$ is an effectively closed subset of $C$, which is thus strongly computably compact since $C$ is. By the Computably Dominated Basis Theorem, there is an element $\omega$ of $\pi^{-1}(D)$ of computably dominated degree.
\end{proof}

The following example shows that one cannot in general assume that the $\mathsf{MLR}^{\nu}$'s of computably dominated degree in the previous proposition are non-atoms.
\begin{ex}\label{ex:existencecompdom}
There is an uncountable computably compact computable Polish space $X$ and a computable point $\nu$ of $\mathcal{P}(X)$ such that the only elements of $\mathsf{MLR}^{\nu}(X)$ of computably dominated degree are among the atoms. 

This follows from a construction of Ng et. al.\footnote{\cite[Lemma 2.1, Theorem 2.2]{Ng2012-pc}.} Let $\mu$ be the uniform measure on Cantor space $Y=\{0,1\}^{\mathbb{N}}$ and let $Z=\{0,1,2\}^{\mathbb{N}}$.  Ng et. al. constructs a computable continuous map $f:Y\rightarrow Z$ with image $X$ such that every $\omega$ in $\mathsf{MLR}^{\mu, \emptyset^{\prime}}(Y)$ is such that $f(\omega)$ is non-isolated in $X$, and vice-versa, and in this circumstance $f(\omega)$ and $\omega$ have the same Turing degree.

The image $X$ is a computable Polish space.\footnote{Since it is the computable continuous image of Cantor space, cf. \cite[Theorem 2.4.8(3) pp. 73-74]{Cenzer2017-nm}.} Further, pushforwards of computable probability measures under computable continuous maps are computable probability measures (by Proposition~\ref{prop:hoyruprojas}), and so $\nu:=f\#\mu$ is a computable point of $\mathcal{P}(X)$. Note that $\nu$ has full support since $\mu$ has full support. 

Suppose that $\omega^{\prime}$ in $X$ is in $\mathsf{MLR}^{\nu}(X)$ and is of computably dominated degree. Then we claim that $\omega^{\prime}$ is an atom. For reductio, suppose not. Since any isolated point in a space with full support is an atom, one has that $\omega^{\prime}$ is not isolated. Since $f:Y\rightarrow X$ is a surjection, choose $\omega$ in $Y$ with $f(\omega)=\omega^{\prime}$. By the construction, $\omega$ is in $\mathsf{MLR}^{\mu, \emptyset^{\prime}}(Y)$. But these points are not of computably dominated degree.\footnote{E.g. \cite[Theorem 8.21.2 p. 382]{Downey2010aa}.} Since $\omega, \omega^{\prime}$ have the same Turing degree, $\omega^{\prime}$ is not of computably dominated degree, contrary to hypothesis.
\end{ex}

It is not clear to us what happens in the general atomless non-compact case:
\begin{Q}\label{q:existencecompdom}
Suppose that $X$ is a computable Polish space which is not computably compact, and that $\nu$ in $\mathcal{P}(X)$ is computable and atomless. Is there an element in $\mathsf{MLR}^{\nu}(X)$ that is of computably dominated degree? 
\end{Q}
\noindent If $X$ is the reals, it can be written as an effective union of computably compact Polish subspaces, and so the answer is affirmative, by Proposition~\ref{prop:existencecompdom}. If $X$ is Baire space, then the answer is again affirmative, by using effective tightness to describe the $\mathsf{MLR}^{\nu}$'s as a subset of a countable union of computably compact sets, and then applying the Computably Dominated Basis Theorem again. Hence to answer the question negatively one should be looking for spaces which are not ``effectively~$K_{\sigma}$'' and spaces where effective tightness does not produce a union of computably compact sets containing the $\mathsf{MLR}^{\nu}$'s.

If $\nu$ is in $\mathcal{P}(X)$ and $f:X\rightarrow [-\infty,\infty]$ is in $L_1(\nu)$, then it induces the push-forward probability measure $(f\#\nu)(A)=\nu(f^{-1}(A))$ in $\mathcal{P}(\mathbb{R})$. The following proposition tells us that the map $f\mapsto f\#\nu$ is computable continuous. We use this proposition primarily in conjunction with Proposition~\ref{prop:contcomp} and Proposition~\ref{prop:hoyruprojas}, which together tell us that pushforwards of $L_1(\nu)$-computable functions are themselves computable.
\begin{prop}\label{prop:pushforwards}
Let $X$ be a computable Polish space. Suppose that $\nu$ is a computable point of $\mathcal{P}(X)$. Then the map from $L_1(\nu)$ to $\mathcal{P}(\mathbb{R})$ given by sending $f$ to $f\# \nu$ is continuous computable. Similarly, the map from $L_1^+(\nu)$ to $\mathcal{P}(\mathbb{R}^{\geq 0})$ given by sending $f$ to $f\# \nu$ is continuous computable.
\end{prop}
\begin{proof}
We apply Proposition~\ref{prop:suffforcont}. 

Suppose that $f$ is an element of the countable dense set of $L_1(\nu)$.  Then $f=\sum_{i=1}^{m} q_{i}\cdot I_{A_{i}}$, where $q_i$ is rational and the $A_{i}$ are elements of the algebra generated by a $\nu$-computable basis. Then uniformly in rationals $p<q$ one has that $\nu(f^{-1}(p,q)) = \nu(\cup \{ A_i: 1\leq i\leq m, q_i\in (p,q)\})$, which is left-c.e. and indeed computable. Hence $f\# \nu$ is a computable point of $\mathcal{P}(\mathbb{R}^{\geq 0})$ by Proposition~\ref{prop:hoyruprojas}.

Any computable function $m:\mathbb{Q}^{>0}\rightarrow \mathbb{Q}^{>0}$ satisfying $m(\epsilon)<\epsilon$ is a computable modulus of uniform continuity. To see this, suppose that $\epsilon>0$, and suppose that $h:\mathbb{R}\rightarrow \mathbb{R}$ is 1-Lipschitz, and that $\mathbb{E}_{\nu} \left|f-g\right| <m(\epsilon)$. By change of variables, one has that $\left| \mathbb{E}_{f\# \nu} h-\mathbb{E}_{g\# \nu} h\right|=\left| \mathbb{E}_{\nu} (h\circ f)-\mathbb{E}_{\nu} (h\circ g)\right|\leq \mathbb{E}_{\nu} \left| h\circ f - h\circ g\right| \leq \mathbb{E}_{\nu} \left|f-g\right| <m(\epsilon)$, where the second-to-last inequality uses that $h$ is 1-Lipschitz. By taking the supremum over all $1$-Lipschitz $h:\mathbb{R}\rightarrow \mathbb{R}$ with $\|h\|_{\infty}\leq 1$, one has $d_{KR}(f\# \nu, g\# \nu)\leq m(\epsilon)$, which by construction is $<\epsilon$. 

Since $\{ f\in L_1(\nu): f\geq 0\}$ is a computable Polish subspace of $L_1(\nu)$, the restriction of $f\mapsto f\#\nu$ to it is also computable continuous.
\end{proof}

The above proposition has the following extremely useful consequence:\footnote{Outside of density, the statement of this lemma is contained in Miyabe's proof of his characterisation of $\mathsf{SR}^{\nu}$ in terms of $L_p(\nu)$ Schnorr tests. See e.g. the line ``It follows that $\mu(\{x: t(x)>r_n\})$ is computable uniformly in $n$'' (\cite[p. 6]{Miyabe2013}). Miyabe does not use pushforwards, but rather does it out by hand for $L_p(\nu)$ Schnorr tests.}
\begin{lem}\label{lem:usingst}
Let $X$ be a computable Polish space. Suppose that $\nu$ is a computable point of $\mathcal{P}(X)$.

Suppose $f:X\rightarrow [0,\infty]$ is lsc with $f<\infty$ $\nu$-a.s. Suppose that $f\#\nu$ is a computable point of $\mathcal{P}(\mathbb{R}^{\geq 0})$. Then there is a computable sequence of reals $r_i>0$ dense in $[0,\infty)$ such that $f^{-1}(r_i,\infty]$ is c.e. open with uniformly $\nu$-computable measure.

In particular, this is true of any $L_p(\nu)$ Schnorr test.
\end{lem}
\begin{proof}
Let $\mu:=f\#\nu$, which by hypothesis is a computable point of $\mathcal{P}(\mathbb{R}^{\geq 0})$. By the Hoyrup-Rojas result discussed in \S\ref{sec:efdst:measure}, there is a $\mu$-computable basis of the form $(q-r_i, q+r_i)\cap \mathbb{R}^{\geq 0}$, where $q$ ranges over rationals and $r_i>0$ is a computable sequence dense in $[0,\infty)$, and where further $(q-r_i, q+r_i)\cap \mathbb{R}^{\geq 0}$ has the same $\mu$-measure as $[q-r_i, q+r_i]\cap \mathbb{R}^{\geq 0}$. Since $f<\infty$ $\nu$-a.s., we have $\nu(f^{-1}(r_i,\infty])=\nu(f^{-1}(r_i,\infty))=(f\#\nu)(r_i,\infty)=\mu(r_i,\infty)= 1-\mu[0,r_i] = 1-\mu([-r_i, r_i]\cap \mathbb{R}^{\geq 0})$, which is computable.

The last point follows from the previous proposition.
\end{proof}

The following proposition is elementary but useful. (Recall usc was defined in Definition~\ref{defn:core}(\ref{defn:core:1a})).
\begin{prop}\label{prop:shiftbasistolpst}
For any element $f$ of the countable dense set of $L_p^+(\nu)$, one can compute an index for a non-negative lsc function $g$ and a non-negative usc function $h$ such that $f=g=h$ on $\mathsf{KR}^{\nu}$.

Likewise, for any element $f$ of the countable dense set of $L_p(\nu)$, one can compute a rational $q$ and an index for a non-negative lsc function $g$ and a non-negative usc function $h$ such that $f-q=g=h$ on $\mathsf{KR}^{\nu}$.
\end{prop}
\begin{proof}
Let $f$ be an element of the countable dense set of $L_p^+(\nu)$. Then $f=\sum_{i=1}^k q_i \cdot I_{A_i}$, where $q_i\geq 0$ is rational and $A_i$ is an element of the algebra generated by a $\nu$-computable basis. By Proposition~\ref{prop:krplusbasis}, suppose that $U_i$ is a c.e. open which is uniformly equal on $\mathsf{KR}^{\nu}$ to $A_i$, and suppose that $C_i$ is an effectively closed superset of $U_i$ of the same $\nu$-measure. Then $g:=\sum_{i=1}^k q_i \cdot I_{U_i}$ is non-negative lsc, and $h:=\sum_{i=1}^k q_i \cdot I_{C_i}$ is non-negative usc, and they agree with $f$ on $\mathsf{KR}^{\nu}$.

Let $f$ be an element of the countable dense set of $L_p(\nu)$. Then $f=\sum_{i=1}^k q_i \cdot I_{A_i}$, where $q_i$ is rational and $A_i$ is an element of the algebra generated by a $\nu$-computable basis. Let $q=\min_i q_i$. Then $f-q$ is an element of the countable dense set of $L_p^+(\nu)$. 
\end{proof}

\subsection{The space of measurable functions}\label{subsection:convergence:measure}

The space of equivalence classes of Borel measurable functions that are finite $\nu$-a.s. under $\nu$-a.s. identity is denoted by $L_0(X,\nu)$, where $\nu$ is in $\mathcal{M}^+(X)$. We write $L_0(\nu)$ when $X$ is clear from context. 

In keeping with the notational conventions in \S\ref{defn:versions}, we write $\mathbb{L}_0(\nu)$ for the pointwise-defined Borel measurable functions that are finite $\nu$-a.s.

The topology on $L_0(\nu)$ is given by convergence in measure. To enhance readability, if $h$ is a measurable function, then we write $\nu(\left|h\right|>\epsilon)$ for the more cumbersome $\nu(\{x\in X: \left|h\right|(x)> \epsilon\})$. Then recall $f_n\rightarrow f$ in measure iff for all $\epsilon>0$ one has that $\lim_n \nu(\left|f_n-f\right|> \epsilon)=0$. Recall that a consequence of Egoroff's Theorem is that $f_n\rightarrow f$ $\nu$-a.s. implies $f_n\rightarrow f$ in $L_0(\nu)$ for $\nu$ in $\mathcal{M}^+(X)$.\footnote{\cite[p. 62]{Folland1999aa}.}

A compatible complete metric is given by $d(f,g)=\|f-g\|_0$ where $\|h\|_0 =\inf\{\epsilon>0 : \nu(\left|h\right|>\epsilon)<\epsilon\}$. Note that the set $\{\epsilon>0 : \nu(\left|h\right|>\epsilon)<\epsilon\}$ is upwards closed, so that $\|h\|_0=\sup\{\epsilon>0 : \nu(\left|h\right|>\epsilon)>\epsilon\}$. When $\nu$ in $\mathcal{P}(X)$, this is called the \emph{Ky Fan} metric.\footnote{\cite[289]{Dudley2002-ej}} While $\|\cdot\|_0$ satisfies the triangle inequality $\|f+g\|_0\leq \|f\|_0+\|g\|_0$ and satisfies $\|f\|_0=0$ iff $f=0$ $\nu$-a.s., it does not in general satisfy $\|c\cdot h\|_0=\left|c\right|\cdot \|h\|_0$.\footnote{\cite[65-69]{Doob1994-ki}, \cite[289-290]{Dudley2002-ej}.}{$^{,}$}\footnote{More generally, $L_0(\nu)$ is not a Banach space.}  In working with the metric, it is useful to note that $\|h\|_0\leq \epsilon$ iff $\nu(\left|h\right|>\epsilon)\leq \epsilon$. Finally, note that $\left|f\right|\leq \left|g\right|$ $\nu$-a.s. implies $\|f\|_0\leq \|g\|_0$ in $L_0(\nu)$.

The natural countable dense set for $L_0(\nu)$ is the the rational-valued simple functions formed from the algebra generated by a $\nu$-computable basis, that is, the same countable dense set as we used for $L_p(\nu)$ for $p\geq 1$ computable. Classically, this set is dense in $L_0(\nu)$, so it remains to verify that the distance between these two points is uniformly computable:
\begin{prop}
If $h$ is a rational-valued simple functions formed from the algebra generated by a $\nu$-computable basis, then $\|h\|_0$ is computable, and uniformly~so. If $f,g$ are two such simple functions, then $\|f-g\|_0$ is computable, and uniformly~so.
\end{prop}
\begin{proof}
If $h$ is one of these functions, then so too is $\left|h\right|$. Suppose that $\left|h\right|=\sum_{i=1}^n q_i \cdot I_{A_i}$, where $q_i\geq 0$ is rational and $A_i$ is are pairwise disjoint events from the algebra generated by a $\nu$-computable basis. 

For $\epsilon$ in $\mathbb{Q}^{>0}$, let $J_{\epsilon} = \{i\in [1,n]: q_i>\epsilon\}$, which is a finite set whose index is computable uniformly from $\epsilon>0$. Then $\nu(\left|h\right|>\epsilon) = \sum_{i\in J_{\epsilon}} \nu(A_i)$, which is a computable real, uniformly in $\epsilon>0$ (by Proposition \ref{prop:mucomputablebasis}(\ref{prop:mucomputablebasis:1}).

Then $\epsilon$ in $\mathbb{Q}^{>0}$ satisfies $\nu(\left|h\right|>\epsilon) <\epsilon$ iff $\sum_{i\in J_{\epsilon}} \nu(A_i)<\epsilon$, which is a c.e. condition. If we enumerate these rational $\epsilon$ and take mins as we go, we get a computable decreasing sequence of rationals which converges down to $\|h\|_0$, so that $\|h\|_0$ is right-c.e. 

Likewise, $\delta$ in $\mathbb{Q}^{>0}$ satisfies $\nu(\left|h\right|>\delta) >\delta$ iff $\sum_{i\in J_{\delta}} \nu(A_i)>\delta$, which is a c.e. condition. If we enumerate these rational $\delta$ and take maxes as we go, we get a increasing computable sequence of rationals which converges up to $\|h\|_0$, so that $\|h\|_0$ is left-c.e. 

Similarly if $f, g$ are from a countable dense set then $\|f-g\|_0$ is a computable real since $f-g$ is also an element of the countable dense set. 
\end{proof}

We call the following \emph{the computable embedding of $L_p(\nu)$ into $L_0(\nu)$}. The square root in the rate of convergence is, in our view, explanatory of the many $\sqrt{2}$'s that appear in Pathak et. al. when dealing computable points of $L_1(\nu)$.\footnote{See e.g. \cite[p. 343]{Pathak2014aa}.}
\begin{prop}\label{prop:bfastprop}
Suppose that $p\geq 1$ is computable. Then the inclusion map is a uniformly continuous computable map from $L_p(\nu)$ to $L_0(\nu)$. Further, if $f_n\rightarrow f$ at geometric rate $b>1$ of convergence in $L_p(\nu)$, then $f_n\rightarrow f$ at geometric rate $\sqrt{b}$ in $L_0(\nu)$.
\end{prop}
\begin{proof}
Suppose that $p\geq 1$. Since $L_p(\nu)$ and $L_0(\nu)$ have the same countable dense set, by Proposition~\ref{prop:suffforcont}, it suffices to show that there is a computable modulus $m:\mathbb{Q}^{>0}\rightarrow \mathbb{Q}^{>0}$ of uniform continuity. Given rational $\epsilon>0$, compute rational $\delta<\epsilon$ and compute a rational $m(\epsilon)<\delta^{1+\frac{1}{p}}$. Suppose $f,g$ are in $L_p(\nu)$ with $\|f-g\|_p<m(\epsilon)$. Then $\nu(\left|f-g\right|>\delta)\leq \frac{1}{\delta^p} \|f-g\|_p^p \leq \frac{1}{\delta^p} m(\epsilon)^p <\delta$, and so $\|f-g\|_0\leq \delta <\epsilon$. 

Suppose that $p\geq 1$ and suppose $f_n\rightarrow f$ at geometric rate $b>1$ of convergence in $L_p(\nu)$. Then $\|f-f_n\|_1\leq \|f-f_n \|_p$, and so $f_n\rightarrow f$ at geometric rate $b>1$ in $L_1(\nu)$. Let $n\geq 0$. Then $\nu(\left|f-f_n\right|> (\sqrt{b})^{-n}) \leq (\sqrt{b})^n \cdot \|f-f_n\|_1 \leq (\sqrt{b})^{-n}$. Then $\|f-f_n\|_0 \leq (\sqrt{b})^{-n}$.
\end{proof}
\noindent

The following proposition is the natural effectivization of the Bounded Convergence Theorem:\footnote{\cite[130]{Williams1991aa}.} 
\begin{prop}\label{prop:bddthm} (Effective Bounded Convergence Theorem). Suppose $\nu$ is a computable point of $\mathcal{P}(X)$. Then:
\begin{enumerate}[leftmargin=*]
    \item \label{prop:bddthm:2} Suppose that $f_n\rightarrow f$  in $L_0(\nu)$ at a geometric rate of convergence $b\geq \sqrt{2}$. Suppose that $K\geq 0$ such that $\left|f_n\right|\leq K$ $\nu$-a.s. for all $n\geq 0$. Then $f_{n+2}\rightarrow f$ at a geometric rate of convergence $b$ in $L_1(\nu)$.
    \item \label{prop:bddthm:1} Suppose $f$ is a computable point of $L_0(\nu)$ and $K\geq 0$ is a rational such $\left|f\right|\leq K$ $\nu$-a.s. Then $f$ is a computable point of $L_1(\nu)$.
    \item \label{prop:bddthm:3} Suppose $f_n$ is a uniformly computable point of $L_0(\nu)$ and $K_n\geq 0$ is a uniformly computable sequence of rationals such that $\left|f_n\right|\leq K_n$ $\nu$-a.s. Then $f_n$ is uniformly a computable point of $L_1(\nu)$.    
\end{enumerate}
\end{prop}
\begin{proof}
For (\ref{prop:bddthm:2}), classically some subsequence of $f_n$ converges $\nu$-a.s. to $f$. Hence, $\left|f\right|\leq K$ $\nu$ a.s. Further, we may suppose $K>1$. Suppose that $f_n\rightarrow f$ at a geometric rate of $b\geq \sqrt{2}$ convergence in $L_0(\nu)$, so that $\nu(\left|f_n-f\right|>b^{-n})\leq b^{-n}$ for all $n\geq 0$. Choose $n_0\geq 2$ sufficiently large so that $b^{-n_0}<\frac{1}{2K}$. Let $c=b^{n_0}$, so that for all $n\geq 2$ we have $2K\cdot c^{-n} \leq 2K\cdot b^{-n_0}\cdot b^{-n}<b^{-n}$. Then $\|f_{n+2}-f\|_1\leq \int_{\left|f_{n+2}-f\right|>c^{-(n+2)}} \left|f_{n+2}-f\right|\; d\nu+\int_{\left|f_{n+2}-f\right|\leq c^{-(n+2)}} \left|f_{n+2}-f\right|\; d\nu \leq 2K\cdot c^{-(n+2)} +c^{-(n+2)} \leq 2b^{-(n+2)} \leq b^{-n}$, where the last inequality follows from $b\geq \sqrt{2}$.

For (\ref{prop:bddthm:1}), suppose $f_n\rightarrow f$ fast in $L_0(\nu)$. For $0<\epsilon<1$ one has $\nu(\left|f_n\cdot I_{\left|f_n\right|\leq K+1}-f\right|>\epsilon)\leq \nu(\left|f_n-f\right|>\epsilon)$, and hence we may assume that $\left|f_n\right|\leq K+1$. Then we apply  (\ref{prop:bddthm:2}).

For (\ref{prop:bddthm:3}), this is just the uniformisation of (\ref{prop:bddthm:1}).
\end{proof}

Using Proposition~\ref{prop:suffforcont}, one has that many of the usual operations on $L_0(\nu)$ are computable continuous, such as addition and minimum and maximum. Indeed, each of these three has modulus $m(\epsilon)=\frac{\epsilon}{2}$. We can use this observation to prove the following. It had been previously established by Rute, although our proof is different.\footnote{\cite[Proposition 3.26]{Rute2012aa}.}

\begin{prop}\label{prop:pushforwards:2}
If $f$ is a computable point of $L_0(\nu)$ (resp. $L_0^+(\nu)$), then $f\# \nu$ is a computable point of  $\mathcal{P}(\mathbb{R})$ (resp. $\mathcal{P}(\mathbb{R}^{\geq 0})$).
\end{prop}
\begin{proof}
Let $f_n =\min(\max(f,-n), n)$, so that $f_n=f$ on $f^{-1}(-n, n)$, and $f_n$ is uniformly a computable point of $L_0(\nu)$ (resp. $L_0^+(\nu)$). By Proposition \ref{prop:bddthm}
(\ref{prop:bddthm:3}) one has that $f_n$ is uniformly a computable point of $L_1(\nu)$ (resp. $L_1^+(\nu)$). By Proposition~\ref{prop:pushforwards}, one has that $f_n\#\nu$ is uniformly a computable point of $\mathcal{P}(\mathbb{R})$ (resp. $\mathcal{P}(\mathbb{R}^{\geq 0})$). For rational $p<q$ (resp. rational $0\leq p<q$), compute natural number $n>\max(\left|p\right|,\left|q\right|)$, so that by Proposition~\ref{prop:hoyruprojas} the real $(f\#\nu)(p,q) = (f_n\#\nu)(p,q)$ is uniformly left-c.e. Then by Proposition~\ref{prop:hoyruprojas}, the probability measure $f\#\nu$ is a computable point of $\mathcal{P}(\mathbb{R})$ (resp. $\mathcal{P}(\mathbb{R}^{\geq 0})$).
\end{proof}

We then define:
\begin{defn}\label{defn:l0st}
An \emph{$L_0(\nu)$ Schnorr test} is a lsc function $f:X\rightarrow [0,\infty]$ which is a computable point of $L_0(\nu)$.
\end{defn}

In parallel to Definition~\ref{defn:core}(\ref{defn:core:9}), we have the following new characteriation of $\mathsf{SR}^{\nu}$:
\begin{prop}\label{prop:l0schnorrtest}
A point $x$ is in $\mathsf{SR}^{\nu}$ iff $f(x)<\infty$ for all $L_0(\nu)$ Schnorr tests~$f$.
\end{prop}
\begin{proof}
If $f(x)<\infty$ for all $L_0(\nu)$ Schnorr tests $f$, then by the computable embedding of $L_1(\nu)$ into $L_0(\nu)$, we have $f(x)<\infty$ for all $L_1(\nu)$ Schnorr tests $f$, and so $x$ is in $\mathsf{SR}^{\nu}$. Conversely, suppose that $x$ is in $\mathsf{SR}^{\nu}$. Let $f$ be an $L_0(\nu)$ Schnorr test. Since $f$ is in $L_0(\nu)$, it is finite $\nu$-a.s. By Lemma~\ref{lem:usingst} and the previous proposition, there is a computable sequence of reals $\eta_n$ in the open interval $(2^{n}, 2^{n+1})$ such that $U_n:=f^{-1}(\eta_n, \infty]$ is c.e. open with uniformly $\nu$-computable measure. So we have $\nu(U_n)\rightarrow 0$, and since $\nu(U_n)$ is computable, we can compute a subsequence with $\nu(U_{n_i})<2^{-i}$. Hence $f=\sum_i I_{U_{n_i}}$ is an $L_1(\nu)$ Schnorr test, and so $f(x)<\infty$ and so $x$ is only finitely many of the $U_{n_i}$, and hence there is $i$ such that $f(x)\leq \eta_{n_i}$.
\end{proof}
Miyabe has shown that $x$ being in every $\Sigma^0_2$ $\nu$-measure one class is equivalent to $f(x)<\infty$ for every non-negative lsc $f$ in $L_0(\nu)$.\footnote{\cite[Proposition 3.3]{Miyabe2016-jo}. This notion of randomness is also called weak 2-randomness.} In conjunction with the above proposition, it suggests that there is little room for a simple characterisation of $\mathsf{MLR}^{\nu}$ in terms of non-negative lsc functions in $L_0(\nu)$.

Finally, we update our previous approximation theorem for $L_p(\nu)$ Schnorr tests to $L_0(\nu)$ Schnorr tests:
\begin{prop}\label{prop:exptoflscL0}
For all $L_0(\nu)$ Schnorr tests $f:X\rightarrow [0,\infty]$, one can compute a subsequence of $f_{s(n)}$ the $f_s$ from Proposition~\ref{prop:exptoflsc} such that $f_{s(n)}\rightarrow f$ fast in $L_0(\nu)$.
\end{prop}
\begin{proof}
The $f_s$ come from the countable dense set of $L_0(\nu)$ and hence are computable points of $L_0(\nu)$. Since $f_s\rightarrow f$ everywhere, we have $f_s\rightarrow f$ in measure, and so $f_s\rightarrow f$ in $L_0(\nu)$. Since $\|f_s-f\|_0$ is computable, we just search for a subsequence $f_{s(n)}$ with $\|f_{s(n)}-f\|_0<2^{-n}$.
\end{proof}

\section{Two Schnorr lemmas: Flipping an approximation and Self-location}\label{sec:twolemsonsr}

In this section, we provide two lemmas on Schnorr tests, one involving turning an approximation from below into a non-increasing subsequence converging down to zero, and another based upon a distinctive self-location property of Schnorr randoms.

The first of these involves a partial subtraction operator, which involves some care since it helps one avoid situations with $\infty-\infty$. These situations can potentially arise since lsc functions are allowed to take infinite values.
\begin{prop}\label{prop:ominusproperties}
Suppose that $p\geq 1$ computable or $p=0$.

Suppose that $f:X\rightarrow (-\infty,\infty]$ is an lsc function in $L_p(\nu)$ (resp. an $L_p(\nu)$ Schnorr test). Suppose that $\mathsf{XR}^{\nu}$ is a $\nu$-measure one set on which the function $f$ is finite. 

Suppose that $g:X\rightarrow (-\infty, \infty]$ is an lsc function such that $g\leq f$ on $\mathsf{XR}^{\nu}$. Suppose that $g$ is paired with a usc function $\breve{g}: X\rightarrow [-\infty, \infty)$ such that $g,\breve{g}$ are equal on $\mathsf{XR}^{\nu}$.

Define $f\ominus g = \max(0, f-\breve{g})$. Then 
\begin{itemize}[leftmargin=*]
\item $g, \breve{g}$ are finite on $\mathsf{XR}^{\nu}$.
\item $f\ominus g$ is non-negative lsc and in $L_p(\nu)$ (resp. an $L_p(\nu)$ Schnorr test) and is equal on  $\mathsf{XR}^{\nu}$ to $f-g$.
\end{itemize}
\end{prop}
\begin{proof}
For the first item, since the lsc function $g$ has codomain $(-\infty,\infty]$ and the usc function $\breve{g}$ has codomain $[-\infty, \infty)$, then when the two agree they have finite value. And they agree on $\mathsf{XR}^{\nu}$.

Since $\breve{g}$ is usc, $-\breve{g}$ is lsc. Since the lsc functions are closed under addition (cf. Proposition~\ref{prop:lscclosure}), one has that $f-\breve{g}$ is lsc. Since the lsc functions are preserved under max (cf. again Proposition~\ref{prop:lscclosure}), we have that $f\ominus g$ is non-negative lsc. Further, since $f,g$ are in $L_p(\nu)$ (resp. are $L_p(\nu)$-computable) and this property is preserved under subtraction and maxes, we have that $f\ominus g$ is also in $L_p(\nu)$ (resp.  $L_p(\nu)$-computable). On $\mathsf{XR}^{\nu}$, one has that $f-g$ is both equal to $f-\breve{g}$ and is non-negative, and hence equal to $f\ominus g$.
\end{proof}

While partial subtraction operation $f\ominus g$ is \emph{not} defined absolutely, but only relative to the hypotheses of the previous proposition, the situation of the following lemma is the one which tends to be operative in applications. We call it ``flipping an approximation'' since it takes a non-decreasing approximation $f_s\rightarrow f$ and turns it into a non-increasing approximation $f\ominus f_n\rightarrow 0$. While classically trivial, it requires some organisation to handle within effective categories:
\begin{lem}\label{ex:approximationbyschnorr2} (Flipping an approximation) Suppose that $p\geq 1$ computable (resp. $p=0$).

For each $L_p(\nu)$ Schnorr test $f$, let $g_s$ be from the countable dense set of $L_p(\nu)$ as in Proposition~\ref{prop:exptoflsc} (resp. Proposition~\ref{prop:exptoflscL0}), so that $g_s\leq g_{s+1}$ everywhere and $f=\sup_s g_s$ everywhere and $g_s\rightarrow f$ fast in $L_p(\nu)$. Using Proposition~\ref{prop:shiftbasistolpst}, let $f_s, \breve{f}_s$ be non-negative lsc and usc respectively with $g_s=f_s=\breve{f}_s$ on $\mathsf{KR}^{\nu}$. Then by the previous proposition, we have:
\begin{itemize}[leftmargin=*]
\item $f_s, \breve{f}_s$ are finite on $\mathsf{KR}^{\nu}$.
\item $f\ominus f_s$ is an $L_p(\nu)$ Schnorr test and is equal on  $\mathsf{KR}^{\nu}$ to $f-g_s$.
\item $f\ominus f_s$ is non-increasing and $f\ominus f_s\rightarrow 0$ on $\mathsf{KR}^{\nu}$.
\item for $t>s$, similarly $f_t\ominus f_s$ is an $L_p(\nu)$ Schnorr test and is equal on $\mathsf{KR}^{\nu}$ to $f_t-g_s$.
\end{itemize}
\end{lem}

Now we turn to self-location. The idea is that given a certain kind of computable ``chart'' of the computable Polish space, the Schnorr randoms can weakly compute their position on the chart. (For the notion of weak computation, see Definition~\ref{defn:core}(\ref{defn:core:15}).

\begin{lem}\label{lem:sefllocate} (Self-location lemma).

Suppose that $\nu$ is a computable point of $\mathcal{P}(X)$.

Suppose that $V_m$ is a computable sequence of c.e. opens with uniformly computable $\nu$-measure. 

Suppose that $x$ is in $\mathsf{SR}^{\nu}$. Then $x$ weakly computes the element $\{m: x\in V_m\}$ of Cantor space.
\end{lem}

\begin{proof}
Let $B_0, B_1, \ldots$ be a $\nu$-computable basis, with associated sequence $C_0, C_1, \ldots$ of effectively closed supersets of the same measure. Let $B_{m,t}$ be a computable subsequence such that $V_m =\bigcup_t B_{m,t}$. Since $V_m$ and the $B_{m,t}$ have uniformly computable $\nu$-measure, there is a computable function $m\mapsto s(m)$ such that $\nu(V_m\setminus U_m)<2^{-m}$ where $U_m=\bigcup_{t\leq s(m)} B_{m,t}$. Let $D_m=\bigcup_{t\leq s(m)} C_{m,t}$, which is an effectively closed set equal to $U_m$ on $\mathsf{KR}^{\nu}$. Then $f=\sum_m I_{V_m\setminus D_m}$ is an $L_1(\nu)$ Schnorr test.

Let $x$ be in $\mathsf{SR}^{\nu}$. Since $f(x)<\infty$, there are only finitely many $m$ such that $x$ is in $V_m\setminus D_m$. Hence,  there are there are only finitely many $m$ such that $x$ is in $V_m\setminus U_m$. Then the sets $\{m: x\in V_m\}$ and $\{m: x\in U_m\}$ differ by only finitely much and hence are Turing equivalent. 

Since $U_m$ comes from the algebra generated by the $\nu$-computable basis, using the $\nu$-computable basis as in Proposition~\ref{prop:krplusbasis} we can compute indexes for c.e. opens $U_m^{\prime}$ such that $U_m^{\prime}=X\setminus U_m$ on $\mathsf{KR}^{\nu}$. Since both $U_m$ and $U_m^{\prime}$ are uniformly c.e. open, choose computable sequences $p_{m,i}, p_{m,i}^{\prime}$ from the countable dense set and $\epsilon_{m,i}, \epsilon_{m,i}^{\prime}$ from $\mathbb{Q}^{>0}$ such that $U_m = \bigcup_i B(p_{m,i},\epsilon_{m,i})$ and $U_m^{\prime} = \bigcup_i B(p_{m,i}^{\prime},\epsilon_{m,i}^{\prime})$, where $B(p,\epsilon)$ denotes again the open ball around $p$ of radius $\epsilon$. We can enumerate these sets as $U_m=\bigcup_s U_{m,s}$ and $U_m^{\prime} = \bigcup_s U_{m,s}^{\prime}$, where $U_{m,s} = \bigcup_{i\leq s} B(p_{m,i},\epsilon_{m,i})$ and $U_{m,s}^{\prime} = \bigcup_{i\leq s} B(p_{m,i}^{\prime},\epsilon_{m,i}^{\prime})$.

Consider a sequence from the countable dense set which converges fast to our point $x$ in $\mathsf{SR}^{\nu}$. Given $m$, to compute from the sequence whether $x$ is in $U_m$, we simply start enumerating both $U_m$ and $U_m^{\prime}$: eventually $x$ gets in one of them (and $x$ only ever gets in one of them), and we use the sequence to determine when this happens, by Proposition~\ref{prop:closure:reals:general:0.4}(\ref{prop:closure:reals:general:0.4:shadow}).
\end{proof}

Here are some simple applications of self-location, which we use to obtain the information about weak computation in Theorem~\ref{thm:newnewlevy}(\ref{levy:extra:5}):
\begin{prop}\label{prop:consequences:selflocation}
Suppose that $p\geq 1$ computable (resp. $p=0$).
\begin{enumerate}[leftmargin=*]
    \item \label{prop:consequences:selflocation:1} Suppose that $f_m$ is a sequence of $L_p(\nu)$ Schnorr tests such that $f_m$ is non-increasing on $\mathsf{SR}^{\nu}$ and such that $f_m\rightarrow 0$ on $\mathsf{SR}^{\nu}$. Every $x$ in $\mathsf{SR}^{\nu}$ weakly computes a modulus of convergence for $f_m(x)\rightarrow 0$.
    \item \label{prop:consequences:selflocation:2} Suppose that $f$ is an $L_p(\nu)$ Schnorr test. Suppose that $f_s$ is the approximation as in Proposition~\ref{prop:exptoflsc} (resp. Proposition~\ref{prop:exptoflscL0}). Every $x$ in $\mathsf{SR}^{\nu}$ weakly computes a modulus of convergence for $f_s(x)\rightarrow f(x)$.
\end{enumerate}
\end{prop}
\begin{proof}
For (\ref{prop:consequences:selflocation:1}), using Lemma~\ref{lem:usingst} (resp. in conjunction with Proposition~\ref{prop:pushforwards:2}), choose a computable sequence of reals $r_i$ decreasing to zero such that the c.e. open $V_{m,i}=f_m^{-1}(r_i,\infty]$ has $\nu$-computable measure, uniformly. Consider a sequence from the countable dense set which converges fast to $x$. By the Self-location lemma, we can Turing compute from it the ``chart'' set $C=\{(m,i): x\in V_{m,i}\}$. Let $\epsilon>0$ be rational. We show how to compute from $C$ a natural number $m(\epsilon)$ such that $f_n(x)<\epsilon$ for all $n\geq m(\epsilon)$. By hypothesis, $f_n(x)$ decreases down to zero. Hence to compute $m(\epsilon)$ from $C$ we just search for $r_i<\epsilon$ and then search for $m$ with $x\notin V_{m,i}$.

For~(\ref{prop:consequences:selflocation:2}), just use Lemma~\ref{ex:approximationbyschnorr2} to rewrite the convergence $f_s\rightarrow f$ as $(f\ominus g_s)\rightarrow 0$ on $\mathsf{KR}^{\nu}$, where $g_s$ is an $L_p(\nu)$ Schnorr test equal to $f_s$ on $\mathsf{KR}^{\nu}$, and then use  (\ref{prop:consequences:selflocation:1}). 
\end{proof}

\section{Recovering pointwise values on Schnorr randoms}\label{sec:recover}

In this section we prove some results about pointwise limits existing on the Schnorr randoms for various effective functions convering fast in $L_p(\nu)$. By way of motivation for these kinds of results, consider $X=[0,1]$ and let $\nu$ be Lebesgue measure and recall the canonical example of $L_1(\nu)$ convergence with $\nu$-a.s. \emph{lack} of pointwise convergence:
\begin{align*}
\ f_1 = I_{[0,\frac{1}{2})}, \ f_2 = I_{[\frac{1}{2}, 1]}, \ f_3 = I_{[0,\frac{1}{4})}, \ f_4 = I_{[\frac{1}{4}, \frac{1}{2})}, \ f_5 = I_{[\frac{1}{2}, \frac{3}{4})}, \ f_6 = I_{[\frac{3}{4}, 1]}, \ \ldots
\end{align*}
 Proposition~\ref{prop:schnorrtest} below says that the slow $L_1(\nu)$-convergence in this example is essential to the lack of pointwise limits on $\mathsf{SR}^{\nu}$. By modifying the events in $f_n$ to be open, one similarly gets a sequence of $L_1(\nu)$ Schnorr tests $g_n$ which lacks pointwise limits on all $\mathsf{KR}^{\nu}$. Proposition~\ref{prop:schnorrtestforlsc} likewise says that the slow $L_1(\nu)$ convergence of $g_n$ is essential to the  $\nu$-a.s. lack of pointwise limits on $\mathsf{SR}^{\nu}$.

In the setting of $p=1$ and $X=[0,1]^k$ and $\nu$ being the $k$-fold product of Lebesgue measure on $[0,1]$, the following result is due to Pathak, Rojas, and Simpson, who used sequential Schnorr tests.\footnote{\cite[Lemma 3.7]{Pathak2014aa}. See also \cite[\S{3.3}]{Rute2012aa} and \cite[Chapter 3]{Yu1987aa} (cf. \cite[p. 394]{Simpson2009aa}).} 
\begin{prop}\label{prop:schnorrtest}
Suppose that $p\geq 1$ is computable. Suppose that $f$ is a computable point of $L_p(\nu)$. Suppose that $f_n$ is a computable sequence from the countable dense set of $L_p(\nu)$ such that $f_n\rightarrow f$ fast in $L_p(\nu)$. Then $\lim_n f_n$ exists on $\mathsf{SR}^{\nu}$ and is a version of $f$.

Moreover, on $\mathsf{SR}^{\nu}$ this limit does not depend on the choice of $f_n$ or the choice of the $\nu$-computable basis. 

Finally, if $f$ is in addition an $L_p(\nu)$ Schnorr test, then $\lim_n f_n(x)=f(x)$ for all $x$ in $\mathsf{SR}^{\nu}$.
\end{prop}
\begin{proof} (Sketch)
Let $g = \sum_i \left|f_i-f_{i+1}\right|$. Then using Proposition~\ref{prop:shiftbasistolpst}, it is equal on $\mathsf{KR}^{\nu}$ to a $L_p(\nu)$ Schnorr test. This shows that $f_i(x)$ is a Cauchy sequence for $x$ in $\mathsf{SR}^{\nu}$.

If $f_n^{\prime}$ is another such witness to the $L_p(\nu)$ computabilty of $f$, then let  $h = \sum_i \left|f_i-f_i^{\prime}\right|$, and it is similarly equal on $\mathsf{KR}^{\nu}$ to a $L_p(\nu)$ Schnorr test. 

To see that the partially defined function $\lim_n f_n$ is a version of $f$ in $L_p(\nu)$, simply note that classically some subsequence $f_n^{\prime}:=f_{m(n)}$ converges to $f$ $\nu$-a.s. and so $\lim_n f_n^{\prime}$ is a version of $f$. The sequence $f_n^{\prime}$ is computable in some oracle, and so by the previous paragraph we get that $\lim_n f_n, \lim_n f_n^{\prime}$ agree on all the Schnorr randoms relative to that oracle, and so $\lim_n f_n$ is also a version of $f$.

The final remark follows from the second paragraph by choosing $f_n^{\prime}$ to be the  approximation to the lsc function $f$ from Proposition~\ref{prop:exptoflsc}.
\end{proof}

The following is an analogue of the above proposition for $L_0(\nu)$. This proposition is essentially the natural effectivization of the classical proof that Cauchy-in-measure sequences converge in measure.\footnote{E.g. \cite[Theorem 2.30 p. 61]{Folland1999aa}.}
\begin{prop}\label{prop:srlimexistsL0v1}
Suppose that $f$ is a computable point of $L_0(\nu)$. Suppose that $f_n$ is a computable sequence from the countable dense set of $L_0(\nu)$ such that $f_n\rightarrow f$ at a geometric rate of convergence in $L_0(\nu)$. Then for all $x$ in $\mathsf{SR}^{\nu}$ one has that $\lim_n f_n(x)$ exists and is a version of $f$.

This limit does not depend on the choice of $f_n$ or the choice of the $\nu$-computable basis or the choice of the rate of geometric convergence.

Hence, if $f$ is in addition an $L_0(\nu)$ Schnorr test, then $\lim_n f_n(x)=f(x)$ for all $x$ in $\mathsf{SR}^{\nu}$.
\end{prop}
Note that by the computable embedding of $L_p(\nu)$ into $L_0(\nu)$, the limit in this proposition agrees with the limit in the previous proposition on $\mathsf{SR}^{\nu}$.
\begin{proof}
We may suppose that the geometric rate of convergence $b>1$ is rational. Then we can compute whether $b^{-j}$ is rational or irrational, and hence uniformly in $j\geq 0$ we have that $b^{-j}$ has uniformly computable left- and right Dedekind cuts. Since the $f_j$ are from the countable dense set, so is $\left|f_j-f_{j+1}\right|$, and hence we can write it as $\sum_{k=1}^{n_j} q_{j,k}\cdot I_{A_{j,k}}$, where $q_{j,k}\geq 0$ is rational and the events $\{A_{j,k}: 1\leq k\leq n_j\}$ are pairwise disjoint and come from the algebra generated by a  $\nu$-computable basis. By Proposition~\ref{prop:krplusbasis}, this is equal on $\mathsf{KR}^{\nu}$ to the finite sum $\sum_{k=1}^{n_j} q_{j,k} \cdot I_{U_{j,k}}$, where $U_{j,k}$ is a c.e. open which is equal on $\mathsf{KR}^{\nu}$ to $A_{j,k}$. Let $E_j =  \{x\in X: \sum_{k=1}^{n_j} q_{j,k} \cdot I_{U_{j,k}}>2\cdot b^{-j}\}$, which is a c.e. open since it is equal to $\bigcup_{K\in J_j} \bigcap_{k\in K} U_{j,k}$, where $J_j = \{K\subseteq  [1,n_j]: \sum_{k\in K} q_{j,k}>2\cdot b^{-j}\}$, and $J_j$ is computable since $b^{-j}$ has uniformly computable right Dedkind cuts. Then $E_j$ is equal on $\mathsf{KR}^{\nu}$ to $\{x\in X: \left|f_j-f_{j+1}\right|>2\cdot b^{-j}\}$ and $E_j$ is a c.e. open with computable $\nu$-measure, uniformly in $j\geq 0$.

We then have $\nu(E_j)\leq \nu(\left|f-f_j\right|> b^{-j}) +  \nu(\left|f-f_{j+1}\right|> b^{-(j+1)}) \leq b^{-j}+b^{-(j+1)} \leq 2\cdot b^{-j}$. Letting $F_k$ be the c.e. open $\bigcup_{j=k}^{\infty} E_j$, we have $\nu(F_k)\leq 2\cdot \frac{b}{b-1} \cdot b^{-k}$. Further for $k^{\prime}>k$ we have $\bigcup_{j=k}^{k^{\prime}} E_j$ has computable $\nu$-measure since it is a finite union of events with $\nu$-computable measure coming from the algebra generated by a $\nu$-computable basis. And then $\nu(F_k)$ is computable since we can approximate it by $\nu(\bigcup_{j=k}^{k^{\prime}} E_j)$ since $\nu(F_k)-\nu(\bigcup_{j=k}^{k^{\prime}} E_j)\leq \nu(F_{k^{\prime}})\leq 2\cdot \frac{b}{b-1} \cdot b^{-k^{\prime}}$. Hence $\sum_k I_{F_k}$ is an $L_1(\nu)$ Schnorr test. 

If a point is in $\mathsf{SR}^{\nu}$, then it is not in some $F_k$, while it is in $\mathsf{KR}^{\nu}$. Then we argue for the following six items about elements of $\mathsf{KR}^{\nu}\setminus F_k$:
\begin{enumerate}[leftmargin=*]
    \item \label{prop:srlimexistsL0v1:1} For all $k\geq 0$ and all $x$ in $\mathsf{KR}^{\nu}\setminus F_k$, for all $j_1> j_0\geq k$ we have $\left|f_{j_0}(x)-f_{j_1}(x)\right|\leq \sum_{i=j_0}^{j_1-1} \left|f_i(x)-f_{i+1}(x)\right| \leq \sum_{i=j_0}^{\infty} 2\cdot b^{-i} \leq 2\cdot \frac{b}{b-1} \cdot b^{-j_0}$.
    \item \label{prop:srlimexistsL0v1:2} Hence for all $k\geq 0$ and all $x$ in $\mathsf{KR}^{\nu}\setminus F_k$, we have that $f_{j}(x)$ for $j\geq k$ is a Cauchy sequence and thus $\lim_j f_j(x)$ exists.
    \item \label{prop:srlimexistsL0v1:3} For all $x$ in $\mathsf{KR}^{\nu}\setminus F_k$ and all $j\geq k\geq 0$, we have $\left|f_j(x)-\lim_j f_j(x)\right|\leq 2\cdot \frac{b}{b-1} \cdot b^{-j}$. For, let $\epsilon>0$. Let $j_0=j$ and choose $j_1>j_0$ such that  $\left|f_{j_1}(x)-\lim_j f_j(x)\right|<\epsilon$. Then by (\ref{prop:srlimexistsL0v1:1}) one has that $\left|f_j(x)-\lim_j f_j(x)\right|\leq \left|f_{j_1}(x)-\lim_j f_j(x)\right|+\left|f_{j_1}(x)-f_{j_0}(x)\right|<\epsilon+2\cdot \frac{b}{b-1} \cdot b^{-j_0}$. Since this holds for all $\epsilon>0$, we are done. 
    \item \label{prop:srlimexistsL0v1:4} Since $\mathsf{SR}^{\nu}$ is a $\nu$-measure one set, one has that $\lim_j f_j$ exists $\nu$-a.s. 
    \item \label{prop:srlimexistsL0v1:5} Further one has that $f_j\rightarrow \lim_j f_j$ in $L_0(\nu)$. For let $\epsilon>0$. Choose $k$ such that $2\cdot \frac{b}{b-1} \cdot b^{-k}<\epsilon$. Let $j\geq k$, so that $2\cdot \frac{b}{b-1} \cdot b^{-j}<\epsilon$. Then by (\ref{prop:srlimexistsL0v1:3})  we have $\nu(\left|f_j-\lim_j f_j\right|>\epsilon)\leq\nu(\left|f_j-\lim_j f_j\right|>2\cdot \frac{b}{b-1} \cdot b^{-j})\leq \nu(F_k)\leq 2\cdot \frac{b}{b-1} \cdot b^{-k}<\epsilon$.
    \item \label{prop:srlimexistsL0v1:6} Since both  $f_j\rightarrow \lim_j f_j$ in $L_0(\nu)$ and  $f_j\rightarrow f$ in $L_0(\nu)$, we have that $\lim_j f_j = f$ $\nu$-a.s.
\end{enumerate}

Suppose that $h_j$ is another computable sequence from the countable dense set of $L_0(\nu)$ such that $h_j\rightarrow f$ at a geometric rate $c>1$ of convergence in $L_0(\nu)$.  Note that $\lim_j f_j$ and $\lim_j h_j$ are equal $\nu$-a.s. since they are both equal $\nu$-a.s. to $f$. Let $G_j$ and $H_k$ be constructed from $h_j$ and $c$ just as we constructed $E_j$ and $F_k$ from $f_j$ and $b$ above. Let $d=\min(b,c)$, rational number~$>1$. Let $e=\max\{2\cdot \frac{b}{b-1},2\cdot \frac{c}{c-1}\}$. Note that for all $j\geq 0$, we have $e\cdot d^{-j}\geq 2\cdot \frac{b}{b-1} \cdot b^{-j}$ and $e\cdot d^{-j}\geq 2\cdot \frac{c}{c-1} \cdot c^{-j}$. Let $D_j$ be a c.e. open which is equal on $\mathsf{KR}^{\nu}$ to $\{x\in X: \left|f_j(x)-h_j(x)\right|>3 \cdot e\cdot d^{-j}\}$, and note that $D_j$ has computable $\nu$-measure, as in the argument of the first paragraph of the proof. Then one has that $\nu(D_j)\leq \nu(\left|f_j-\lim_j f_j\right|> 2\cdot \frac{b}{b-1} \cdot b^{-j})+\nu(\left|\lim_j f_j-\lim_j h_j\right|>e\cdot d^{-j})+\nu(\left|h_j-\lim_j h_j\right|> 2\cdot \frac{c}{c-1} \cdot c^{-j})\leq \nu(F_j)+0+\nu(H_j)\leq  2\cdot \frac{b}{b-1} \cdot b^{-j} + 2\cdot \frac{c}{c-1} \cdot c^{-j}$, where the middle term is zero since $\lim_j f_j$ and $\lim_j h_j$ are equal $\nu$-a.s.  Hence $\sum_j I_{D_j}$ is an $L_1(\nu)$ integral test, and thus, for $x$ in $\mathsf{SR}^{\nu}$ one has that $\lim_j f_j(x) =\lim_j h_j(x)$.


As in the previous proof, the limit does not depend on the choice of $\nu$-computable basis since $\nu$-computable bases are closed under effective unions (cf. Proposition~\ref{prop:union:mucompbasis}).

The remarks about $\lim_n f_n$ being a version of $f$, and the remark about $L_0(\nu)$ Schnorr tests, follows as in the proof of the previous proposition.
\end{proof}

There is a result similar to Proposition~\ref{prop:schnorrtest} when the $f_n$ are themselves $L_p(\nu)$ Schnorr tests:
\begin{prop}\label{prop:schnorrtestforlsc}
Suppose that $p\geq 1$ is computable (resp. $p=0$). Suppose that $f_n$ are uniformly $L_p(\nu)$ Schnorr tests with $f_n\rightarrow f$ fast in $L_p(\nu)$, so that $f$ is also a computable point of $L_p(\nu)$. Then $\lim_n f_n(x)$ exists and for all $x$ in $\mathsf{SR}^{\nu}$. If $f$ is also an $L_p(\nu)$ Schnorr test, then $\lim_n f_n(x)=f(x)$ for all $x$ in $\mathsf{SR}^{\nu}$.
\end{prop}
\begin{proof}
By Proposition~\ref{prop:exptoflsc} (resp. Proposition~\ref{prop:exptoflscL0}), choose doubly-indexed computable sequence $f_{n,s}$ from the countable dense set of $L_p(\nu)$ such that for all $n\geq 0$ we have $0\leq f_{n,s}\leq f_{n,s+1}\leq f_n$ everywhere and $f_n=\sup_s f_{n,s}$ and $f_{n,s}\rightarrow f_n$ fast in $L_p(\nu)$. Then $f_{n+1,n+1}\rightarrow f$ fast in $L_p(\nu)$. Hence, by Proposition~\ref{prop:schnorrtest} (resp. Proposition~\ref{prop:srlimexistsL0v1}), 
$\lim_n f_{n,n}$ exists on $\mathsf{SR}^{\nu}$. Note that by these propositions, if $f$ is also an $L_p(\nu)$ Schnorr test then $\lim_n f_{n,n}= f$ on $\mathsf{SR}^{\nu}$.

It suffices to show that $\lim_n (f_n-f_{n,n}) =0$ on $\mathsf{SR}^{\nu}$. Use Lemma~\ref{ex:approximationbyschnorr2} to rewrite what we are to show as $\lim_n (f_n\ominus g_{n,n}) =0$, where $g_{n,n}$ is an $L_p(\nu)$ Schnorr test equal to $f_{n,n}$ on $\mathsf{KR}^{\nu}$, so that $f_n\ominus g_{n,n}$ is an $L_p(\nu)$ Schnorr test equal to $f_n-f_{n,n}$ on $\mathsf{KR}^{\nu}$. By Lemma~\ref{lem:usingst} in conjunction with Proposition~\ref{prop:pushforwards} (resp. Proposition~\ref{prop:pushforwards:2}), choose a computable sequence $\eta_n$ in the interval $(2^{-(n+1)}, 2^{-n})$ such that the $(f_n\ominus g_{n,n})^{-1}(\eta_n,\infty]$ has computable $\nu$-measure. Then $U_n=(f_n\ominus g_{n,n})^{-1}(\eta_n,\infty]$ is c.e. open with $\nu$-computable measure. Let $h_n=(f_n\ominus g_{n,n})\cdot I_{U_n}$ which is an $L_p(\nu)$ Schnorr test. 
Let $h=\sum_n h_n$ and $h_m = \sum_{n<m} h_n$. Then $\|h-h_m\|_p\leq \sum_{n>m} \|h_n\|_p \leq \sum_{n>m} \|f_n-f_{n,n}\|_p \leq \sum_{n>m} 2^{-n} \leq 2^{-m}$. Then $h$ is an $L_p(\nu)$ Schnorr test: it is non-negative lsc as a sum of non-negative lsc functions, and the sequence $h_m$ is uniformly $L_p(\nu)$-computable and we just showed that $h_m\rightarrow h$ fast in $L_p(\nu)$. Now we verify $\lim_n (f_n-f_{n,n}) =0$ on $\mathsf{SR}^{\nu}$. Let $x$ in $\mathsf{SR}^{\nu}$. Let $\epsilon>0$. Since $x$ is in $\mathsf{SR}^{\nu}$, choose $n_0\geq 0$ such that we have the estimate $\sum_{n\geq n_0} (f_n(x)-f_{n,n}(x))\cdot I_{U_n}(x)<\epsilon$. Choose $n_1\geq n_0$ such that $\eta_n<\epsilon$ for all $n\geq n_1$. Let $n\geq n_1$. If $x$ is in $U_n$, then by our estimate we have $f_n(x)-f_{n,n}(x)<\epsilon$. If $x$ is not in $U_n$, then by the definition of $U_n$ we have $f_n(x)-f_{n,n}(x)\leq \eta_n<\epsilon$.

\end{proof}

\section{Classical features of the maximal function}\label{sec:maximal:classical}

Suppose that $\nu$ is a point of $\mathcal{P}(X)$ and $\mathscr{F}_n$ is any increasing filtration of Borel subsets of $X$. In this section, we recall some classical features of the maximal function $f^{\ast}=\sup_n \mathbb{E}_{\nu}[f\mid \mathscr{F}_n]$ of an integrable function $f$.

First, we recall the following, which gives us information about the codomain of the maximal function:
\begin{lem}\label{lem:doob}
\item[]
\begin{enumerate}[leftmargin=*]
\item \label{lem:doob:1} If $p>1$ then  $\|g^{\ast}\|_p\leq \frac{p}{p-1} \cdot \|g\|_p$ for $g$ in $L_p(\nu)$, and the maximal function maps $L_p(\nu)$ to $L_p(\nu)$.
\item \label{lem:doob:2} If $p=1$, then the maximal function maps $L_p(\nu)$ to $L_0(\nu)$.
\end{enumerate}
\end{lem}
\begin{proof}
For $p>1$ and $g$ in $L_p(\nu)$, the sequence $\mathbb{E}_{\nu}[g\mid \mathscr{F}_n]$ is a non-negative martingale, and so by Doob's Maximal Inequality\footnote{\cite[Theorem 9.4 pp. 505-506]{Gut2013-ou}.}  followed by conditional Jensen we have: $\|\sup_{n\leq m} \mathbb{E}_{\nu}[g\mid \mathscr{F}_n]\|_p\leq \frac{p}{p-1}\cdot  \|\mathbb{E}_{\nu}[g\mid \mathscr{F}_m]\|_p\leq \frac{p}{p-1} \cdot \|g\|_p$. Then by the Monotone Convergence Theorem we have $\|g^{\ast}\|_p\leq \frac{p}{p-1} \|g\|_p$. For $p=1$, let $\mathscr{F}_{\infty}$ be the $\sigma$-algebra generated by all the $\mathscr{F}_n$. 
By the classical L\'evy Upward Theorem we have that $\mathbb{E}_{\nu}[g\mid \mathscr{F}_n]\rightarrow \mathbb{E}_{\nu}[g\mid \mathscr{F}_{\infty}]$ $\nu$-a.s. which shows that $g^{\ast}$ is finite $\nu$-a.s., and hence that $g^{\ast}$ is in $L_0(\nu)$.
\end{proof}

The following proposition collects together all the other classical facts about the maximal function which we need:
\begin{prop}\label{prop:abstract:1}\mbox{\;}
\begin{enumerate}[leftmargin=*]
    \item \label{prop:abstract:1a} For $p>1$, the maximal function ${\cdot}^{\ast}: L_p(\nu) \rightarrow L_p(\nu)$ is uniformly continuous with modulus $m(\epsilon)=\frac{p-1}{p}\cdot \epsilon$ of uniform continuity.
    \item \label{prop:abstract:1b} For $p=1$, the maximal function ${\cdot}^{\ast}: L_p(\nu) \rightarrow L_0(\nu)$ is uniformly continuous with a modulus $m(\epsilon)=\epsilon^2$ of uniform continuity.
\end{enumerate}
\end{prop}
\begin{proof}
Let $p>1$ and $f$ in $L_p(\nu)$. By Lemma~\ref{lem:doob} one has for $f,g$ in $L_p(\nu)$ with $\|f-g\|_p<\frac{p-1}{p}\cdot \epsilon$ that $\|f^{\ast}-g^{\ast}\|_p\leq \|\left|f-g\right|^{\ast}\|_p\leq \frac{p}{p-1} \cdot \|f-g\|_p <\epsilon$.

Suppose that $f,g$ are in $L_1(\nu)$ with $\|f-g\|_1<\epsilon^2$. Then one has
\begin{align*}
& \nu(\left|f^{\ast}-g^{\ast}\right|>\epsilon) \leq \nu(\left|f-g\right|^{\ast})>\epsilon)= \lim_m  \nu(\sup_{n\leq m} \mathbb{E}[\left| f-g\right|\mid \mathscr{F}_n] >\epsilon) \\
\leq & \lim_m \epsilon^{-1} \cdot \mathbb{E}_{\nu} \mathbb{E}_{\nu}[\left|f-g\right|\mid \mathscr{F}_m]\leq \epsilon^{-1} \cdot \mathbb{E}_{\nu} \left|f-g\right|  \leq  \epsilon^{-1} \epsilon^{2} =\epsilon 
\end{align*}
The first step of the second line follows from Doob's Submartingale Inequality,\footnote{\cite[137-138]{Williams1991aa}.} where we apply it to the martingale $\mathbb{E}[\left| f-g\right|\mid \mathscr{F}_n]$. 
\end{proof}

\section{An abstract version of L\'evy's Theorem for Schnorr randomness}\label{sec:abstract}

Before we state our abstract version of L\'evy's Theorem, we need the following definition:
\begin{defn}\label{defn:approx}
Suppose that $\mathsf{XR}^{\nu}$ has $\nu$-measure one.  

Suppose that $\mathcal{C}$ is a class of $L_p(\nu)$ Schnorr tests. 

Then the $L_p(\nu)$ Schnorr tests are \emph{approximated from below} on $\mathsf{XR}^{\nu}$ by $\mathcal{C}$ if from an index for an $L_p(\nu)$ Schnorr test $f$ one can compute
\begin{itemize}[leftmargin=*]
\item an index for a sequence of $L_p(\nu)$ Schnorr tests $f_s$ in $\mathcal{C}$ such that $f_s\leq f_{s+1}$ on $\mathsf{XR}^{\nu}$ and $f=\sup_s f_s$ on $\mathsf{XR}^{\nu}$ and $f_s\rightarrow f$ fast in $L_p(\nu)$; and
\item an index for a sequence of non-negative usc functions $\breve{f}_s$ equal to $f_s$ on $\mathsf{XR}^{\nu}$.
\end{itemize}
\end{defn}

Recall from Proposition~\ref{prop:ominusproperties} that the $f_s, \breve{f}_s$ are finite on $\mathsf{XR}^{\nu}$, and we define $f\ominus f_s = \max(0, f-\breve{f}_s)$, and we have that $f\ominus f_s$ is an $L_p(\nu)$ Schnorr test equal on $\mathsf{XR}^{\nu}$ to $f-f_s$. Similarly if $t\geq s$, then we define $f_t\ominus f_s = \max(0, f_t-\breve{f}_s)$, and we have that $f_t\ominus f_s$ is an $L_p(\nu)$ Schnorr test equal on $\mathsf{XR}^{\nu}$ to $f_t-f_s$.

The basic example of Definition~\ref{defn:approx} comes from Lemma~\ref{ex:approximationbyschnorr2}.

The following is our abstract version of L\'evy's Theorem for Schnorr randomness. It is an abstract version in that we are not told more about the function  $E[\cdot \mid n](\cdot)$ other than that what is stated explicitly in the hypotheses (\ref{prop:newlevy:1})-(\ref{prop:newlevy:4}). In particular, we do not assume that $E[\cdot \mid n](\cdot)$ comes from an effective disintegration, although in the next section we will show that effective disintegrations satisfy the hypotheses of the theorem.

\begin{thm}\label{thm:newlevy}
Suppose that $\nu$ is a computable point of $\mathcal{P}(X)$. Suppose that $\mathscr{F}_n$ is an increasing filtration of Borel sets. Suppose that $p\geq  1$ is computable.

Suppose that $E[\cdot \mid n](\cdot): \mathbb{L}_p^+(\nu)\times X\rightarrow [0,\infty]$ is a function such that for every $f$ in $\mathbb{L}_p^+(\nu)$, one has that $E[f \mid n]:X\rightarrow [0,\infty]$ is a version of the conditional expectation of $f$ with respect to~$\mathscr{F}_n$. Define  the function $\cdot^{\flat}(\cdot): \mathbb{L}_p^+(\nu)\times X\rightarrow [0,\infty]$ by $f^{\flat}(x) = \sup_n E[\cdot \mid n](x)$.

Suppose that $\mathsf{XR}^{\nu}$ is a superset of $\mathsf{SR}^{\nu}$.

Suppose that:
\begin{enumerate}[leftmargin=*, align=left, label=(\Roman*), ref=\Roman*]
\item \label{prop:newlevy:1} $E[\cdot\mid n]$ maps non-negative lsc functions to non-negative lsc functions.
\item \label{prop:newlevy:2} $E[\cdot\mid n]$ satisfies the following properties on $\mathsf{XR}^{\nu}$:
\begin{enumerate}[leftmargin=*]
\item \label{prop:newlevy:2a} If $f\leq g$ on $\mathsf{XR}^{\nu}$, then $E[f\mid n]\leq E[g\mid n]$ on $\mathsf{XR}^{\nu}$;
 \item \label{prop:newlevy:2b} If $c$ in $\mathbb{R}^{\geq 0}$ then $c \cdot E[f\mid n] = E[c\cdot f\mid n]$ on $\mathsf{XR}^{\nu}$;
\item \label{prop:newlevy:2c} $E[f+g\mid n] = E[f\mid n]+E[g\mid n]$ on $\mathsf{XR}^{\nu}$;
\end{enumerate}
\item \label{prop:newlevy:3} Both $E[\cdot\mid n]$ and $\cdot^{\flat}$ send the countable dense set of $L_p^+(\nu)$ uniformly to computable points of~$L_p^+(\nu)$.
\item \label{prop:newlevy:4} The $L_p(\nu)$ Schnorr tests are approximated from below on $\mathsf{XR}^{\nu}$ by a class $\mathcal{C}$ of $L_p(\nu)$ Schnorr tests such that $\lim_n E[f\mid n]=f$ on~$\mathsf{XR}^{\nu}$ for each $f$ in $\mathcal{C}$.
\end{enumerate}

Then the following three items are equivalent for $x$ in $X$: 
  \begin{enumerate}[leftmargin=*]
  \item \label{thm:abstractlevy:1} $x$ is in $\mathsf{SR}^{\nu}$.
  \item \label{thm:abstractlevy:2} $x$ is in $\mathsf{XR}^{\nu}$ and $\lim_n E[f\mid n](x)=f(x)$ for every $L_p(\nu)$ Schnorr test $f$.
  \item \label{thm:abstractlevy:3} $x$ is in $\mathsf{XR}^{\nu}$ and $\lim_n E[f\mid n](x)$ exists for every $L_p(\nu)$ Schnorr test $f$ and $\lim_n E[I_U\mid n](x)=I_U(x)$ for every c.e. open $U$ with $\nu$-computable measure. 
\end{enumerate}

We also have:

\begin{enumerate}[leftmargin=*, align=left, label=(\roman*), ref=\roman*]
 \item \label{thm:abstractlevy:extra2} Every Borel set $B$ is equal on $\mathsf{SR}^{\nu}$ to a Borel set $B^{\prime}$ in the $\sigma$-algebra generated by the union of the $\mathscr{F}_n$.\footnote{Indeed, if $n\geq 1$ and $B$ is in $\utilde{\Sigma}^0_n$ (resp. $\utilde{\Pi}^0_n$) then we can take $B^{\prime}$ to be $\utilde{\Sigma}^0_{n+5}$ (resp. $\utilde{\Pi}^0_{n+5}$), and if $\alpha\geq \omega$ and $B$ is $\utilde{\Sigma}^0_{\alpha}$ (resp. $\utilde{\Pi}^0_{\alpha}$) then  $B^{\prime}$ may be taken to be $\utilde{\Sigma}^0_{\alpha}$ (resp. $\utilde{\Pi}^0_{\alpha}$). And the same for the lightface classes.}
 \item \label{thm:abstractlevy:extra3} Suppose one adds to (\ref{prop:newlevy:4}) the condition that every $x$ in $\mathsf{SR}^{\nu}$ weakly computes a modulus of convergence for $E[f\mid n](x)\rightarrow f(x)$, uniformly in $f$ from $\mathcal{C}$. Then one can further conclude for every $x$ in $\mathsf{SR}^{\nu}$ and every $L_p(\nu)$ Schnorr test $f$, the point $x$ weakly computes a modulus of convergence for $E[f\mid n](x)\rightarrow f(x)$ in (\ref{thm:abstractlevy:2}).
 \end{enumerate}
\end{thm}

Regarding (\ref{thm:abstractlevy:extra2}), note that this is saying that the hypotheses of the Theorem amount collectively to an assumption that the union of the filtration generates a $\sigma$-algebra very close to the Borel $\sigma$-algebra, from the perspective of $\nu$.

\begin{proof}

First we note three things about the maps $E[\cdot\mid n]$ and $\cdot^{\flat}$ and $p\geq 1$ computable.

For $p\geq 1$, the map $E[\cdot\mid n]:\mathbb{L}_p^+(\nu)\rightarrow \mathbb{L}_p^+(\nu)$ maps $L_p(\nu)$ Schnorr tests uniformly to $L_p(\nu)$ Schnorr tests. For, by (\ref{prop:newlevy:1}), it sends non-negative lsc functions to non-negative lsc functions. And by conditional Jensen and (\ref{prop:newlevy:3})  and Propositions~\ref{prop:contcomp},\ref{prop:suffforcont}, it sends $L_p^+(\nu)$ computable points to $L_p^+(\nu)$ computable points.

For $p>1$, the map $\cdot^{\flat}:\mathbb{L}_p^+(\nu)\rightarrow \mathbb{L}_p^+(\nu)$ maps $L_p(\nu)$ Schnorr tests uniformly to $L_p(\nu)$ Schnorr tests. For, by (\ref{prop:newlevy:1}), it sends non-negative lsc functions to non-negative lsc functions. And by Proposition~\ref{prop:abstract:1}(\ref{prop:abstract:1a}) and (\ref{prop:newlevy:3}) and Propositions~\ref{prop:contcomp},\ref{prop:suffforcont}, it sends $L_p^+(\nu)$ computable points to $L_p^+(\nu)$ computable points.

 For $p=1$, the map $\cdot^{\flat}:\mathbb{L}_p^+(\nu)\rightarrow \mathbb{L}_0^+(\nu)$ sends $L_p(\nu)$ Schnorr tests uniformly to $L_0(\nu)$ Schnorr tests. For, by (\ref{prop:newlevy:1}), it sends non-negative lsc functions to non-negative lsc functions. And by Proposition~\ref{prop:abstract:1}(\ref{prop:abstract:1b}) and (\ref{prop:newlevy:3}) and Propositions~\ref{prop:contcomp},\ref{prop:suffforcont} it sends $L_p^+(\nu)$ computable points to $L_0^+(\nu)$ computable points.

Now we work on the equivalence of (\ref{thm:abstractlevy:1})-(\ref{thm:abstractlevy:3}).

Suppose (\ref{thm:abstractlevy:1}); we show (\ref{thm:abstractlevy:2}). Suppose that $f$ is an $L_p(\nu)$ Schnorr test; we want to show that $f=\lim_n E[f\mid n]$ on $\mathsf{SR}^{\nu}$. Choose $f_s$ from $\mathcal{C}$ as in (\ref{prop:newlevy:4}). By definition, one has that $f_s\rightarrow f$ pointwise on $\mathsf{XR}^{\nu}$ and fast in $L_p(\nu)$ and is non-decreasing on $\mathsf{XR}^{\nu}$.  Let $g_s$ be the $L_p(\nu)$ Schnorr test $f\ominus f_s$.  Then $g_s\rightarrow 0$ pointwise on $\mathsf{XR}^{\nu}$ and is non-increasing on $\mathsf{XR}^{\nu}$ and $g_s\rightarrow 0$ fast in $L_p(\nu)$.

Suppose $p>1$ (resp. $p=1$). Since $f_s$ is finite on $\mathsf{XR}^{\nu}$, we have that $f= f_s +f\ominus f_s$ on $\mathsf{XR}^{\nu}$ and hence by (\ref{prop:newlevy:2c}) we have $E[f\mid n] = E[f_s \mid n]+ E[f\ominus f_s\mid n] \leq  E[f_s \mid n]+g_s^{\flat}$ on $\mathsf{XR}^{\nu}$. These are all $L_p(\nu)$ Schnorr tests (resp. except for $g_s^{\flat}$ which is an $L_0(\nu)$ Schnorr test), and so they are finite on $\mathsf{SR}^{\nu}$ and we hence have $E[f\mid n]-E[f_s \mid n]\leq g_s^{\flat}$ on $\mathsf{SR}^{\nu}$. Since $f_s\leq f$ on $\mathsf{XR}^{\nu}$, we have $E[f_s\mid n]\leq E[f\mid n]$ on $\mathsf{XR}^{\nu}$ by (\ref{prop:newlevy:2a}). Hence  $\left|E[f\mid n]-E[f_s \mid n]\right|\leq g_s^{\flat}$ on $\mathsf{SR}^{\nu}$. Since the maximal function maps into $L_p(\nu)$ (resp. $L_0(\nu)$) and has a computable modulus of uniform continuity (cf. Proposition~\ref{prop:abstract:1}), and since $g_s\rightarrow 0$ fast in $L_p(\nu)$ (resp. at a geometric rate in $L_0(\nu)$), one can compute a subsequence such that $g^{\flat}_{s(n)} \rightarrow 0$ fast in $L_p(\nu)$ (resp. in $L_0(\nu)$). By Proposition~\ref{prop:schnorrtestforlsc} one has that $g_{s(n)}^{\flat}\rightarrow 0$ pointwise on $\mathsf{SR}^{\nu}$. Let $x$ in $\mathsf{SR}^{\nu}$ and let $\epsilon>0$. Since $f, f_s, E[f\mid n], E[f_s\mid n]$ are $L_p(\nu)$ Schnorr tests and since $g_s^{\flat}$ is an $L_p(\nu)$ Schnorr test (resp. $L_0(\nu)$ Schnorr test), these values are all finite on the $\mathsf{SR}^{\nu}$ point~$x$. Choose $n_0\geq 0$ such that for all $n\geq n_0$ and $f(x)-f_n(x)<\frac{\epsilon}{3}$. Choose $n_1\geq n_0$ such that for all $n\geq n_1$ one has $g_{s(n)}^{\flat}(x)<\frac{\epsilon}{3}$. By the hypothesis on the $f_s$ coming from~$\mathcal{C}$, choose $n_2\geq n_1$ such that $\left|f_{s(n_1)}(x)-E[f_{s(n_1)}\mid n](x)\right|<\frac{\epsilon}{3}$ for all $n\geq n_2$. Hence for all $n\geq n_2$ one has that $\left|f(x)-E[f\mid n](x)\right|\leq
\left|f(x)-f_{s(n_1)}(x)\right|+\left|f_{s(n_1)}(x)-E[f_{s(n_1)}\mid n](x)\right|+\left|E[f_{s(n_1)}\mid n](x)-E[f\mid n](x)\right| <\frac{\epsilon}{3}+\frac{\epsilon}{3}+g_{s(n_1)}^{\flat}(x)<\epsilon$.

Note that the previous paragraph yields~(\ref{thm:abstractlevy:extra3}). For, Proposition~\ref{prop:consequences:selflocation}(\ref{prop:consequences:selflocation:1}) tells us that $x$ can weakly compute a modulus of convergence for $g_{s(n)}^{\flat}(x)\rightarrow 0$. And Proposition~\ref{prop:consequences:selflocation}(\ref{prop:consequences:selflocation:2}) tells us that $x$ can weakly compute a modulus of convergence for $f_n(x)\rightarrow f(x)$. And the extra hypothesis in~(\ref{thm:abstractlevy:extra3}) says that $x$ can weakly compute a modulus of convergence for $E[f_{s(n_1)}\mid n](x)\rightarrow f_{s(n_1)}(x)$.

The implication from (\ref{thm:abstractlevy:2}) to (\ref{thm:abstractlevy:3}) is trivial. 

Suppose (\ref{thm:abstractlevy:3}); we show (\ref{thm:abstractlevy:1}). Suppose that $x$ is a point satisfying (\ref{thm:abstractlevy:3}). We want to show that $x$ is in $\mathsf{SR}^{\nu}$. Let $f$ be an $L_p(\nu)$ Schnorr test. We want to show that $f(x)<\infty$. Suppose for reductio that $f(x)=\infty$. By hypothesis, $\lim_n E[f\mid n](x)$ exists and is finite. Choose $n_0\geq 0$ and rationals $p<q$ such that $E[f\mid n](x)<p$ for all $n\geq n_0$. Then $x$ is in the c.e. open $f^{-1}(q,\infty]$. Using a $\nu$-computable basis, choose c.e. open $U$ which is a subset of $f^{-1}(q,\infty]$ and which contains $x$ and which has computable $\nu$-measure. Let $g=q\cdot I_U$, which is an $L_p(\nu)$ Schnorr test. Since $g - q \leq 0$ everywhere, by  (\ref{thm:abstractlevy:3}), there is $n_1\geq n_0$ such that $q - p > \left|E[g\mid n](x)-q\right| = q - E[g \mid n](x)$ for all $n\geq n_1$, and thus $E[g \mid n](x) > p$ for all such $n$. Let $n\geq n_1$. Since $g\leq f$ everywhere, by (\ref{prop:newlevy:2a}), we have $E[g\mid n](x)\leq E[f\mid n](x)<p$, a contradiction.

For (\ref{thm:abstractlevy:extra2}), by using a $\nu$-computable basis it suffices to prove it for $U$ c.e. open with $\nu(U)$ computable. In this case, $f=I_U$ is an $L_p(\nu)$ Schnorr test. Then $f=\lim_n E[f\mid n]$ on $\mathsf{SR}^{\nu}$. By  (\ref{prop:newlevy:2a}), one has $0\leq E[f\mid n]\leq 1$ on $\mathsf{XR}^{\nu}$. For rational $\epsilon$ in the interval $(0,\frac{1}{2})$, let $V_{n,\epsilon}=E[f\mid n]^{-1}(1-\epsilon, \infty]$. This set is in $\mathscr{F}_n$ since $E[f\mid n]$ is a version of the conditional expectation of $f$ relative to $\mathscr{F}_n$. Further this set is c.e. open by the second paragraph of this proof. Then on $\mathsf{SR}^{\nu}$ one has that $U=\bigcap_{\epsilon\in \mathbb{Q}\cap (0,\frac{1}{2})} \bigcup_{n_0\geq 0} \bigcap_{n\geq n_0} V_{n,\epsilon}$.\footnote{This event is further in $\utilde{\Pi}^0_4$. When we decompose an arbitrary open as a union of c.e. opens with $\nu$-computable measure, we will get an event in $\utilde{\Sigma}^0_5$.} Hence, $U$ is equal on $\mathsf{SR}^{\nu}$ to an event $B^{\prime}$ in the $\sigma$-algebra generated by the union of the $\mathscr{F}_n$.\
\end{proof}

\section{Fundamental properties of effective disintegrations}\label{sec:ed-prop}

In this section, we develop the properties of effective disintegrations (cf. Definition~\ref{defn:eff:disin}, and for examples see Appendicies~\ref{sec:app:classical:disintegrations}-\ref{sec:app:effective:disintegrations}). In the next propositions, $\mathsf{XR}^{\nu}$ denotes a $\nu$-measure one subset of $\mathsf{KR}^{\nu}$, as in the definition of an effective disintegration. Further, throughout this section, the expression $\mathbb{E}_{\nu}[f\mid \mathscr{F}]$ is defined as in (\ref{eqn:disintegrate}) of \S\ref{sec:disintegration}, namely the version of conditional expectation coming from the effective disintegration.

\begin{prop}\label{prop:dis:1}
 Suppose $\rho: X\rightarrow \mathcal{M}^+(X)$ is an $\mathsf{XR}^{\nu}$ disintegration of $\mathscr{F}$. Then for lsc $f:X\rightarrow [0,\infty]$, the map $\mathbb{E}_{\nu}[f\mid \mathscr{F}]:X\rightarrow [0,\infty]$ is lsc, uniformly in $f$.
\end{prop}
\begin{proof}
By Definition~\ref{defn:eff:disin}(\ref{defn:eff:disin:3}) it suffices to show that for rational $r\geq 0$ we have 
\begin{equation*}
\int f \; d\rho_x >r \mbox{ iff } \exists \; q\in \mathbb{Q}^{>0} \; \big( \rho_x(X)>\frac{r}{q} \; \& \; \rho_x( f^{-1}(q, \infty])>0\big) 
\end{equation*}

Suppose that  $\int f \; d\rho_x >r$. Since $r\geq 0$, we have $\rho_x(X)> 0$. Choose rational $q>0$ in the interval $(\frac{r}{\rho_x(X)}, \frac{\int f \; d\rho_x}{\rho_x(X)})$. Then $\int f \; d\rho_x >\rho_x(X)\cdot q$ and $\rho_x(X)>\frac{r}{q}$.  Then $\rho_x(f^{-1}(q,\infty])>0$, since otherwise $0\leq f \leq q$ $\rho_x$-a.e. and hence  $\int f \; d\rho_x \leq q\cdot \rho_x(X)$. 

 Suppose that rational $q>0$ satisfies $\rho_x(X)>\frac{r}{q}$ and $\rho_x(f^{-1}(q,\infty])>0$. If $f$ is not $\rho_x$-integrable then trivially we have $\int f \; d\rho_x >r$. Hence suppose $f$ is $\rho_x$-integrable. Since $\rho_x(f^{-1}(q,\infty])>0$, choose  $\epsilon>0$ such that $\rho_x(f-q>\epsilon)>0$. Then $0<\rho_x(f-q>\epsilon)\leq \frac{1}{\epsilon} \int f -q \; d\rho_x$. Then $0<\int f -q \; d\rho_x$. Then $q\cdot \rho_x(X)<\int f \; d\rho_x$. Then $r<\int f \; d\rho_x$.
\end{proof}

\begin{prop}\label{prop:dis:2}
Suppose $\rho: X\rightarrow \mathcal{M}^+(X)$ is a $\mathsf{XR}^{\nu}$ disintegration of $\mathscr{F}$. Then for $f,g$ in $\mathbb{L}_1^+(\nu)$  the conditional expectation satisfies the following monotone linearity properties:
\begin{enumerate}[leftmargin=*]
\item \label{prop:newlevy:2a:again} If $f\leq g$ on $\mathsf{XR}^{\nu}$, then $\mathbb{E}_{\nu}[f\mid \mathscr{F}]\leq \mathbb{E}_{\nu}[g\mid \mathscr{F}]$ on $\mathsf{XR}^{\nu}$;
 \item \label{prop:newlevy:2b:again} If $c$ in $\mathbb{R}^{\geq 0}$ then $c \cdot \mathbb{E}_{\nu}[f\mid \mathscr{F}] = \mathbb{E}_{\nu}[c\cdot f\mid \mathscr{F}]$ everywhere.
\item \label{prop:newlevy:2c:again} $\mathbb{E}_{\nu}[f+g\mid \mathscr{F}] = \mathbb{E}_{\nu}[f\mid \mathscr{F}]+\mathbb{E}_{\nu}[g\mid \mathscr{F}]$ everywhere.
\end{enumerate}
\end{prop}
\begin{proof}
For (\ref{prop:newlevy:2a:again}), suppose that $f\leq g$ on $\mathsf{XR}^{\nu}$. Suppose that $x$ is in $\mathsf{XR}^{\nu}$. By  Definition~\ref{defn:eff:disin}(\ref{defn:eff:disin:2}), $\rho_x$ is in $\mathcal{P}(X)$ and $\rho_{x}([x]_{\mathscr{F}}\cap \mathsf{XR}^{\nu})=1$. Then $f\leq g$ on a $\rho_x$-measure one set, and hence $\int f(v) \; d\rho_{x}(v) \leq \int g(v) \; d\rho_{x}(v)$. For (\ref{prop:newlevy:2b:again})-(\ref{prop:newlevy:2c:again}), these just follow from the properties of the integral.
\end{proof}

The use of Definition~\ref{defn:eff:disin}(\ref{defn:eff:disin:2}) in the  proof of the previous proposition is typical, and henceforth we do not explicitly reiterate it as we go along.  

We stated the previous proposition for non-negative functions. For these functions, the conditional expectation $\mathbb{E}_{\nu}[f\mid \mathscr{F}](x)$ in (\ref{eqn:disintegrate}) is automatically defined for all points $x$ in $X$, even if it is infinite. However, $\mathbb{E}_{\nu}[f\mid \mathscr{F}](x)$ in (\ref{eqn:disintegrate})  is automatically defined and finite when $f$ is a simple function, and so the previous proposition holds for these functions as well. More generally, the previous proposition holds for functions which take negative values, provided that $\mathbb{E}_{\nu}[f\mid \mathscr{F}](x)$ is defined and finite on all points $x$ of $\mathsf{XR}^{\nu}$.

\begin{prop}\label{prop:dis:4.1}
Suppose $\rho: X\rightarrow \mathcal{M}^+(X)$ is a $\mathsf{XR}^{\nu}$ disintegration of $\mathscr{F}$. 

If $f$ in $\mathbb{L}^+_1(\nu)$ is equal on $\mathsf{XR}^{\nu}$ to a function which is $\mathscr{F}$-measurable, then one has $\mathbb{E}_{\nu}[f\mid \mathscr{F}] = f$ on $\mathsf{XR}^{\nu}$.
\end{prop}
\begin{proof}
By Proposition~\ref{prop:dis:2}(\ref{prop:newlevy:2a:again}), it suffices to consider functions in $\mathbb{L}^+_1(\nu)$ which are themselves $\mathscr{F}$-measurable (as opposed to being merely equal on $\mathsf{XR}^{\nu}$ to such a function).

Suppose that $x$ is in $\mathsf{XR}^{\nu}$. If $\mathbb{E}_{\nu}[f\mid \mathscr{F}](x)<f(x)$, then for some rationals $p,q$ we have $\mathbb{E}_{\nu}[f\mid \mathscr{F}](x)<p<q<f(x)$. Then $x$ is in the  event $f^{-1}(q,\infty]$ in $\mathscr{F}$. Then $[x]_{\mathscr{F}}\subseteq f^{-1}(q,\infty]$. Then $f\geq q$ for $\rho_x$-a.s. many values and hence $\int f \; d\rho_x \geq q$, a contradiction. The case of $\mathbb{E}_{\nu}[f\mid \mathscr{F}](x)>f(x)$ is similar.
\end{proof}

In the below proposition, we use the traditional names for the properties of conditional expectation.\footnote{E.g. \cite[88]{Williams1991aa}.}. Of course, the hypothesis of effective disintegrations in Definition~\ref{defn:eff:disin}(\ref{defn:eff:disin:1}) is that $\mathbb{E}_{\nu}[f\mid \mathscr{F}](x)=\int f \; d\rho_x$ is a version of conditional expectation. But in this proposition and several others in this section, what we are verifying is that they hold \emph{pointwise} on specifiable measure $\nu$-one subsets.
\begin{prop}\label{prop:bigprop}
Suppose $\rho: X\rightarrow \mathcal{M}^+(X)$ is a $\mathsf{XR}^{\nu}$ disintegration of $\mathscr{F}$.

\begin{enumerate}[leftmargin=*, align=left]
    \item \label{prop:cmct} (Conditional MCT). Suppose that $f_n, f$ are in $\mathbb{L}_1^+(\nu)$ and $0\leq f_n\leq f_{n+1}$ on $\mathsf{XR}^{\nu}$ and $\lim_n f_n=f$ on $\mathsf{XR}^{\nu}$. Then $\lim_n \mathbb{E}_{\nu}[f_n\mid \mathscr{F}] = \mathbb{E}_{\nu}[f \mid \mathscr{F}]$ on $\mathsf{XR}^{\nu}$.
    \item \label{prop:cdct}  (Conditional DCT) Suppose that $f_n, f, g$ are in $\mathbb{L}_1^+(\nu)$ and $\left|f_n\right|\leq g$ on $\mathsf{XR}^{\nu}$ and $\lim_n f_n=f$ on $\mathsf{XR}^{\nu}$. Then:
    \begin{enumerate}[leftmargin=*, align=left]
    \item \label{prop:cdcta} If $x$ in $\mathsf{XR}^{\nu}$ and $\mathbb{E}_{\nu}[g\mid\mathscr{F}](x)<\infty$ then $\lim_n \mathbb{E}_{\nu}[f_n\mid \mathscr{F}](x) = \mathbb{E}_{\nu}[f \mid \mathscr{F}](x)$. 
    \item If $\mathbb{E}_{\nu}[g\mid \mathscr{F}]<\infty$ on $\mathsf{XR}^{\nu}$, then $\lim_n \mathbb{E}_{\nu}[f_n\mid \mathscr{F}] = \mathbb{E}_{\nu}[f \mid \mathscr{F}]$ on $\mathsf{XR}^{\nu}$.
    \end{enumerate}
    \item \label{prop:ctowisk} (`Taking out what is known'). Suppose that $f,g$ in $\mathbb{L}_1^+(\nu)$, and suppose that $g$ is equal on $\mathsf{XR}^{\nu}$ to a $\mathscr{F}$-measurable function.  Then $\mathbb{E}_{\nu}[f\cdot g\mid \mathscr{F}]=g\cdot \mathbb{E}_{\nu}[f\mid \mathscr{F}]$ on $\mathsf{XR}^{\nu}$.
\end{enumerate}
\end{prop}

\begin{proof}
For (\ref{prop:cmct}), let $x$ in $\mathsf{XR}^{\nu}$. By hypothesis, $0\leq f_n\leq f_{n+1}$ for $\rho_x$-a.s. many points, and likewise  $\lim_n f_n=f$ for  $\rho_x$-a.s. many points. Then by MCT applied to $\rho_x$ we have  $\lim_n \int f_n \; d\rho_x = \int f \; d\rho_x$.

For (\ref{prop:cdct}), let $x$ in $\mathsf{XR}^{\nu}$ with $\mathbb{E}_{\nu}[g\mid \mathscr{F}](x)<\infty$. This means that $\int g \; d\rho_x<\infty$, and so $g$ is in $\mathbb{L}^+_1(\rho_x)$. By hypothesis, $\left| f_n\right|\leq g$ for $\rho_x$-a.s. many points, and likewise  $\lim_n f_n=f$ for  $\rho_x$-a.s. many points. Hence by the DCT applied to $\rho_x$, we have that $\lim_n \int f_n \; d\rho_x = \int f \; d\rho_x$.

For (\ref{prop:ctowisk}), by Proposition~\ref{prop:dis:2}(\ref{prop:newlevy:2a:again}), it suffices to prove it for $g$ which is itself $\mathscr{F}$-measurable. We show it by induction on complexity of $g$. 

Suppose $g=I_A$ where $A$ is $\mathscr{F}$-measurable.  If $x$ in $A$ then $[x]_{\mathscr{F}}\subseteq A$ and then $A$ is a $\rho_x$-measure one event and then it reduces to the observation that $\int_A f(v) \; d\rho_x(v)=\int f(v) \; d\rho_x(v)$. If $x$ is not in $A$ then $[x]_{\mathscr{F}}\subseteq X\setminus A$ and then $A$ is a $\rho_x$-measure zero event and then it reduces to the observation that $\int_A f(v) \; d\rho_x(v)=0$.

By Proposition~\ref{prop:dis:2}(\ref{prop:newlevy:2b:again})-(\ref{prop:newlevy:2c:again}), it extends to simple functions. By Conditional MCT it extends to all elements of $\mathbb{L}_1^+(\nu)$.
\end{proof}

Unlike the previous propositions, this proposition concerns Kurtz disintegrations:
\begin{prop}\label{prop:compcont}
 Suppose $\rho: X\rightarrow \mathcal{M}^+(X)$ is a Kurtz disintegration of $\mathscr{F}$. If $p\geq 1$ is computable, then $\mathbb{E}_{\nu}[\cdot\mid \mathscr{F}]:L_p(\nu)\rightarrow L_p(\nu)$ is computable continuous.
\end{prop}
\begin{proof}
By conditional Jensen, the function $m(\epsilon)=\epsilon$ is a computable modulus of uniform continuity. Hence by Proposition~\ref{prop:suffforcont}, it suffices to show that if $\varphi=\sum_{i=1}^n q_i \cdot I_{A_i}$ is an element of the countable dense set of $L_p(\nu)$, then $\mathbb{E}_{\nu}[\varphi\mid \mathscr{F}]$ is a computable point of $L_p(\nu)$. Since we can effectively separate $\varphi$ into positive and negative parts, it suffices by the linearity of conditional expectation (Proposition~\ref{prop:dis:2}) to consider the case where $q_i\geq 0$. We may assume further that the $A_i$ are pairwise disjoint, which like in the discussion at the beginning of \S\ref{sec:Lp} implies that $\varphi^p =\sum_{i=1}^n q_i^p \cdot I_{A_i}$. 

By Proposition~\ref{prop:krplusbasis}, let $U_i$ be a c.e. open which is equal to $A_i$ on $\mathsf{KR}^{\nu}$. Let $f=\sum_{i=1}^n q_i\cdot I_{U_i}$, so that likewise $f^p=\sum_{i=1}^n q_i\cdot I_{U_i}$ on $\mathsf{KR}^{\nu}$. Then $f=\varphi$ on $\mathsf{KR}^{\nu}$ and $f^p=\varphi^p$ on $\mathsf{KR}^{\nu}$. By Proposition~\ref{prop:dis:2} we have for $x$ in $\mathsf{KR}^{\nu}$:
\begin{align*}
& \mathbb{E}_{\nu}[f\mid \mathscr{F}](x) =  \sum_{i=1}^n q_i \int I_{U_i}(v) \; d\rho_x(v)=\sum_{i=1}^n q_i \cdot \rho_x(U_i), \hspace{15mm} \\
&  \mathbb{E}_{\nu}[f^p\mid \mathscr{F}](x) = \sum_{i=1}^n q_i^p \int I_{U_i}(v) \; d\rho_x(v)= \sum_{i=1}^n q_i^p \cdot \rho_x(U_i)
\end{align*}
By Definition~\ref{defn:eff:disin}(\ref{defn:eff:disin:3}), the functions $x\mapsto \rho_x(U_i)$ are lsc. Hence, by Proposition~\ref{prop:exptoflsc}, choose functions $\upsilon_{i,s}$ from the countable dense set of $L_p(\nu)$ that converge upward to $\rho_{\cdot}(U_i)$, in that $0\leq \upsilon_{i,s}\leq \upsilon_{i,s+1}$ everywhere and $\rho_{\cdot}(U_i)=\sup_s \upsilon_{i,s}$ everywhere. Let $g_s = \sum_{i=1}^n q_i \cdot \upsilon_{i,s}$ and $h_s=\sum_{i=1}^n q_i^p\cdot \upsilon_{i,s}$, which likewise converge upward to $\mathbb{E}_{\nu}[f\mid \mathscr{F}]$ and $\mathbb{E}_{\nu}[f^p\mid \mathscr{F}]$ respectively on $\mathsf{KR}^{\nu}$. 

Suppose $x$ is in $\mathsf{KR}^{\nu}$. Since we are working with a Kurtz disintegration, we then have that $\mathsf{KR}^{\nu}$ is a $\rho_x$-measure one set, and so $\rho_x(U_i)=\rho_x(A_i)$. Since the $A_i$ are pairwise disjoint, we then have $\sum_{i=1}^n \rho_x(U_i)= \sum_{i=1}^n \rho_x(A_i)\leq 1$, and so $\sum_{i=1}^n \rho_x(U_{i})-\upsilon_{i,s}(x)\leq 1$. Then for $x$ in $\mathsf{KR}^{\nu}$, by the convexity of the $p$-th power function applied with coefficients 
$\rho_x(U_{i})-\upsilon_{i,s}(x)$ and points $q_i$, we have:
\begin{align*}
 & (\mathbb{E}_{\nu}[f\mid \mathscr{F}](x) - g_s(x))^p =\bigg( \sum_{i=1}^n q_i \cdot (\rho_x(U_i)-\upsilon_{i,s}(x)) \bigg)^p \\
\leq &  \sum_{i=1}^n q_i^p \cdot (\rho_x(U_i)-\upsilon_{i,s}(x)) =\mathbb{E}_{\nu}[f^p\mid \mathscr{F}](x)-h_s(x)
\end{align*}
Since this estimate holds on the $\nu$-measure one set $\mathsf{KR}^{\nu}$, by taking expectations and then $p$-th roots, we have $\| \mathbb{E}_{\nu}[f\mid \mathscr{F}] - g_s\|_p\leq \big( \mathbb{E}_{\nu} f^p - \mathbb{E}_{\nu} h_s \big)^{\frac{1}{p}}$. Since the right-hand side is a computable value which goes to zero as $s$ goes to infinity (by MCT), we can compute a subsequence of the $g_s$ which is a witness to the $L_p(\nu)$ computability of $\mathbb{E}_{\nu}[f\mid \mathscr{F}]$. Since $f,\varphi$ are equal on $\mathsf{KR}^{\nu}$, we have $\mathbb{E}_{\nu}[\varphi\mid \mathscr{F}]$ is also a computable point of $L_p(\nu)$.
\end{proof}

The previous proposition has the following elementary consequence:
\begin{prop}\label{prop:preservetests}
 Suppose $\rho: X\rightarrow \mathcal{M}^+(X)$ is a Kurtz disintegration of $\mathscr{F}$. Suppose $p\geq 1$ is computable. 
 \begin{enumerate}[leftmargin=*]
     \item \label{prop:preservetests:1} If $f$ is an $L_p(\nu)$ Martin-L\"of test, then $\mathbb{E}_{\nu}[f\mid \mathscr{F}]$ is an $L_p(\nu)$ Martin-L\"of test and $\mathbb{E}_{\nu}[f^p\mid \mathscr{F}]$ is an $L_1(\nu)$ Martin-L\"of test.
     \item \label{prop:preservetests:2} If $f$ is an $L_p(\nu)$ Schnorr test, then $\mathbb{E}_{\nu}[f\mid \mathscr{F}]$ is an $L_p(\nu)$ Schnorr test and $\mathbb{E}_{\nu}[f^p\mid \mathscr{F}]$ is an $L_1(\nu)$ Schnorr test.
 \end{enumerate}
\end{prop}
\begin{proof}
For (\ref{prop:preservetests:1}), suppose that $f$ is an $L_p(\nu)$ Martin-L\"of test. Then $f,f^p$ are non-negative lsc, and so by Proposition~\ref{prop:dis:1} one has that $\mathbb{E}_{\nu}[f\mid \mathscr{F}]$ and $\mathbb{E}_{\nu}[f^p\mid \mathscr{F}]$ are non-negative lsc. By conditional Jensen, $\|\mathbb{E}_{\nu}[f\mid \mathscr{F}]\|_p\leq \|f\|_p<\infty$, and likewise $\|\mathbb{E}_{\nu}[f^p\mid \mathscr{F}]\|_1\leq \|f^p\|_1=\|f\|^p_p<\infty$. 

For (\ref{prop:preservetests:2}), suppose that $f$ is an $L_p(\nu)$ Schnorr test. Since $\|f^p\|_1=\|f\|^p_p$, one has that $f^p$ is an $L_1(\nu)$ Schnorr test. By Proposition~\ref{prop:exptoflsc}, $f$ is a computable point of $L_p(\nu)$, and $f^p$ is a computable point of $L_1(\nu)$. By the previous proposition $\mathbb{E}_{\nu}[f\mid \mathscr{F}]$ is a computable point of $L_p(\nu)$, and $\mathbb{E}_{\nu}[f^p\mid \mathscr{F}]$ is a computable point of $L_1(\nu)$.
\end{proof}

This next proposition seems specific to Schnorr tests and $\mathsf{SR}^{\nu}$:
\begin{prop}\label{prop:tower} (Tower) 
Suppose that $\mathscr{H}\subseteq \mathscr{G}$ are two effective $\sigma$-algebras, each of which has a Kurtz disintegration. Then for every $L_1(\nu)$ Schnorr test $f$, one has that $\mathbb{E}_{\nu}[ \mathbb{E}_{\nu}[f\mid \mathscr{G}] \mid \mathscr{H}] = \mathbb{E}_{\nu}[f \mid \mathscr{H}]$ on $\mathsf{SR}^{\nu}$.
\end{prop}
\begin{proof}
By the previous proposition, one has that the two functions $g:=\mathbb{E}_{\nu}[ \mathbb{E}_{\nu}[f\mid \mathscr{G}] \mid \mathscr{H}]$ and $h:=\mathbb{E}_{\nu}[f \mid \mathscr{H}]$ are $L_1(\nu)$ Schnorr tests. Suppose that they are not equal on $x$ in $\mathsf{SR}^{\nu}$. Then, without loss of generality, there are rationals $a,b,c$ with $g(x)<a<b<c<h(x)$. By Lemma~\ref{lem:usingst}, there is a computable real $\epsilon$ in the interval $(b,c)$ with $(h\#\nu)(\epsilon,\infty]$ computable. Since $h$ is lsc, the set $U:=h^{-1}(\epsilon, \infty]$ is c.e. open, and it has computable measure. By the same lemma, there is a computable real $\delta$ in the interval $(a,b)$ with $(g\#\nu)(\delta,\infty]$ computable, so that $(g\#\nu)[0,\delta]$ is likewise computable. Since $g$ is lsc, the set $C:=g^{-1}[0,\delta]$ is effectively closed. Since it has computable $\nu$-measure, by Proposition~\ref{prop:hoyruprojas:compbasis:cor}, there is a a decreasing sequence of c.e. opens $V_n\supseteq C$ with $\nu(V_n)$ uniformly computable and $\nu(V_n\setminus C)<2^{-n}$.  We then claim that $\nu(U\cap C)>0$. For, suppose not. Then $0=\nu(U\cap C)=\lim_i \nu(U\cap V_i)$. Since $\nu(U\cap V_i)$ is computable by Proposition~\ref{prop:union:mucompbasis}(\ref{prop:union:mucompbasis:2.2}), we can then compute a subsequence $U\cap V_{n(i)}$ with $\nu(U\cap V_{n(i)})\leq 2^{-i}$, so that $\sum_i I_{U\cap V_{n(i)}}$ is an $L_1(\nu)$ Schnorr test. But since $x$ in $\mathsf{SR}^{\nu}$ and $x$ in $U\cap C$ by construction, we have a contradiction. Hence indeed $\nu(U\cap C)>0$. Since $g,h$ are by definition $\mathscr{H}$-measurable, we have that $U\cap C$ is also $\mathscr{H}$-measurable and hence $\mathscr{G}$-measurable. Then one has the following identities by the definition of conditional expectation (these identities being the classical proof of the tower property):
\begin{equation*}
\int_{U\cap C} \mathbb{E}_{\nu}[f\mid \mathscr{H}] \; d\nu = \int_{U\cap C} f \; d\nu = \int_{U\cap C} \mathbb{E}_{\nu}[f\mid \mathscr{G}] \; d\nu =\int_{U\cap C} \mathbb{E}_{\nu}[ \mathbb{E}_{\nu}[f\mid \mathscr{G}] \mid \mathscr{H}] \;d\nu
\end{equation*}
But these identities give us the below identity, where the remaining inequalities follow from the definitions of $g,h,U,C$:
\begin{equation*}
\epsilon \cdot \nu(U\cap C) \leq \int_{U\cap C} h =\int_{U\cap C} g \leq \delta\cdot \nu(U\cap C)
\end{equation*}
But since $\nu(U\cap C)>0$, we then have that $\epsilon\leq \delta$, contrary to construction.

\end{proof}

\begin{prop}\label{prop:roleindepenedence} (The r\^ole of independence) 
Suppose $\rho: X\rightarrow \mathcal{M}^+(X)$ is a Kurtz disintegration of $\mathscr{F}$. 

If $f$ in $\mathbb{L}_1^+(\nu)$ is independent of $\mathscr{F}$, then $\mathbb{E}_{\nu}[f\mid \mathscr{F}]\leq \mathbb{E}_{\nu}[f]$ everywhere.

If $f$ is an $L_1(\nu)$ Martin-L\"of test $f$ independent of $\mathscr{F}$, then $\mathbb{E}_{\nu}[f\mid \mathscr{F}]= \mathbb{E}_{\nu}[f]$ on $\mathsf{KR}^{\nu}$.
\end{prop}
\begin{proof}
Let us abbreviate $g=\mathbb{E}_{\nu}[f\mid \mathscr{F}]$.

First suppose that $f$ in $\mathbb{L}_1^+(\nu)$ is independent of $\mathscr{F}$. Let $x$ in $X$ be arbitrary. Suppose that $g(x)>\mathbb{E}_{\nu} f$. Choose rational $a$ with $g(x)>a>\mathbb{E}_{\nu} f$. Let $B=g^{-1}(a,\infty]$, which is in $\mathscr{F}$. Then we have the following, where the first step is independence: $\mathbb{E}_{\nu}[I_B \cdot f] = \nu(B)\cdot \mathbb{E}_{\nu}[f] <a \cdot \nu(B) \leq\int_B g \; d\nu = \mathbb{E}_{\nu}[I_B \cdot f]$.

Second suppose that $f$ is an $L_1(\nu)$ Martin-L\"of test, $f$ independent of $\mathscr{F}$. Then $g$ is non-negative lsc by Proposition~\ref{prop:dis:1}. Suppose $x$ is in $\mathsf{KR}^{\nu}$. Suppose $g(x)<\mathbb{E}_{\nu} f$. Choose rational $a>0$ with  $g(x)<a<\mathbb{E}_{\nu} f$. Let $C=g^{-1}[0,a]$, which is effectively closed and in $\mathscr{F}$. Since it contains the $\mathsf{KR}^{\nu}$ point $x$, we have that $\nu(C)>0$. Then we have the following, where the first step is from independence: $\nu(C)\cdot \mathbb{E}_{\nu} f =\mathbb{E}_{\nu} [I_C\cdot f]=\mathbb{E}_{\nu} [I_C\cdot g] \leq a\cdot \nu(C)$. Since $\nu(C)>0$ we have $\mathbb{E}_{\nu} f\leq a$, a contradiction.
\end{proof}

Now we turn towards effective properties of the maximal function (cf. \S\ref{sec:maximal:classical} for the classical properties). Recall that almost-full was defined in Definition~\ref{defn:core}(\ref{defn:core:5}).

\begin{prop}\label{prop:dis:3}
Let $\mathscr{F}_n$ be an almost-full effective filtration equipped with Kurtz disintegrations. Let $f^{\flat}(x)=\sup_n \mathbb{E}_{\nu}[f\mid \mathscr{F}_n](x)$ be the associated version of the maximal function. 

For $p>1$, the maximal function is computable continuous from $L_p(\nu)$ to $L_p(\nu)$. 

For $p= 1$, the maximal function is computable continuous from $L_p(\nu)$ to $L_0(\nu)$.

For $p\geq 1$, the maximal function sends the countable dense set of $L_p^+(\nu)$ (resp. $L_p(\nu)$) uniformly to computable points of $L_p^+(\nu)$ (resp. $L_p(\nu)$).
\end{prop}
\begin{proof}
First suppose $p>1$. Let $f=\sum_{i=1}^k q_i \cdot I_{A_i}$, where $q_i$ is rational and $A_k$ comes from the algebra generated by a $\nu$-computable basis. Without loss of generality, $q_i\neq 0$ and the $A_k$ are pairwise disjoint. By almost-fullness of the filtration and Proposition~\ref{prop:krplusbasis}, for each $1\leq i\leq k$, let $A_i=\bigcup_s U_{i,s}$ on $\mathsf{KR}^{\nu}$, where $U_{i,s}$ is a c.e. open equal on $\mathsf{KR}^{\nu}$ to an element from the sequence which generates $\mathscr{F}_{g(i,s)}$, where $g$ is a computable function. By replacing $U_{i,s}$ by $\bigcup_{t\leq s} U_{i,s}$, we may assume that $U_{i,s}$ and $g(i,s)$ are non-decreasing in $s$, for each $i\geq 0$. Let $f_s=\sum_{i=1}^k q_i \cdot I_{U_{i,s}}$. The $A_i$ and $U_{i,s}$ come from the algebra generated by a $\nu$-computable basis, and so for each $1\leq i\leq k$, we can compute a subsequence $U_{i, s(n)}$ such that $\nu(A_i\setminus U_{i,s(n)})<\big(\frac{1}{\max_{1\leq j\leq k} \left|q_j\right|}\cdot \frac{1}{k}\cdot \frac{p-1}{p}\cdot 2^{-n}\big)^p$. By Doob's Maximal Inequality (Lemma~\ref{lem:doob}), we have
\begin{equation*}
\|f^{\flat}-f_{s(n)}^{\flat}\|_p\leq \|(f-f_{s(n)})^{\flat}\|_p\leq \frac{p}{p-1} \cdot \|f-f_{s(n)}\|_p \leq \sum_{i=1}^k \left|q_i\right| \cdot \frac{p}{p-1} \cdot \nu(A_i\setminus U_{i,s(n)})^{\frac{1}{p}}
\end{equation*}
which is $<2^{-n}$. Hence it remains to show that $f_{s(n)}^{\flat}$ is uniformly a computable point of $L_p(\nu)$. Since for all $t\geq \max_{1\leq i\leq k} g(i,s(n))$, the c.e. open $U_{i,s(n)}$ is equal on $\mathsf{KR}^{\nu}$ to an element of $\mathscr{F}_t$, one has that $f_{s(n)}=\mathbb{E}_{\nu}[f_{s(n)}\mid \mathscr{F}_t]$ on $\mathsf{KR}^{\nu}$ by Proposition~\ref{prop:dis:4.1}. Hence, $f_{s(n)}^{\flat}=\sup_{t\leq \max_{1\leq i\leq k} g(i,s(n))} \mathbb{E}[f_{s(n)}\mid \mathscr{F}_t]$ on $\mathsf{KR}^{\nu}$. And the latter is a computable element of $L_p(\nu)$ since it is a finite max of conditional expectations which are computable elements of $L_p(\nu)$ by Proposition~\ref{prop:compcont}.

For $p=1$, one just appeals to the fact that $L_1(\nu)$ and $L_q(\nu)$ for computable $q>1$ share a common dense set and that $L_q(\nu)$ computably embeds in $L_1(\nu)$.

Then computability continuity follow from Proposition~\ref{prop:suffforcont} and Proposition~\ref{prop:abstract:1}.
\end{proof}

\begin{prop}\label{prop:dis:4}
Let $\mathscr{F}_n$ be an almost-full effective filtration equipped with Kurtz disintegrations.

For every $L_p(\nu)$ Schnorr test $f$, we can compute an index for a sequence of $L_p(\nu)$ Schnorr tests $g_s$ such that $g_s\leq g_{s+1}$ on $\mathsf{KR}^{\nu}$ and $f=\sup_s g_s$ on $\mathsf{KR}^{\nu}$ and $g_s\rightarrow f$ fast in $L_p(\nu)$ and $g_s=\lim_n \mathbb{E}_{\nu}[g_s\mid \mathscr{F}_n]$ on $\mathsf{KR}^{\nu}$. Indeed, there is computable function $n(\cdot)$ such that 
\begin{equation}\label{prop:exptoflsc:fsexpressv3:5}
g_s= \mathbb{E}_{\nu}[g_s\mid \mathscr{F}_{m}] \mbox{ on } \mathsf{KR}^{\nu} \mbox{ for all } m\geq n(s)
\end{equation}

Further, we can compute an index for a non-negative usc function $h_s$ such that $g_s, h_s$ are equal on $\mathsf{KR}^{\nu}$.

Finally, we can compute from $s$ an index for a non-negative rational which bounds $\left|g_s\right|$ on $\mathsf{KR}^{\nu}$.
\end{prop}
\begin{proof}
Enumerate $\mathbb{Q}\cap [0,\infty)$ as $q_0, q_1, \ldots$. For each $n\geq 0$, one has that $f^{-1}(q_n,\infty]$ is uniformly c.e. open. Since the filtration is almost-full, by using Proposition~\ref{prop:krplusbasis} there is a computable sequence $U_{n,j}$ of c.e. opens such that $f^{-1}(q_n, \infty] = \bigcup_j U_{n,j}$ on  $\mathsf{KR}^{\nu}$, where the $U_{n,j}$ are equal on $\mathsf{KR}^{\nu}$ to events from the sequence which generates $\mathscr{F}_m$. Then define
\begin{equation}\label{prop:exptoflsc:fsexpressv3}
f_s(x) = \max\{0, q_n: n,j\leq s, x\in U_{n,j}\}
\end{equation}
As in the proof of Proposition~\ref{prop:exptoflsc}, this is a rational-valued step function, whose events are equal on $\mathsf{KR}^{\nu}$ to events from the sequence which generates $\mathscr{F}_{n(s)}$, where $n(\cdot)$ is a computable function. Then by Proposition~\ref{prop:dis:4.1} one has that
\begin{equation}\label{prop:exptoflsc:fsexpressv3:52}
f_s= \mathbb{E}_{\nu}[f_s\mid \mathscr{F}_{m}] \mbox{ on } \mathsf{KR}^{\nu} \mbox{ for all } m\geq n(s)
\end{equation}

Further from (\ref{prop:exptoflsc:fsexpressv3}) one sees that $f_s\leq f_{s+1}$ everywhere since the sum over which we taking the maximum grows in $s$. Further, one has $f_s\leq f$ on $\mathsf{KR}^{\nu}$ since if we had $f_s(x)>f(x)$ for $x$ in $\mathsf{KR}^{\nu}$, then $f_s(x)=q_n$ for some $n,j\leq s$ with $x$ in $U_{n,j}$. But since $U_{n,j}\subseteq f^{-1}(q_n,\infty]$ on $\mathsf{KR}^{\nu}$, we then have $f(x)>q_n$. Finally, one has $\sup_s f_s=f$ on $\mathsf{KR}^{\nu}$, since if not we would have $\sup_s f_s(x) < q_n < f(x)$ for some $x$ in $\mathsf{KR}^{\nu}$ and some $n$ and hence $x$ would be in $f^{-1}(q_n,\infty]= \bigcup_j U_{n,j}$ and so $x$ would be in $U_{n,j}$ for some $j$ and hence for $s\geq j$ one would have that $f_s(x)\geq q_n$ by definition in (\ref{prop:exptoflsc:fsexpressv3}).

We can pass to a subsequence of $f_s$ which goes to $f$ fast in $L_p(\nu)$ as in the proof of Proposition~\ref{prop:exptoflsc}.

Since $f_s$ is formed from events which are equal on $\mathsf{KR}^{\nu}$ to events coming from a $\nu$-computable basis, by Proposition~\ref{prop:shiftbasistolpst} there is non-negative lsc $g_s$ and non-negative usc $h_s$ such that $f_s, g_s, h_s$ are equal on $\mathsf{KR}^{\nu}$. Since $f_s, g_s$ are equal on $\mathsf{KR}^{\nu}$, we can use Proposition~\ref{prop:dis:2}(\ref{prop:newlevy:2a:again}) to infer from (\ref{prop:exptoflsc:fsexpressv3:52}) to (\ref{prop:exptoflsc:fsexpressv3:5}). 
\end{proof}

\section{Proof of Theorems~\ref{thm:newnewlevy}-\ref{thm:newnewlevyrates}}\label{sec:proofmain}

First we prove Theorem~\ref{thm:newnewlevy}:

\begin{proof}
The results of the previous section show that the conditions of Theorem~\ref{thm:newlevy} are satisfied:
\begin{itemize}[leftmargin=*]
\item Condition~(\ref{prop:newlevy:1}) is Proposition~\ref{prop:dis:1}.
\item Condition~(\ref{prop:newlevy:2}) is Proposition~\ref{prop:dis:2}.
\item Condition~(\ref{prop:newlevy:3}) is Propositions~\ref{prop:compcont}, \ref{prop:dis:3}.
\item Condition~(\ref{prop:newlevy:4}) is Proposition~\ref{prop:dis:4}.
\end{itemize}

Finally, we argue from Theorem~\ref{thm:newnewlevy}(\ref{levy:thm:3}) to Theorem~\ref{thm:newnewlevy}(\ref{levy:thm:1}). Suppose Theorem~\ref{thm:newnewlevy}(\ref{levy:thm:3}) is satisfied.  We want to show that $x$ is in $\mathsf{SR}^{\nu}$. Let $f$ be an $L_p(\nu)$ Schnorr test. We want to show that $f(x)<\infty$. Suppose not. Since by hypothesis $\lim_n \mathbb{E}_{\nu}[f\mid \mathscr{F}_n](x)$ exists, there are rationals $b,a$ and there is $n_0\geq 0$ such that $f(x)>b>a>\mathbb{E}_{\nu}[f\mid \mathscr{F}_n](x)$ for all $n\geq n_0$. Then $x$ is in the c.e. open $f^{-1}(b,\infty]$. Since the filtration is almost-full and $x$ is in $\mathsf{KR}^{\nu}$, there is $n_1\geq n_0$ and an event $A$ from $\mathscr{F}_{n_1}$ such that $x$ is in $A$ and $A\cap \mathsf{KR}^{\nu}\subseteq f^{-1}(b,\infty]$. Then $[x]_{\mathscr{F}_{n_1}}\subseteq A$, and hence $[x]_{\mathscr{F}_{n_1}}\cap \mathsf{KR}^{\nu}\subseteq A\cap \mathsf{KR}^{\nu}\subseteq f^{-1}(b,\infty]$. Hence $f^{-1}(b,\infty]$ is a $\rho_x^{(n_1)}$-measure one event, and hence $\mathbb{E}_{\nu}[f\mid \mathscr{F}_{n_1}](x)=\int f \; d\rho_x ^{(n_1)}\geq b>a$, a contradiction.

\end{proof}

Now we turn to Theorem~\ref{thm:newnewlevyrates}:
\begin{proof}
For Theorem~\ref{thm:newnewlevyrates}(\ref{levy:extra:5}), one appeals to Theorem~\ref{thm:newlevy}(\ref{thm:abstractlevy:extra3}), along with Proposition~\ref{prop:dis:4}.

For Theorem~\ref{thm:newnewlevyrates}(\ref{levy:extra:6}) suppose that $x$ in $\mathsf{SR}^{\nu}$ is of computably dominated degree. By the previous paragraph, we have that $x$ weakly computes a modulus $m:\mathbb{Q}^{>0}\rightarrow \mathbb{N}$ for the convergence $\mathbb{E}[f\mid \mathscr{F}_n](x)\rightarrow f(x)$. Likewise $m^{\prime}:\mathbb{N}\rightarrow \mathbb{N}$ defined by $m^{\prime}(i)=m(2^{-i})$ is computable from $m$. Since $x$ is of computably dominated degree, we have that at least one of these $m^{\prime}$ is dominated by some computable function, call it $m^{\dag}$, so that that past some point, call it $i_0$, we have $m^{\dag}\geq m^{\prime}$. Let $\epsilon>0$ be rational. Compute $j\geq i_0$ such that $2^{-j}<\epsilon$. Let $n\geq m^{\dag}(j)
$. Then $n\geq m^{\dag}(j)\geq  m^{\prime}(j) = m(2^{-j})$. Then $\left|\mathbb{E}[f\mid \mathscr{F}_n](x)- f(x)\right|<2^{-j}<\epsilon$.
\end{proof}

\section{Proof of Theorems~\ref{thm:drlevy}-\ref{thm:drlevy:drates}}\label{sec:thm:dnesity:proof}

We begin with an elementary proposition.
\begin{prop}\label{prop:limitsinductive}
Suppose $\mu_n, \mu$ in $\mathcal{P}(X)$ such that for every c.e. open $U$ one has $\lim_n \mu_n(U)=\mu(U)$. 
\begin{enumerate}[leftmargin=*]
    \item\label{prop:limitsinductive:1} For every event $A$ in the algebra generated by the c.e. opens, one has $\lim_n \mu_n(A)=\mu(A)$.
    \item\label{prop:limitsinductive:2} For every simple function $f$ generated from events in this algebra, one has $\lim_n \int f \; \mu_n = \int f \; \mu$.
    \item\label{prop:limitsinductive:3} For every lsc $f:X\rightarrow [0,\infty]$ which is in each of $\mathbb{L}^+_1(\mu_n), \mathbb{L}^+_1(\mu)$, one has $\int f \; d\mu\leq \liminf_n \int f\; d\mu_n$.
\end{enumerate}
\end{prop}
This proposition too illustrates that convergence to the truth is a strengthening of weak convergence, since in (\ref{prop:limitsinductive:3}) there is no boundedness constraint on the lsc function.
\begin{proof}
For (\ref{prop:limitsinductive:1}), every event in this algebra can be written as a finite disjoint union of sets of the form $U_1\cap \cdots \cap U_m \cap V_1^c \cap \cdots \cap V_n^c$, where $U_i, V_i$ are c.e. opens. Let $U=U_1\cap \cdots \cap U_m$ and $V=V_1\cup \cdots \cup V_n$, so that the set has the form $U\setminus V$. Since $\mu_n, \mu$ are in $\mathcal{P}(X)$ and since $U\cap V$ is c.e. open as well, we have that 
\begin{equation*}
\lim_n \mu_n(U\setminus V) = \lim_n \mu_n(U)-\lim_n  \mu_n(U\cap V) = \mu(U)-\mu(U\cap V) = \mu(U\setminus V)
\end{equation*}

For (\ref{prop:limitsinductive:2}), one just applies the previous item and the properties of the integral. 

For (\ref{prop:limitsinductive:3}), suppose not. Choose rational $q$ such that $\int f \; d\mu>q> \liminf_n \int f\; d\mu_n$. Let $f_s$ be the approximation to $f$ as in Proposition~\ref{prop:exptoflsc}. Then by the MCT applied in $\mathbb{L}^+_1(\mu)$, there is $s\geq 0$ such that $\int f_s \; d\mu>q$. By (\ref{prop:limitsinductive:2}), there is $n_0\geq 0$ such that for all $n\geq n_0$ one has that $\int f_s \; d\mu_n>q$. But this contradicts that $q>\liminf_n \int f\; d\mu_n$.
\end{proof}

Now we prove Theorem~\ref{thm:drlevy}:
\begin{proof}
Suppose Theorem~\ref{thm:drlevy}(\ref{dr:thm:1}); we show Theorem~\ref{thm:drlevy}(\ref{dr:thm:2}). Since $x$ is in $\mathsf{MLR}^{\nu}$, one has that $\rho_x^{(n)}$ is in $\mathcal{P}(X)$. Suppose that $f$ is an $L_p(\nu)$ Martin-L\"of test. Since $x$ is in $\mathsf{MLR}^{\nu}$, one has $f(x)<\infty$. By Proposition~\ref{prop:preservetests}(\ref{prop:preservetests:1}), the functions $\mathbb{E}_{\nu}[f\mid \mathscr{F}_n]$ are $L_p(\nu)$ Martin-L\"of tests as well, and hence $\mathbb{E}_{\nu}[f\mid \mathscr{F}_n](x)<\infty$, which is just to say that $\int f\; d\rho^{(n)}_x<\infty$. By  Proposition~\ref{prop:limitsinductive}(\ref{prop:limitsinductive:3}) applied to $\rho^{(n)}_x, \delta_x$ one has that $\int f \; d\delta_x \leq \liminf_n \int f\; d\rho^{(n)}_x$, which is just to say that $f(x)\leq \liminf_n \mathbb{E}_{\nu}[f\mid \mathscr{F}_n](x)$. Hence, it remains to show that $\limsup_n \mathbb{E}_{\nu}[f\mid \mathscr{F}_n](x)\leq f(x)$. Suppose not. For reductio, suppose there are rational $a,b$ with $f(x)<a<b<\limsup_n \mathbb{E}_{\nu}[f\mid \mathscr{F}_n](x)$.

Let $U=f^{-1}(a,\infty]$ and $C=f^{-1}[0,a]$, so that $U$ is c.e. open and $C$ is effectively closed. Since $x$ is in $\mathsf{DR}_{\rho}^{\nu}$ and $x$ is not in $U$, we have $\rho^{(n)}_x(U)\rightarrow 0$. 

By definition of $C$, we have for all $n\geq 0$ that $\int_C f \; d\rho^{(n)}_x \leq a$. For any $n\geq 0$ such that $\mathbb{E}_{\nu}[f\mid \mathscr{F}_n](x)>b$, we then have $\int_C f \; d\rho^{(n)}_x \leq a<b<\int f d\rho^{(n)}_x$, so that $\int_U f \; d\rho^{(n)}_x>b-a$. Hence, our reductio hypothesis gives that there are infinitely many $n\geq 0$ with $\int_U f  \;d\rho^{(n)}_x>b-a$.

By Proposition~\ref{prop:preservetests}(\ref{prop:preservetests:1}), we have that $\mathbb{E}_{\nu}[f^p\mid \mathscr{F}_n]$ are $L_1(\nu)$ Martin-L\"of tests. Hence its maximal function $\sup_n \mathbb{E}_{\nu}[f^p\mid \mathscr{F}_n]$ is also non-negative lsc. Let $U_k = \{y\in X: \sup_n \mathbb{E}_{\nu}[f^p\mid \mathscr{F}_n](y)>2^k\}$, which is c.e. open. By Doob's Submartingale Inequality, one has that $\nu(U_k)$ is $\leq$ the following:
\begin{equation*}
\lim_m \nu(\{y\in X: \sup_{n\leq m} \mathbb{E}_{\nu}[f^p\mid \mathscr{F}_n](y)>2^k\}) \leq \lim_m 2^{-k} \int \mathbb{E}_{\nu}[f^p\mid \mathscr{F}_m] \; d\nu = 2^{-k}  \|f\|_p^p
\end{equation*}
Hence $g=\sum_k I_{U_k}$ is an $L_1(\nu)$ Martin-L\"of test, and hence there is constant $K>0$ such that $\sup_n \mathbb{E}_{\nu}[f^p\mid \mathscr{F}_n](x)<K$. 

Then for all $n\geq 0$ we have $\int f^p \; d\rho^{(n)}_x<K$, so that $f$ is in $L_p(\rho^{(n)}_x)$ with $\|f\|_{L_p(\rho^{(n)}_x)}<K^{\frac{1}{p}}$.  Let $q$ be the conjugate exponent to $p$. Then, for each $n\geq 0$, we have by H\"older with respect to $\rho^{(n)}_x$ that:
\begin{equation}\label{eqn:iamholder}
\int_U f \; d\rho^{(n)}_x = \| f\cdot I_U\|_{L_1(\rho^{(n)}_x)}\leq \|f\|_{L_p(\rho^{(n)}_x)} \cdot \|I_U\|_{L_q(\rho^{(n)}_x)} \leq K^{\frac{1}{p}} \cdot (\rho^{(n)}_x(U))^{\frac{1}{q}}
\end{equation}
Since $\rho^{(n)}_x(U)\rightarrow 0$, we have that $\int_U f \; d\rho^{(n)}_x\rightarrow 0$, contradicting the previous conclusion from the reductio hypothesis. 

The implication from (\ref{dr:thm:2}) to (\ref{dr:thm:3}) is trivial.

The argument from (\ref{dr:thm:3}) to (\ref{dr:thm:1}) is nearly identical to the proof in \S\ref{sec:proofmain} from Theorem~\ref{thm:newnewlevy}(\ref{levy:thm:3}) to Theorem~\ref{thm:newnewlevy}(\ref{levy:thm:1}): one just replaces $\mathsf{SR}^{\nu}$ with $\mathsf{MLR}^{\nu}$ and replaces $L_p(\nu)$ Schnorr tests with $L_p(\nu)$ Martin-L\"of tests.
\end{proof}

Now we prove Theorem~\ref{thm:drlevy:drates}:
\begin{proof}
We work in Cantor space with the uniform measure $\nu$, the effective full filtration $\mathscr{F}_n$ of the algebra of events generated by the length $n$ strings, and with the effective disintegration $\rho^{(n)}_{\omega} = \nu(\cdot \mid [\omega\upharpoonright n])$. Then $\mathbb{E}[f\mid \mathscr{F}_n](\omega) = \frac{1}{\nu([\omega\upharpoonright n])} \int_{[\omega\upharpoonright n]} f \; d\nu$, and likewise $\mathbb{E}[I_A\mid \mathscr{F}_n](\omega) = \nu(A\mid [\omega\upharpoonright n])$. (This is just Example~\ref{ex:refinedpartitionconcrete} for Cantor space and uniform measure).

We show that for $\omega$ in $\mathsf{DR}_{\rho}^{\nu}$ there is c.e. open $U$ with $0<\nu(U)<1$ such that $\omega$ is not in $U$ and the convergence $\nu(U\mid [\omega\upharpoonright n])\rightarrow 0$ does not have a computable rate.

(Since the example involves an indicator function, we have that $I_U$ is in $L_p(\nu)$ for all $p\geq 1$ computable).

Let $k\geq 0$. Let $K$ be the halting set $\{e: \varphi_e(e)\ddownarrow\}$. Enumerate it as $e_0, e_1, \ldots$, where the map $n\mapsto e_n$ is injective.

Define $c_0=0$ and $c_{n+1}=\max(\varphi_{e_n}(e_n), c_n)+1$. 

For $k,n\geq 0$, define clopen $U_{k,n}=\{\omega: \forall \; i\in [c_{n+1}, c_{n+1}+e_n+k+1) \; \omega(i)=0\}$ and define the c.e. open $U_k=\bigcup_n U_{k,n}$. Then $\nu(U_{k,n})=2^{-(e_n+k+1)}$ and $0<\nu(U_k)\leq \sum_n \nu(U_{k,n})<2^{-k}$. Since $U_{k,n}$ just makes decisions on bits $\geq c_{n+1}$, it is independent of all bits $<c_{n+1}$ (since $\nu$ is uniform measure). We then have that $\nu(U_{k,n}\mid [\omega\upharpoonright c_{n+1}])=2^{-(e_n+k+1)}$ for any $\omega$, and hence $\nu(U_k\mid [\omega\upharpoonright c_{n+1}])\geq 2^{-(e_n+k+1)}$ for any $\omega$.

Let $\omega$ in $\mathsf{DR}_{\rho}^{\nu}$. Since $\sum_k I_{U_k}$ is an $L_1(\nu)$ Martin-L\"of test, one has that there is $k$ such that $\omega$ is not in $U_k$. Hence $\nu(U_k\mid [\omega\upharpoonright n])\rightarrow 0$. Suppose that $\nu(U_k\mid [\omega\upharpoonright n])\rightarrow 0$ with computable rate $m$. Let $e_n$ be such that $\varphi_{e_n}(i) = m(2^{-(i+k+1)})$ for all $i\geq 0$. Then $\varphi_{e_n}(e_n)=m(2^{-(e_n+k+1)})$. Since $c_{n+1}\geq \varphi_{e_n}(e_n) =  m(2^{-(e_n+k+1)})$, 
one has that $\nu(U_k\mid [\omega\upharpoonright c_{n+1}])<2^{-(e_n+k+1)}$, a contradiction to the previous paragraph. 
\end{proof}

Note that the c.e. sets $U_k$ constructed above are dense, and their definition interleaves Example~\ref{ex:ceopencenzer} with the halting set. This seems natural, since as noted in Proposition~\ref{prop:closureissues}, if $\overline{U_k}$ were effectively closed with the same $\nu$-measure as $U_k$, then we could include it in a $\nu$-computable basis.

\section{Proof of Theorem~\ref{thm:doobmax}}\label{sec:mdr}

First we note a fact mentioned in the introduction, namely that Maximal Doob Randomness is inbetween Martin-L\"of and Schnorr randomness:
\begin{prop}\label{prop:doob:basecase}
For all computable $p\geq 1$ one has $\mathsf{MLR}^{\nu}\subseteq \mathsf{MDR}^{\nu,p}\subseteq \mathsf{SR}^{\nu}$. 
\end{prop}
\begin{proof}
Since any $L_p(\nu)$ maximal Doob test is an $L_p(\nu)$ Martin-L\"of test, we have $\mathsf{MLR}^{\nu}\subseteq \mathsf{MDR}^{\nu,p}$. To show that $\mathsf{MDR}^{\nu,p}\subseteq \mathsf{SR}^{\nu}$, it suffices to show that any $L_p(\nu)$ Schnorr test $f$ is an $L_p(\nu)$ maximal Doob test. By Proposition~\ref{prop:exptoflsc}, let $f_s$ be from the countable dense set of $L_p(\nu)$ so that $0\leq f_s\leq f_{s+1}$ on $\mathsf{KR}^{\nu}$ and $f=\sup_s f_s$ on $\mathsf{KR}^{\nu}$ and $f_s\rightarrow f$ fast in $L_p(\nu)$. Then we can compute a subsequence $s(n)$ such that $\|f-f_{s(n)}\|_p <e^{-n}$. Then for all $k\geq 0$ one has that $\sum_n \|f-f_{s(n)}\|_p \cdot (n+1)^k \leq \sum_n e^{-n} (n+1)^k<\infty$. Then we are done by Lemma~\ref{ex:approximationbyschnorr2}.
\end{proof}

The following closure condition on $L_p(\nu)$ maximal Doob tests is the difficult component of the proof of Theorem~\ref{thm:doobmax}:
\begin{prop}\label{thm:theharddoobmax}
If $p>1$ is computable and $f$ is an $L_p(\nu)$ maximal Doob test with witness $f_s$, then $g=\sum_s \sup_n \mathbb{E}_{\nu}[f-f_s\mid \mathscr{F}_n]$ is equal on $\mathsf{KR}^{\nu}$ to a $L_p(\nu)$ maximal Doob test with witness equal on $\mathsf{KR}^{\nu}$ to $g_t= \sum_{s<t} \sup_n \mathbb{E}_{\nu}[f_t-f_s\mid \mathscr{F}_n]$.
\end{prop}
\begin{proof}
The function $g$ is equal on $\mathsf{KR}^{\nu}$ to a non-negative lsc since it is equal on $\mathsf{KR}^{\nu}$ to a supremum of non-negative lsc functions (cf. Proposition~\ref{prop:dis:1}, Proposition~
\ref{prop:dis:2}(\ref{prop:newlevy:2a:again})). Since $p>1$, by Proposition~\ref{prop:dis:3}, we have that $g_t$ is equal on $\mathsf{KR}^{\nu}$ to an $L_p(\nu)$ Schnorr test. For $k\geq 0$, the quantity $\sum_t \|g-g_t\|_p \cdot (t+1)^k$ is $\leq$:
\begin{align}
&  \sum_t \|\sum_{s\geq t} \sup_n \mathbb{E}_{\nu} [f-f_s \mid \mathscr{F}_n]\|_p\cdot (t+1)^k\label{eqn:dobomax:split:1} \\
+&  \sum_t \|\sum_{s< t} \sup_n \mathbb{E}_{\nu} [f-f_s \mid \mathscr{F}_n]- \sup_n \mathbb{E}_{\nu} [f_t-f_s \mid \mathscr{F}_n]\|_p\cdot (t+1)^k \label{eqn:dobomax:split:2}
\end{align}
To estimate (\ref{eqn:dobomax:split:1}), let us first define a computable sequence of non-negative left-c.e. reals:
\begin{equation*}
c_{s,t}=
\begin{cases}
0      & \text{if $t>s$}, \\
\| \sup_n \mathbb{E}_{\nu} [f-f_s \mid \mathscr{F}_n]\|_p\cdot (s+1)^k   & \text{if $t\leq s$}.
\end{cases}
\end{equation*}
For fixed $s\geq 0$ we have $\sum_t c_{s,t} = (s+1)\cdot c_{s,s}= \| \sup_n \mathbb{E}_{\nu} [f-f_s \mid \mathscr{F}_n]\|_p\cdot (s+1)^{k+1}$. To estimate~(\ref{eqn:dobomax:split:1}), we have the following, where the last line follows from Doob's Maximal Inequality (Lemma~\ref{lem:doob}(\ref{lem:doob:1})):
\begin{align*}
 & \sum_t \|\sum_{s\geq t} \sup_n \mathbb{E}_{\nu} [f-f_s \mid \mathscr{F}_n]\|_p\cdot (t+1)^k  \leq \sum_t \sum_{s\geq t} \| \sup_n \mathbb{E}_{\nu} [f-f_s \mid \mathscr{F}_n]\|_p\cdot (t+1)^k  \\
\leq & \sum_t \sum_{s\geq t} \| \sup_n \mathbb{E}_{\nu} [f-f_s \mid \mathscr{F}_n]\|_p\cdot (s+1)^k = \sum_t \sum_{s\geq t} c_{s,t} \\
= & \sum_t \sum_s c_{s,t} =  \sum_s \sum_t c_{s,t} = \sum_s \| \sup_n \mathbb{E}_{\nu} [f-f_s \mid \mathscr{F}_n]\|_p\cdot (s+1)^{k+1} \\
\leq & \sum_s \frac{p}{p-1} \cdot \| f-f_s\|_p \cdot (s+1)^{k+1} <\infty
\end{align*}
For~(\ref{eqn:dobomax:split:2}), we have the following, where we use Doob's Maximal Inequality (Lemma~\ref{lem:doob}) again at the end:
\begin{align*}
& \sum_t \|\sum_{s< t} \sup_n \mathbb{E}_{\nu} [f-f_s \mid \mathscr{F}_n]- \sup_n \mathbb{E}_{\nu} [f_t-f_s \mid \mathscr{F}_n]\|_p\cdot (t+1)^k \\
\leq & \sum_t \sum_{s< t} \|\sup_n \mathbb{E}_{\nu} [f-f_t \mid \mathscr{F}_n]\|_p \cdot (t+1)^k \\
= & \sum_t t\cdot \|\sup_n \mathbb{E}_{\nu} [f-f_t \mid \mathscr{F}_n]\|_p \cdot (t+1)^k \\ 
\leq & \sum_t \|\sup_n \mathbb{E}_{\nu} [f-f_t \mid \mathscr{F}_n]\|_p \cdot (t+1)^{k+1} \\ 
\leq & \sum_t  \frac{p}{p-1} \cdot \|f-f_t\|_p \cdot (t+1)^{k+1}<\infty
\end{align*}
\end{proof}

Here is the proof of Theorem~\ref{thm:doobmax}:
\begin{proof}
Suppose (\ref{thm:doobmax:1}); we prove (\ref{thm:doobmax:2}). One has that $x$ in $\mathsf{KR}^{\nu}$ and indeed $x$ in $\mathsf{SR}^{\nu}$ by Proposition~\ref{prop:doob:basecase}. Suppose now that $f$ is an $L_p(\nu)$ maximal Doob test with witness $f_s$. By the previous proposition and~(\ref{thm:doobmax:1}), we have $\lim_s \sup_n \mathbb{E}_{\nu}[f-f_s\mid \mathscr{F}_n](x)=0$.
Let $\epsilon>0$. Choose $s_0\geq 0$ such that for all $s\geq s_0$ we have $\sup_n \mathbb{E}_{\nu}[f-f_s\mid \mathscr{F}_n](x)<\frac{\epsilon}{3}$. Choose $s_1\geq s_0$ such that for all $s\geq s_1$ we have $f(x)-f_s(x)<\frac{1}{3}$. By Theorem~\ref{thm:newnewlevy} applied to $f_{s_1}$, we have that $f_{s_1}(x) = \lim_n \mathbb{E}_{\nu}[f_{s_1}\mid \mathscr{F}_n](x)$. Choose $n_0\geq 0$ such that for all $n\geq n_0$ we have
$\left|f_{s_1}(x) -\mathbb{E}_{\nu}[f_{s_1}\mid \mathscr{F}_n](x)\right|<\frac{\epsilon}{3}$.
Then putting this all together, we have for all $n\geq n_0$ that $\left|f(x)-\mathbb{E}_{\nu}[f\mid \mathscr{F}_n](x)\right|$ is $\leq$ the following:
\begin{equation*}
\left|f(x)-f_{s_1}(x)\right|+\left|f_{s_1}(x)-\mathbb{E}_{\nu}[f_{s_1}\mid \mathscr{F}_n](x)\right|+\left|\mathbb{E}_{\nu}[f_{s_1}\mid \mathscr{F}_n](x)-\mathbb{E}_{\nu}[f\mid \mathscr{F}_n](x)\right|<\epsilon \end{equation*}

The step from (\ref{thm:doobmax:2}) to (\ref{thm:doobmax:3}) is trivial. 

The step from (\ref{thm:doobmax:3}) to (\ref{thm:doobmax:1}) is exactly as in the corresponding step of the proof of Theorem~\ref{thm:newnewlevy} (in \S\ref{sec:proofmain}), but with the class of $L_p(\nu)$ Schnorr tests replaced by the class of $L_p(\nu)$ maximal Doob tests.
\end{proof}

\section{Back and forth between tests and computable points}\label{sec:app1:lpcomputable}

In this section we indicate how to state a version of Theorem~\ref{thm:newnewlevy} in terms of $L_p(\nu)$-computable points. This essentially follows by a translation method of Miyabe.

We begin with how to select a version for each computable point $L_p(\nu)$. Pathak, Rojas, and Simpson and Rute have shown how to do this via Proposition~\ref{prop:schnorrtest}. We use the following slight variant of their selection method:
\begin{defn}\label{defn:canonicalversion}
Suppose that $f$ is a computable point of $L_1(\nu)$ with witness $f_n$. Then we define a version $f_{\infty}$ in $\mathbb{L}_1(\nu)$ by: 
\begin{equation*}
f_{\infty}(x) = \begin{cases}
\lim_n f_n(x)      & \text{if $\lim_n f_n(x)$ exists}, \\
0      & \text{otherwise}.
\end{cases}
\end{equation*}
\end{defn}
Hence, Proposition~\ref{prop:schnorrtest} tells us that the definition of $f_{\infty}$ goes through the first case~break on all points of $\mathsf{SR}^{\nu}$, and on these points it is independent of the choice of the witness $f_n$. However, on $X\setminus \mathsf{SR}^{\nu}$ we have that it is dependent on the witness $f_n$. If one was working more extensively with $f_{\infty}$, one would want to develop some notation which better mark its dependence on the version $f_n$. But this dependence has the following advantage: if $f$ is an $L_p(\nu)$ Schnorr test and $f_n$ is a witness to its being $L_p(\nu)$-computable such that $f_n\rightarrow f$ everywhere (as in Proposition~\ref{prop:exptoflsc}), then $f=f_{\infty}$ everywhere.\footnote{Pathak, Rojas, and Simpson \cite[Definition 3.8 p. 314]{Pathak2014aa} work with a variant of our $f_{\infty}$ that organises the case break depending on whether $x$ is in $\mathsf{SR}^{\nu}$. Their approach has the advantage of making $f_{\infty}$, which they denote as $\widehat{f}$, entirely independent of the witness $f_n$.  Rute \cite[Definition 3.17 p. 16]{Rute2012aa} organises the case break the same as we do but sets it undefined when the limit does not exist, which prevents it from being an $L_p(\nu)$ Schnorr test.}

Miyabe proved the following transfer result for going back and forth between  $L_p(\nu)$-computable functions and differences of $L_p(\nu)$ Schnorr tests:\footnote{\cite[Theorem 4.3 p. 7]{Miyabe2013}. He proved it for $p=1$, but the proof is the same for $p\geq 1$.}
\begin{prop}\label{prop:miyabediff}
Suppose that $\nu$ is a computable point of $\mathcal{P}(X)$ and $p\geq 1$ is computable.
\begin{enumerate}[leftmargin=*]
\item \label{prop:miyabediff:1} Suppose $f$ is  a computable point of $L_p(\nu)$ with witness $f_n$. Then there are $L_p(\nu)$ Schnorr tests $g,h$ such that $f_{\infty}=g-h$ on on $\mathsf{SR}^{\nu}$.
\item \label{prop:miyabediff:2} Suppose that $g,h$ are $L_p(\nu)$ Schnorr tests. Then there is $L_p(\nu)$-computable $f$ with witness $f_n$ such that $f_{\infty}=g-h$ on on $\mathsf{SR}^{\nu}$.
\end{enumerate}
\end{prop}
\begin{proof} (Sketch) For (\ref{prop:miyabediff:1}), using Proposition~\ref{prop:shiftbasistolpst}, one shows that $g= \sum_n (f_{n+1}-f_n)^+$ and $h=\sum_n (f_{n+1}-f_n)^-$ are equal on $\mathsf{KR}^{\nu}$ to $L_p(\nu)$ Schnorr tests, where $\cdot^+$ and $\cdot^-$ denote positive and negative parts. For (\ref{prop:miyabediff:2}), one uses Proposition~\ref{prop:exptoflsc}.
\end{proof}

These observations allow us to restate Theorem~\ref{thm:newnewlevy} in terms of $L_p(\nu)$-computable points, provided one assumes Schnorr disintegrations:
\begin{cor}\label{cor:newnewlevy}
Suppose that $X$ is a computable Polish space and $\nu$ is a computable probability measure. Suppose that $\mathscr{F}_n$ is an almost-full effective filtration, equipped with Schnorr disintegrations.

If $p\geq 1$ is computable, then the following three items are equivalent for $x$ in $X$: 
  \begin{enumerate}[leftmargin=*]
  \item \label{levy:cor:1} $x$ is in $\mathsf{SR}^{\nu}(X)$.
  \item \label{levy:cor:2} $x$ is in $\mathsf{KR}^{\nu}$ and $\lim_n \mathbb{E}_{\nu}[f_{\infty} \mid \mathscr{F}_n](x)=f_{\infty}(x)$ for every $L_p(\nu)$ computable $f$ with witness $f_m$.
  \item \label{levy:cor:3} $x$ is in $\mathsf{KR}^{\nu}$ and $\lim_n \mathbb{E}_{\nu}[f_{\infty} \mid \mathscr{F}_n](x)$ exists for every $L_p(\nu)$ computable $f$ with witness $f_m$.
\end{enumerate}
\end{cor}

\begin{proof}
Suppose (\ref{levy:cor:2}); we show Theorem~\ref{thm:newnewlevy}(\ref{levy:thm:2}). But simply note that any $L_p(\nu)$ Schnorr test $f$ has a witness $f_m$ from Proposition~\ref{prop:exptoflsc} with $f=f_{\infty}$ everywhere.

Suppose Theorem~\ref{thm:newnewlevy}(\ref{levy:thm:2}); we show (\ref{levy:cor:2}). Let $f$ be $L_p(\nu)$ computable with witness $f_m$. By the previous proposition, there are two $L_p(\nu)$ Schnorr tests $g,h$ such that $f_{\infty} = g-h$ on $\mathsf{SR}^{\nu}$. Since we are working with Schnorr disintegrations, by Proposition~\ref{prop:dis:2}(\ref{prop:newlevy:2a:again}), one has that $\mathbb{E}_{\nu}[f_{\infty} \mid \mathscr{F}_n]=\mathbb{E}_{\nu}[g-h \mid \mathscr{F}_n]$ on $\mathsf{SR}^{\nu}$ for all $n\geq 0$. Then by Theorem~\ref{thm:newnewlevy}(\ref{levy:thm:2}) and Proposition~\ref{prop:dis:2}(\ref{prop:newlevy:2c:again}), we have on $\mathsf{SR}^{\nu}$ that $f_{\infty}=g-h=\lim_n \big(\mathbb{E}_{\nu}[g\mid \mathscr{F}_n]-\mathbb{E}_{\nu}[h\mid \mathscr{F}_n]\big)= \lim_n \mathbb{E}_{\nu}[f_{\infty}\mid \mathscr{F}_n]$. 

Finally, (\ref{levy:cor:2}) trivially implies (\ref{levy:cor:3}). And (\ref{levy:cor:3}) implies Theorem~\ref{thm:newnewlevy}(\ref{levy:thm:3}) since again any $L_p(\nu)$ Schnorr test $f$ has a witness $f_m$ from Proposition~\ref{prop:exptoflsc} with $f=f_{\infty}$ everywhere.
\end{proof}

\section{Martingale convergence in \texorpdfstring{$L_2(\nu)$}{L2nu}}\label{sec:ed-martin}

Our topic in this paper is convergence of the conditional expectations of random variables. But of course these are instances of martingales. In this brief section, we prove a result mentioned in \S\ref{sec:relation:work:previous}, namely that one can characterise $\mathsf{SR}^{\nu}$ in terms of convergence of certain $L_2(\nu)$ martingales. As mentioned there, this generalises a result of Rute from Cantor space to the more general setting.

If $p\geq 1$ and if $\mathscr{F}_n$ is a filtration, then a \emph{classical martingale in $L_p(\nu)$ adapted to $\mathscr{F}_n$} is a sequence $M_n$ of $\mathscr{F}_n$ measurable functions in $L_p(\nu)$ such that $M_n=\mathbb{E}_{\nu}[M_{n+1}\mid \mathscr{F}_n]$ $\nu$-a.s. When $\mathscr{F}_n$ is clear from context, we just say  \emph{classical martingale in $L_p(\nu)$}.

If $p\geq 1$ is computable and $\mathscr{F}_n$ is an effective filtration equipped with Schnorr disintegrations $\rho^{(n)}$, then a \emph{martingale of $L_p(\nu)$ Schnorr tests adapted to $\mathscr{F}_n$ and $\rho^{(n)}$} is a uniformly computable sequence $M_n$ of $\mathscr{F}_n$-measurable Schnorr $L_p(\nu)$ tests such that $M_n=\mathbb{E}_{\nu}[M_{n+1}\mid \mathscr{F}_n]$ on $\mathsf{SR}^{\nu}$, where the version of the conditional expectation is that from the disintegration. When $\mathscr{F}_n$ and $\rho^{(n)}$ are clear from context, we just say \emph{martingale of $L_p(\nu)$ Schnorr tests}.

Here is an example:
\begin{ex}\label{ex:product:martingales} (Products of mean one independent variables).

Suppose $p\geq 1$ is computable. Suppose that $f_n:X\rightarrow [0,\infty]$ is a sequence of independent $L_p(\nu)$ Schnorr tests with $\mathbb{E}_{\nu} f_n=1$ for all $n\geq 1$. Suppose that $\mathscr{F}_n=\sigma(f_1, \ldots, f_n)$ is an effective filtration equipped with Schnorr disintegrations. Then $M_n=\prod_{i=1}^n f_i$ is a martingale of $L_p(\nu)$ Schnorr tests.

To see this, note that the $M_n$ are $L_p(\nu)$ Schnorr tests: they are non-negative lsc by Proposition~\ref{prop:lscclosure}, and by independence we have $\|M_n\|_p=\prod_{i=1}^n \|f_i\|_p$, which is computable. On $\mathsf{SR}^{\nu}$ one has 
\begin{equation*}
\mathbb{E}_{\nu}[M_{n+1}\mid \mathscr{F}_n] =\mathbb{E}_{\nu}[f_{n+1} \cdot M_n\mid \mathscr{F}_n]  =M_n\cdot \mathbb{E}_{\nu}[f_{n+1}\mid \mathscr{F}_n] =M_n\cdot \mathbb{E}_{\nu} f_{n+1}=M_n
\end{equation*}
The second identity is by taking out what is known (Proposition~\ref{prop:bigprop}(\ref{prop:ctowisk})) and the third identity is by the r\^ole of independence (Proposition~\ref{prop:roleindepenedence}).
\end{ex}

Our goal is to prove the following:
\begin{thm}\label{thm:margingaleconvergencel2}
Suppose that $\nu$ is a computable point of $\mathcal{P}(X)$. Let $\mathscr{F}_n$ be an almost-full effective filtration equipped with Schnorr disintegrations.

The following are equivalent for $x$ in $X$: 
  \begin{enumerate}[leftmargin=*]
  \item\label{thm:margingaleconvergencel2a}  $x$ is in $\mathsf{SR}^{\nu}$.
  \item\label{thm:margingaleconvergencel2b}  $x$ is in $\mathsf{KR}^{\nu}$ and $\lim_n M_n(x)$ exists for every martingale $M_n$ of $L_2(\nu)$ Schnorr tests such that both $\sup_n \|M_n\|_2$ is computable and the maximal function $\sup_n M_n$ is a $L_2(\nu)$ Schnorr test.
\end{enumerate}
\end{thm}

This theorem generalises a result of Rute.\footnote{\cite[Corollary 6.8 and Theorem 12.6]{Rute2012aa}.}  But there are two important differences between our result and Rute's. First, Rute's analogue of the direction from (\ref{thm:margingaleconvergencel2b}) to (\ref{thm:margingaleconvergencel2a}) of Theorem \ref{thm:margingaleconvergencel2} only works for Cantor space and the uniform measure. Second, Rute's results do not require that the maximal function $\sup_n M_n$ be a $L_2(\nu)$ Schnorr test.

We do not know the answer to the following:
\begin{Q}
Does Theorem~\ref{thm:margingaleconvergencel2} also hold for all computable $p>1$? 
\end{Q}
\noindent It is clear from the proofs below that (\ref{thm:margingaleconvergencel2b}) to (\ref{thm:margingaleconvergencel2a}) holds for computable $p>1$. Hence it is a question of (\ref{thm:margingaleconvergencel2a}) to (\ref{thm:margingaleconvergencel2b}).

Throughout the remainder of this section, $X$ is a computable Polish space, $\nu$ is a computable point of $\mathcal{P}(X)$, and $\mathscr{F}_n$ is an effective filtration equipped with Schnorr disintegrations. We only assume that $\mathscr{F}_n$ is almost-full in the last proposition, and flag this assumption when it comes up.

We begin by noting the following two elementary results:

\begin{prop}\label{prop:goingup}
If $M_n$ is a martingale of $L_p(\nu)$ Schnorr tests, then $M_n=\mathbb{E}_{\nu}[M_{m}\mid \mathscr{F}_n]$ on $\mathsf{SR}^{\nu}$ for all $m>n$.
\end{prop}
\begin{proof}
This is by an induction on $m>n$. Suppose it holds for $m>n$. Then $M_n=\mathbb{E}_{\nu}[M_{m}\mid \mathscr{F}_n]$ on $\mathsf{SR}^{\nu}$. Since $M_m=\mathbb{E}_{\nu}[M_{m+1}\mid \mathscr{F}_m]$ on $\mathsf{SR}^{\nu}$ and since we are working with a Schnorr disintegration, we have by Proposition~\ref{prop:dis:2}(\ref{prop:newlevy:2a:again}) that  $\mathbb{E}_{\nu}[M_{m}\mid \mathscr{F}_n]=\mathbb{E}_{\nu}[\mathbb{E}_{\nu}[M_{m+1}\mid \mathscr{F}_m]\mid \mathscr{F}_n]$ on $\mathsf{SR}^{\nu}$. By the tower property (Proposition~\ref{prop:tower}), this latter is equal to $\mathbb{E}_{\nu}[M_{m+1}\mid \mathscr{F}_n]$ on $\mathsf{SR}^{\nu}$.
\end{proof}

\begin{prop}\label{prop:goingup2}
If $M_n$ is a classical martingale in $L_p(\nu)$, then $\|M_m\|_p\geq \|M_n\|_p$ for all $m>n$.
\end{prop}
\begin{proof}
The function $x\mapsto \left|x\right|^p$ is a convex function, and hence $\left|M_n\right|^p$ is a submartingale.\footnote{\cite[p. 138]{Williams1991aa}.} And the expectation of a submartingale is always non-decreasing.
\end{proof}

The following gives a canonical example of a martingale of $L_p(\nu)$ Schnorr tests, and in conjunction with Theorem~\ref{thm:newnewlevy} gives the (\ref{thm:margingaleconvergencel2b}) to (\ref{thm:margingaleconvergencel2a}) direction of Theorem~\ref{thm:margingaleconvergencel2}.

\begin{prop}\label{prop:exofmartingale}
Suppose that $p\geq 1$ is computable. 

If $f$ is an $L_p(\nu)$ Schnorr test, then $M_n:=\mathbb{E}_{\nu}[f \mid \mathscr{F}_n]$ is a martingale of $L_p(\nu)$ Schnorr tests. 

If $p>1$ and the filtration is almost-full then the maximal function $\sup_n M_n$ is also an $L_p(\nu)$ Schnorr test and $\sup_n \|M_n\|_p$ is computable.
\end{prop}
\begin{proof}
 The function $M_n$ is non-negative lsc by Proposition~\ref{prop:dis:1}, and it is $L_p(\nu)$-computable by Proposition~\ref{prop:compcont}. To see that it satisfies the martingale condition, from $M_{n+1}=\mathbb{E}_{\nu}[f\mid \mathscr{F}_{n+1}]$ everywhere we have $\mathbb{E}_{\nu}[M_{n+1}\mid \mathscr{F}_n] =\mathbb{E}_{\nu}[\mathbb{E}_{\nu}[f\mid \mathscr{F}_{n+1}]\mid \mathscr{F}_n]$ everywhere. And by the tower property (Proposition~\ref{prop:tower}), the latter is equal to $\mathbb{E}_{\nu}[f \mid \mathscr{F}_n]$ on $\mathsf{SR}^{\nu}$, which is by definition $M_n$. 

 Suppose now that the filtration is almost-full and $p>1$. By almost-fullness and Theorem~\ref{thm:newnewlevy}, we have that $f=\lim_n M_n$ on $\mathsf{SR}^{\nu}$. Since $p>1$ we have $\sup_n M_n$ is in $L_p(\nu)$ by Lemma~\ref{lem:doob}(\ref{lem:doob:1}). Then we can dominate $M_n^p$ by $(\sup_n M_n)^p$ and argue by DCT as follows, where the first identity comes from Proposition~\ref{prop:goingup2}:
 \begin{equation}\label{eqn:estimate}
\sup_n \|M_n\|_p^p =\lim_n \|M_n\|_p^p = \lim_n \int M_n^p \; d\nu = \int \lim_n M_n^p \; d\nu=\int f^p \; d\nu
\end{equation}
Since $f$ is by hypothesis a computable point of $L_p(\nu)$, we have that $\sup_n \|M_n\|_p^p$ is computable and hence likewise $\sup_n \|M_n\|_p$ is computable.
\end{proof}

In conjunction with Corollary~\ref{cor:newnewlevy}, the following proposition then gives the (\ref{thm:margingaleconvergencel2a}) to (\ref{thm:margingaleconvergencel2b}) direction of Theorem~\ref{thm:margingaleconvergencel2}. As mentioned in \S\ref{sec:relation:work:previous}, the proof largely follows the outline of Rute's own Hilbert space proof.

\begin{prop}\label{thm:uniformme}
Suppose that the filtration is almost-full.

Suppose $M_n$ is a martingale $M_n$ of $L_2(\nu)$ Schnorr tests such that both $\sup_n \|M_n\|_2$ is computable and $\sup_n M_n$ is a $L_2(\nu)$ Schnorr test.

Then there is $L_2(\nu)$-computable function $f$ such that $M_n=\mathbb{E}_{\nu}[f_\infty\mid \mathscr{F}_n]$ on $\mathsf{SR}^{\nu}$ for each $n\geq 0$.

Further, the $L_2(\nu)$-computable function $f$ can be taken to be a pointwise limit of a computable subsequence of the $M_n$, which limit exists at least on $\mathsf{SR}^{\nu}$.
\end{prop}
\begin{proof}
Recall that for $n>k$ we have by Hilbert space methods in $L_2(\nu)$ that $\mathbb{E}_{\nu} M_k M_n = \mathbb{E}_{\nu} M_k^2$.\footnote{\cite[488]{Gut2013-ou}.} This implies that for $n>k$ we have:
\begin{equation*}
\|M_n-M_k\|_2^2 = \mathbb{E}_{\nu} (M_n-M_k)^2 =   \mathbb{E}_{\nu} M_n^2 -  \mathbb{E}_{\nu} M_k^2  =\|M_n\|_2^2 - \|M_k\|_2^2 
\end{equation*}
Let $f=\lim_n M_n$, which classically is in $L_2(\nu)$. Since $\sup_n M_n$ is in $L_2(\nu)$, we can dominate $(M_n-M_k)^2$ by $2\cdot (\sup_n M_n)^2$ and argue by DCT and the previous equation that:
\begin{equation*}
\|f-M_k\|_2^2 =\lim_n \|M_n-M_k\|_2^2 =\lim_n \|M_n\|_2^2 - \|M_k\|_2^2 =(\sup_n \|M_n\|_2)^2 - \|M_k\|_2^2
\end{equation*}
Since the latter is a computable real which goes to zero, we can compute a subsequence $M_{k(n)}$ which converges to $f$ fast in $L_2(\nu)$. By Proposition~\ref{prop:schnorrtestforlsc} and Definition~\ref{defn:canonicalversion}, we have that $M_{k(n)}\rightarrow f_{\infty}$ on $\mathsf{SR}^{\nu}$.

For each $m\geq 0$, let $g_m=\mathbb{E}_{\nu}[\sup_n M_n \mid \mathscr{F}_m]$, so that $g_m$ is an $L_2(\nu)$ Schnorr test by Proposition~\ref{prop:preservetests}(\ref{prop:preservetests:2}), and so $g_m$ is finite on $\mathsf{SR}^{\nu}$. Then by Conditional DCT (Proposition~\ref{prop:bigprop}(\ref{prop:cdct})), for each $m\geq 0$ we have  that $\mathbb{E}_{\nu}[M_{k(n)} \mid \mathscr{F}_m]\rightarrow \mathbb{E}_{\nu}[f_\infty\mid \mathscr{F}_m]$ on $\mathsf{SR}^{\nu}$. But by Proposition~\ref{prop:goingup} for $k(n)>m$ we have that the former is equal to $M_m$ on $\mathsf{SR}^{\nu}$. Hence we have that $M_m=\mathbb{E}_{\nu}[f_\infty\mid \mathscr{F}_m]$ on $\mathsf{SR}^{\nu}$ for each $m\geq 0$.
\end{proof}


\section{Conclusion}

The main results of this paper (Theorems~\ref{thm:newnewlevy}, \ref{thm:newnewlevyrates}, \ref{thm:drlevy}, \ref{thm:drlevy:drates}, \ref{thm:doobmax}) characterise the points under which L\'evy's Upward Theorem holds in terms of notions from algorithmic randomness and the rates of convergence in terms of concepts from the classical theory of computation. As discussed in \S\ref{sec:relation:work:previous} this builds on work by previous authors. That which is new are the results on rates of convergence in Theorems~\ref{thm:newnewlevyrates}, \ref{thm:drlevy:drates}, the articulation of the general framework of effective disintegrations (see Definition~\ref{defn:eff:disin}, \S\ref{sec:ed-prop} for fundamental properties, and Appendix~\ref{sec:app:effective:disintegrations} for examples), a conceptually new proof of the characterisation of density randomness in the more general framework of effective disintegrations for $p>1$ (Theorem~\ref{thm:drlevy:drates}), and the articulation of the new concept of Maximal Doob Randomness (cf. Definition~\ref{defn:maximaldoob}, Theorem~\ref{thm:doobmax}, and Question~\ref{q:doob}). As far as Schnorr randomness goes, we noted in \S\ref{sec:relation:work:previous} that Theorem \ref{thm:newnewlevy} can be derived from Rute's work, modulo the verification of certain properties of effective disintegrations and the Miyabe translation method in \S\ref{sec:app1:lpcomputable}. We have extended Rute on Schnorr randomness in the generalisation of the $L_2(\nu)$ martingale result in \S\ref{sec:ed-martin}. We have also sought to present very accessible proofs, based almost entirely on the concept of $L_p(\nu)$ Schnorr test.

Our results also contribute to understanding the significance of convergence to the truth results for Bayesian inference. As was pointed out by philosophers of science, the probability one qualification in theorems like L\'evy's Upward Theorem raises the spectre of arbitrariness: a Bayesian with credences represented by a probability measure $\nu$ believes in convergence to the truth with certainty, but might do so only by arbitrarily packaging into a set of probability zero those points at which convergence fails.\footnote{\cite{Belot2013aa}, \cite[pp. 144 ff]{Earman1992aa}, \cite[pp. 28-29]{Howson2006aa}.} We have shown that for certain classes of effective random variables the packaging is anything but arbitrary. The probability one set on which convergence to the truth is successful coincides with standard classes of points which are algorithmically random by the lights of the computable probability measure. Thus, the effective typicality expressed by convergence to the truth is extensionally equivalent with a principled effective typicality of the underlying probability measure.

 \appendix

\section{Examples of classical disintegrations}\label{sec:app:classical:disintegrations}

In this appendix, we review two classical examples of disintegrations. This also affords us the opportunity to illustrate natural circumstances in which L\'evy's Upward Theorem need not hold for all points. Another reason to dwell on these two examples is that one of the main theorems of Rohlin is that, up to Borel isomorphism, ``blendings'' of these two examples are the only examples of disintegrations of countably generated $\sigma$-algebras.\footnote{\cite[\S{4} pp. 40-41]{Rohlin1962-rn}, \cite[Theorem 1.12 p. 12]{Climenhaga2012-ft}.} Our diagrams in this appendix are inspired by the few diagrams in Einsiedler-Ward,\footnote{\cite[pp. 122-123]{Einsiedler2010-yb}.} although they only work with a single $\sigma$-algebra rather than a filtration.

Most concrete examples of disintegrations involve products. We write $\lambda\otimes \mu$ for the product measure on $Y\times Z$ formed from finite measure $\lambda$ on $Y$ and finite measure $\mu$ on $Z$.

\begin{ex}\label{ex:refinedpartition}
(Refined partitions of the unit square). Let $X=[0,1]\times [0,1]$ with measure $\nu=m\otimes m$  being the product of Lebesgue measure $m$ on $[0,1]$ with itself. Let $\mathscr{D}_n$ be the dyadic partition of $X$ into $4^{n-1}$ many squares, and let $\mathscr{F}_n$ be the $\sigma$-algebra it generates. We can visualise the elements of $\mathscr{F}_n$ as any shape one can form from the squares in the below diagrams, so that like in pixelations more detailed shapes become available as $n$ gets larger:
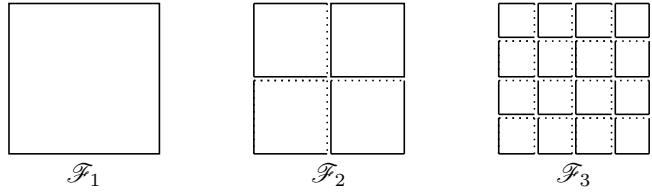
\begin{center}
\begin{tikzpicture}
\draw[semithick] (0,0) rectangle (2,2);
\node[below] at (current bounding box.south) {$\mathscr{F}_1$};
\end{tikzpicture}
\hspace{10mm}
\begin{tikzpicture}
\draw[semithick] (0,0) -- (.975,0);
\draw[semithick] (1.025, 0) -- (2,0);
\draw[semithick] (0, 1.025) -- (.975,1.025);
\draw[semithick] (0, 1.025) -- (0,2);
\draw[semithick] (0, 2) -- (.975,2);
\draw[semithick, dotted] (.975, 2) -- (.975,1.025);
\draw[semithick, dotted] (.975, .975) -- (.975, 0);
\draw[semithick] (0, 0) -- (0, .975);
\draw[semithick, dotted] (0, .975) -- (.975, .975);
\draw[semithick, dotted] (1.025, .975) -- (2, .975);
\draw[semithick, dotted] (0, 0) -- (0, .975);
\draw[semithick, dotted] (0, .975) -- (.975, .975);
\draw[semithick] (1.025, 0) -- (1.025, .975);
\draw[semithick] (2, 0) -- (2, .975);
\draw[semithick] (1.025,1.025) rectangle (2,2);
\node[below] at (current bounding box.south) {$\mathscr{F}_2$};
\end{tikzpicture}
\hspace{10mm}
\begin{tikzpicture}
\draw[semithick] (0, 0) -- (.475, 0);
\draw[semithick] (.525, 0) -- (.975, 0);
\draw[semithick] (0, .525) -- (.475, .525);
\draw[semithick] (0, .525) -- (0, .975);
\draw[semithick, dotted] (0, .975) -- (.475, .975);
\draw[semithick, dotted] (.475, .975) -- (.475, .525);
\draw[semithick, dotted] (.475, .475) -- (.475, 0);
\draw[semithick] (0, 0) -- (0, .475);
\draw[semithick, dotted] (0, .475) -- (.475, .475);
\draw[semithick, dotted] (.525, .475) -- (.975, .475);
\draw[semithick, dotted] (0, 0) -- (0, .475);
\draw[semithick, dotted] (0, .475) -- (.475, .475);
\draw[semithick] (.525, 0) -- (.525, .475);
\draw[semithick, dotted] (.975, 0) -- (.975, .475);
\draw[semithick] (.525, .525) -- (.525, .975);
\draw[semithick] (.525, .525) -- (.975, .525);
\draw[semithick, dotted] (.975, .525) -- (.975, .975);
\draw[semithick, dotted] (.975, .975) -- (.525, .975);
\draw[semithick] (1.025+0, 0) -- (1.025+.475, 0);
\draw[semithick] (1.025+.525, 0) -- (1.025+.975, 0);
\draw[semithick] (1.025+0, .525) -- (1.025+.475, .525);
\draw[semithick] (1.025+0, .525) -- (1.025+0, .975);
\draw[semithick, dotted] (1.025+0, .975) -- (1.025+.475, .975);
\draw[semithick, dotted] (1.025+.475, .975) -- (1.025+.475, .525);
\draw[semithick, dotted] (1.025+.475, .475) -- (1.025+.475, 0);
\draw[semithick] (1.025+0, 0) -- (1.025+0, .475);
\draw[semithick, dotted] (1.025+0, .475) -- (1.025+.475, .475);
\draw[semithick, dotted] (1.025+.525, .475) -- (1.025+.975, .475);
\draw[semithick, dotted] (1.025+0, 0) -- (1.025+0, .475);
\draw[semithick, dotted] (1.025+0, .475) -- (1.025+.475, .475);
\draw[semithick] (1.025+.525, 0) -- (1.025+.525, .475);
\draw[semithick] (1.025+.975, 0) -- (1.025+.975, .475);
\draw[semithick] (1.025+.525, .525) -- (1.025+.525, .975);
\draw[semithick] (1.025+.525, .525) -- (1.025+.975, .525);
\draw[semithick] (1.025+.975, .525) -- (1.025+.975, .975);
\draw[semithick, dotted] (1.025+.975, .975) -- (1.025+.525, .975);
\draw[semithick] (1.025+0, 0+1.025) -- (1.025+.475, 0+1.025);
\draw[semithick] (1.025+.525, 0+1.025) -- (1.025+.975, 0+1.025);
\draw[semithick] (1.025+0, .525+1.025) -- (1.025+.475, .525+1.025);
\draw[semithick] (1.025+0, .525+1.025) -- (1.025+0, .975+1.025);
\draw[semithick] (1.025+0, .975+1.025) -- (1.025+.475, .975+1.025);
\draw[semithick, dotted] (1.025+.475, .975+1.025) -- (1.025+.475, .525+1.025);
\draw[semithick, dotted] (1.025+.475, .475+1.025) -- (1.025+.475, 0+1.025);
\draw[semithick] (1.025+0, 0+1.025) -- (1.025+0, .475+1.025);
\draw[semithick, dotted] (1.025+0, .475+1.025) -- (1.025+.475, .475+1.025);
\draw[semithick, dotted] (1.025+.525, .475+1.025) -- (1.025+.975, .475+1.025);
\draw[semithick, dotted] (1.025+0, 0+1.025) -- (1.025+0, .475+1.025);
\draw[semithick, dotted] (1.025+0, .475+1.025) -- (1.025+.475, .475+1.025);
\draw[semithick] (1.025+.525, 0+1.025) -- (1.025+.525, .475+1.025);
\draw[semithick] (1.025+.975, 0+1.025) -- (1.025+.975, .475+1.025);
\draw[semithick] (1.025+.525, .525+1.025) -- (1.025+.525, .975+1.025);
\draw[semithick] (1.025+.525, .525+1.025) -- (1.025+.975, .525+1.025);
\draw[semithick] (1.025+.975, .525+1.025) -- (1.025+.975, .975+1.025);
\draw[semithick] (1.025+.975, .975+1.025) -- (1.025+.525, .975+1.025);
\draw[semithick] (0, 0+1.025) -- (.475, 0+1.025);
\draw[semithick] (.525, 0+1.025) -- (.975, 0+1.025);
\draw[semithick] (0, .525+1.025) -- (.475, .525+1.025);
\draw[semithick] (0, .525+1.025) -- (0, .975+1.025);
\draw[semithick] (0, .975+1.025) -- (.475, .975+1.025);
\draw[semithick, dotted] (.475, .975+1.025) -- (.475, .525+1.025);
\draw[semithick, dotted] (.475, .475+1.025) -- (.475, 0+1.025);
\draw[semithick] (0, 0+1.025) -- (0, .475+1.025);
\draw[semithick, dotted] (0, .475+1.025) -- (.475, .475+1.025);
\draw[semithick, dotted] (.525, .475+1.025) -- (.975, .475+1.025);
\draw[semithick, dotted] (0, 0+1.025) -- (0, .475+1.025);
\draw[semithick, dotted] (0, .475+1.025) -- (.475, .475+1.025);
\draw[semithick] (.525, 0+1.025) -- (.525, .475+1.025);
\draw[semithick, dotted] (.975, 0+1.025) -- (.975, .475+1.025);
\draw[semithick] (.525, .525+1.025) -- (.525, .975+1.025);
\draw[semithick] (.525, .525+1.025) -- (.975, .525+1.025);
\draw[semithick, dotted] (.975, .525+1.025) -- (.975, .975+1.025);
\draw[semithick] (.975, .975+1.025) -- (.525, .975+1.025);
\node[below] at (current bounding box.south) {$\mathscr{F}_3$};
\end{tikzpicture}
\end{center}
\noindent In this diagram, we use the familiar diagrammatic conventions from point-set topology to indicate which components of the partition contain the edges: for instance, for $n\geq 2$, the southwest square and the northwest square have two edges, the southeast square has three edges, and the northeast square has four edges. Then the following map is a disintegration of $\mathscr{F}_n$, where $\mathcal{M}^+(X)$ is again the set of finite Borel measures on $X$:
\begin{equation}\label{eqn:refinedpartition:dis}
\rho^{(n)}:X\rightarrow \mathcal{M}^+(X) \hspace{3mm} \mbox{ by } \hspace{3mm} \rho^{(n)}_{w}=\sum_{Q\in \mathscr{D}_n} \nu(\cdot \mid Q)\cdot I_{Q}(w)
\end{equation}
Further, using equation~(\ref{eqn:disintegrate}) from \S\ref{sec:disintegration}, one has the following associated formula for the version of conditional expectation of $f$ with respect to $\mathscr{F}_n$:
\begin{equation}\label{eqn:refinedpartition:cexp}
\mathbb{E}_{\nu}[f\mid \mathscr{F}_n](w) = \sum_{Q\in \mathscr{D}_n} \bigg(\frac{1}{\nu(Q)} \int_{Q} f(x) \; d\nu(x) \bigg) \cdot I_{Q}(w)
\end{equation}
Suppose that at stage $n$, the agent's world $w$ is located in the square $Q_n(w)$ from~$\mathscr{D}_n$. Intuitively this means that the agent's evidence at this stage of inquiry is~$Q_n(w)$. Then (\ref{eqn:refinedpartition:cexp}) says that the agent's best estimate as to the value of a random variable $f$ at this stage is obtained by averaging $f$ over the event $Q_n(w)$ according to the prior probability measure $\nu$, and then making it higher to the extent that the prior probability $\nu(Q_w(n))$ is lower. In the case where the random variable $f$ is the indicator function $I_C$ of a Borel event $C$, this best estimate is just the usual conditional probability $\nu(C\mid Q_n(w))$. For instance, if $C$ is the closed polygon displayed below, then the conditional probability of the agent at stage $n$ is higher than if she were at another world $w^{\prime}$ iff there is more overlap between $Q_n(w),C$ than between $Q_n(w^{\prime}),C$:
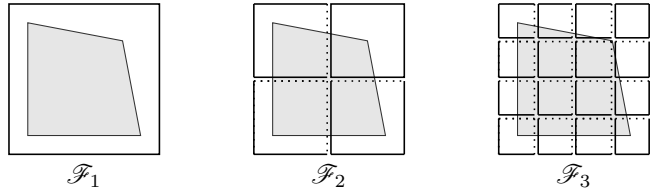
\vspace{2mm}
\begin{center}
\begin{tikzpicture}
\draw[fill=gray!20, draw=black!80] (1.75, .25)--(1.51, 1.51)--(.25,1.75)--(.25,.25)--cycle;
\draw[semithick] (0,0) rectangle (2,2);
\node[below] at (current bounding box.south) {$\mathscr{F}_1$};
\end{tikzpicture}
\hspace{10mm}
\begin{tikzpicture}
\draw[fill=gray!20, draw=black!80] (1.75, .25)--(1.51, 1.51)--(.25,1.75)--(.25,.25)--cycle;
\draw[semithick] (0,0) -- (.975,0);
\draw[semithick] (1.025, 0) -- (2,0);
\draw[semithick] (0, 1.025) -- (.975,1.025);
\draw[semithick] (0, 1.025) -- (0,2);
\draw[semithick] (0, 2) -- (.975,2);
\draw[semithick, dotted] (.975, 2) -- (.975,1.025);
\draw[semithick, dotted] (.975, .975) -- (.975, 0);
\draw[semithick] (0, 0) -- (0, .975);
\draw[semithick, dotted] (0, .975) -- (.975, .975);
\draw[semithick, dotted] (1.025, .975) -- (2, .975);
\draw[semithick, dotted] (0, 0) -- (0, .975);
\draw[semithick, dotted] (0, .975) -- (.975, .975);
\draw[semithick] (1.025, 0) -- (1.025, .975);
\draw[semithick] (2, 0) -- (2, .975);
\draw[semithick] (1.025,1.025) rectangle (2,2);
\node[below] at (current bounding box.south) {$\mathscr{F}_2$};
\end{tikzpicture}
\hspace{10mm}
\begin{tikzpicture}
\draw[fill=gray!20, draw=black!80] (1.75, .25)--(1.51, 1.51)--(.25,1.75)--(.25,.25)--cycle;
\draw[semithick] (0, 0) -- (.475, 0);
\draw[semithick] (.525, 0) -- (.975, 0);
\draw[semithick] (0, .525) -- (.475, .525);
\draw[semithick] (0, .525) -- (0, .975);
\draw[semithick, dotted] (0, .975) -- (.475, .975);
\draw[semithick, dotted] (.475, .975) -- (.475, .525);
\draw[semithick, dotted] (.475, .475) -- (.475, 0);
\draw[semithick] (0, 0) -- (0, .475);
\draw[semithick, dotted] (0, .475) -- (.475, .475);
\draw[semithick, dotted] (.525, .475) -- (.975, .475);
\draw[semithick, dotted] (0, 0) -- (0, .475);
\draw[semithick, dotted] (0, .475) -- (.475, .475);
\draw[semithick] (.525, 0) -- (.525, .475);
\draw[semithick, dotted] (.975, 0) -- (.975, .475);
\draw[semithick] (.525, .525) -- (.525, .975);
\draw[semithick] (.525, .525) -- (.975, .525);
\draw[semithick, dotted] (.975, .525) -- (.975, .975);
\draw[semithick, dotted] (.975, .975) -- (.525, .975);
\draw[semithick] (1.025+0, 0) -- (1.025+.475, 0);
\draw[semithick] (1.025+.525, 0) -- (1.025+.975, 0);
\draw[semithick] (1.025+0, .525) -- (1.025+.475, .525);
\draw[semithick] (1.025+0, .525) -- (1.025+0, .975);
\draw[semithick, dotted] (1.025+0, .975) -- (1.025+.475, .975);
\draw[semithick, dotted] (1.025+.475, .975) -- (1.025+.475, .525);
\draw[semithick, dotted] (1.025+.475, .475) -- (1.025+.475, 0);
\draw[semithick] (1.025+0, 0) -- (1.025+0, .475);
\draw[semithick, dotted] (1.025+0, .475) -- (1.025+.475, .475);
\draw[semithick, dotted] (1.025+.525, .475) -- (1.025+.975, .475);
\draw[semithick, dotted] (1.025+0, 0) -- (1.025+0, .475);
\draw[semithick, dotted] (1.025+0, .475) -- (1.025+.475, .475);
\draw[semithick] (1.025+.525, 0) -- (1.025+.525, .475);
\draw[semithick] (1.025+.975, 0) -- (1.025+.975, .475);
\draw[semithick] (1.025+.525, .525) -- (1.025+.525, .975);
\draw[semithick] (1.025+.525, .525) -- (1.025+.975, .525);
\draw[semithick] (1.025+.975, .525) -- (1.025+.975, .975);
\draw[semithick, dotted] (1.025+.975, .975) -- (1.025+.525, .975);
\draw[semithick] (1.025+0, 0+1.025) -- (1.025+.475, 0+1.025);
\draw[semithick] (1.025+.525, 0+1.025) -- (1.025+.975, 0+1.025);
\draw[semithick] (1.025+0, .525+1.025) -- (1.025+.475, .525+1.025);
\draw[semithick] (1.025+0, .525+1.025) -- (1.025+0, .975+1.025);
\draw[semithick] (1.025+0, .975+1.025) -- (1.025+.475, .975+1.025);
\draw[semithick, dotted] (1.025+.475, .975+1.025) -- (1.025+.475, .525+1.025);
\draw[semithick, dotted] (1.025+.475, .475+1.025) -- (1.025+.475, 0+1.025);
\draw[semithick] (1.025+0, 0+1.025) -- (1.025+0, .475+1.025);
\draw[semithick, dotted] (1.025+0, .475+1.025) -- (1.025+.475, .475+1.025);
\draw[semithick, dotted] (1.025+.525, .475+1.025) -- (1.025+.975, .475+1.025);
\draw[semithick, dotted] (1.025+0, 0+1.025) -- (1.025+0, .475+1.025);
\draw[semithick, dotted] (1.025+0, .475+1.025) -- (1.025+.475, .475+1.025);
\draw[semithick] (1.025+.525, 0+1.025) -- (1.025+.525, .475+1.025);
\draw[semithick] (1.025+.975, 0+1.025) -- (1.025+.975, .475+1.025);
\draw[semithick] (1.025+.525, .525+1.025) -- (1.025+.525, .975+1.025);
\draw[semithick] (1.025+.525, .525+1.025) -- (1.025+.975, .525+1.025);
\draw[semithick] (1.025+.975, .525+1.025) -- (1.025+.975, .975+1.025);
\draw[semithick] (1.025+.975, .975+1.025) -- (1.025+.525, .975+1.025);
\draw[semithick] (0, 0+1.025) -- (.475, 0+1.025);
\draw[semithick] (.525, 0+1.025) -- (.975, 0+1.025);
\draw[semithick] (0, .525+1.025) -- (.475, .525+1.025);
\draw[semithick] (0, .525+1.025) -- (0, .975+1.025);
\draw[semithick] (0, .975+1.025) -- (.475, .975+1.025);
\draw[semithick, dotted] (.475, .975+1.025) -- (.475, .525+1.025);
\draw[semithick, dotted] (.475, .475+1.025) -- (.475, 0+1.025);
\draw[semithick] (0, 0+1.025) -- (0, .475+1.025);
\draw[semithick, dotted] (0, .475+1.025) -- (.475, .475+1.025);
\draw[semithick, dotted] (.525, .475+1.025) -- (.975, .475+1.025);
\draw[semithick, dotted] (0, 0+1.025) -- (0, .475+1.025);
\draw[semithick, dotted] (0, .475+1.025) -- (.475, .475+1.025);
\draw[semithick] (.525, 0+1.025) -- (.525, .475+1.025);
\draw[semithick, dotted] (.975, 0+1.025) -- (.975, .475+1.025);
\draw[semithick] (.525, .525+1.025) -- (.525, .975+1.025);
\draw[semithick] (.525, .525+1.025) -- (.975, .525+1.025);
\draw[semithick, dotted] (.975, .525+1.025) -- (.975, .975+1.025);
\draw[semithick] (.975, .975+1.025) -- (.525, .975+1.025);
\node[below] at (current bounding box.south) {$\mathscr{F}_3$};
\end{tikzpicture}
\end{center}
This example also vividly illustrates how $\lim_n \mathbb{E}_{\nu}[f\mid \mathscr{F}_n](w)=f(w)$ can fail. For instance, take the vertex $w=(.75, .75)$, which is in the polygon since it is closed. For all $n\geq 3$, this point is the southwest vertex of a dyadic square in $\mathscr{F}_n$, and such a square overlaps the closed polygon only at this vertex. Hence for the usc function $f=I_C$, one has both $f(w)=1$ and $\mathbb{E}_{\nu}[f\mid \mathscr{F}_n](w)=0$ for all $n\geq 3$. 

\end{ex}

In simple examples like this one, geometric intuition can guide us as to what points $\lim_n \mathbb{E}_{\nu}[f\mid \mathscr{F}_n]=f$ holds. The further assurance the classical version of Levy's Upward Theorem provides is that $\lim_n \mathbb{E}_{\nu}[f\mid \mathscr{F}_n]=f$ holds on a set of $\nu$-probability one, even when geometric intuition is unavailable. The additional  assurance that Theorems~\ref{thm:newnewlevy},\ref{thm:drlevy}, \ref{thm:doobmax} provides is that $\lim_n \mathbb{E}_{\nu}[f\mid \mathscr{F}_n]=f$ holds on the random points relative to $\nu$, for a large class of effective random variables~$f$. From this perspective, the problem with our vertex  $w=(.75, .75)$ is that it is not sufficiently random, which enabled us to construct a random variable which failed to converge to the truth at this point.

\begin{ex}\label{ex:refinedlines}
 (Refined lines in the unit square) Let $X=[0,1]\times [0,1]$ with Lebesgue measure $\nu=m\otimes m$ being the product of Lebesgue measure $m$ on $[0,1]$ with itself. Let $\mathscr{G}_n$ be the $\sigma$-algebra on $[0,1]\times [0,1]$ generated by events of the form $B\times Q$, where $B\subseteq [0,1]$ is Borel and where $Q$ is from a dyadic partition $\mathscr{D}_n$ of $[0,1]$ into $2^{n-1}$ (half)-closed intervals of equal length. Intuitively, $\mathscr{G}_n$ is the $\sigma$-algebra of evidence where the agent knows everything there is to know about the $x$-component at the outset, but is progressively learning more about the $y$-component. Since events of the form $\{x\}\times [0,1]$ are in $\mathscr{G}_1$, we can depict the $\sigma$-algebra $\mathscr{G}_1$ as the decomposition of $X$ into vertical lines. Likewise, we can depict $\mathscr{G}_2$ as the decomposition of $X$ into half vertical lines, etc. While we draw only ten such vertical lines in the below diagram, the idea is that $X$ is being decomposed into continuum-many such vertical lines at each stage:
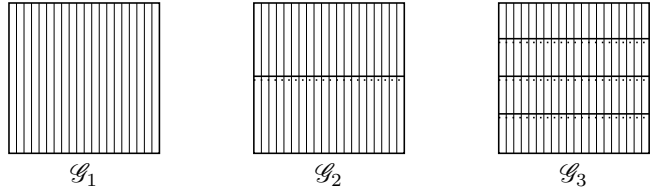
\vspace{2mm}
\begin{center}
\begin{tikzpicture}
\draw[very thin] (.10,0) -- (.10,2);
\draw[very thin] (.20,0) -- (.20,2);
\draw[very thin] (.30,0) -- (.30,2);
\draw[very thin] (.40,0) -- (.40,2);
\draw[very thin] (.50,0) -- (.50,2);
\draw[very thin] (.60,0) -- (.60,2);
\draw[very thin] (.70,0) -- (.70,2);
\draw[very thin] (.80,0) -- (.80,2);
\draw[very thin] (.90,0) -- (.90,2);
\draw[very thin] (1.00,0) -- (1.00,2);
\draw[very thin] (1.10,0) -- (1.10,2);
\draw[very thin] (1.20,0) -- (1.20,2);
\draw[very thin] (1.30,0) -- (1.30,2);
\draw[very thin] (1.40,0) -- (1.40,2);
\draw[very thin] (1.50,0) -- (1.50,2);
\draw[very thin] (1.60,0) -- (1.60,2);
\draw[very thin] (1.70,0) -- (1.70,2);
\draw[very thin] (1.80,0) -- (1.80,2);
\draw[very thin] (1.90,0) -- (1.90,2);
\draw[semithick] (0,0) rectangle (2,2);
\node[below] at (current bounding box.south) {$\mathscr{G}_1$};
\end{tikzpicture}
\hspace{10mm}
\begin{tikzpicture}
\draw[very thin] (.10,0) -- (.10,2);
\draw[very thin] (.20,0) -- (.20,2);
\draw[very thin] (.30,0) -- (.30,2);
\draw[very thin] (.40,0) -- (.40,2);
\draw[very thin] (.50,0) -- (.50,2);
\draw[very thin] (.60,0) -- (.60,2);
\draw[very thin] (.70,0) -- (.70,2);
\draw[very thin] (.80,0) -- (.80,2);
\draw[very thin] (.90,0) -- (.90,2);
\draw[very thin] (1.00,0) -- (1.00,2);
\draw[very thin] (1.10,0) -- (1.10,2);
\draw[very thin] (1.20,0) -- (1.20,2);
\draw[very thin] (1.30,0) -- (1.30,2);
\draw[very thin] (1.40,0) -- (1.40,2);
\draw[very thin] (1.50,0) -- (1.50,2);
\draw[very thin] (1.60,0) -- (1.60,2);
\draw[very thin] (1.70,0) -- (1.70,2);
\draw[very thin] (1.80,0) -- (1.80,2);
\draw[very thin] (1.90,0) -- (1.90,2);
\draw[semithick] (0,0) rectangle (2,2);
\draw[semithick, dotted] (0,.975) -- (2,.975);
\draw[semithick] (0,1.025) -- (2,1.025);
\node[below] at (current bounding box.south) {$\mathscr{G}_2$};
\end{tikzpicture}
\hspace{10mm}
\begin{tikzpicture}
\draw[very thin] (.10,0) -- (.10,2);
\draw[very thin] (.20,0) -- (.20,2);
\draw[very thin] (.30,0) -- (.30,2);
\draw[very thin] (.40,0) -- (.40,2);
\draw[very thin] (.50,0) -- (.50,2);
\draw[very thin] (.60,0) -- (.60,2);
\draw[very thin] (.70,0) -- (.70,2);
\draw[very thin] (.80,0) -- (.80,2);
\draw[very thin] (.90,0) -- (.90,2);
\draw[very thin] (1.00,0) -- (1.00,2);
\draw[very thin] (1.10,0) -- (1.10,2);
\draw[very thin] (1.20,0) -- (1.20,2);
\draw[very thin] (1.30,0) -- (1.30,2);
\draw[very thin] (1.40,0) -- (1.40,2);
\draw[very thin] (1.50,0) -- (1.50,2);
\draw[very thin] (1.60,0) -- (1.60,2);
\draw[very thin] (1.70,0) -- (1.70,2);
\draw[very thin] (1.80,0) -- (1.80,2);
\draw[very thin] (1.90,0) -- (1.90,2);
\draw[semithick] (0,0) rectangle (2,2);
\draw[semithick] (0,.525) -- (2,.525);
\draw[semithick] (0,1.025) -- (2,1.025);
\draw[semithick] (0,1.525) -- (2,1.525);
\draw[semithick, dotted] (0,.475) -- (2,.475);
\draw[semithick, dotted] (0,.975) -- (2,.975);
\draw[semithick, dotted] (0,1.475) -- (2,1.475);

\node[below] at (current bounding box.south) {$\mathscr{G}_3$};
\end{tikzpicture}
\end{center}
\noindent Then the following map is a disintegration of $\mathscr{G}_n$, where $\delta_u$ is the Dirac measure centred on $u$: 
\begin{equation}\label{eqn:refinedlines:dis}
\rho^{(n)}:X\rightarrow \mathcal{M}^+(X) \hspace{3mm} \mbox{ by } \hspace{3mm} \rho^{(n)}_{(u,v)}=\delta_u\otimes \sum_{Q\in \mathscr{D}_n} m(\cdot \mid Q)\cdot I_{Q}(v) 
\end{equation}
Further, using equation~(\ref{eqn:disintegrate}) from \S\ref{sec:disintegration}, one has the following associated formula for the version of conditional expectation of $g$ with respect to $\mathscr{G}_n$:
\begin{equation}\label{eqn:refinedlines:cexp}
\mathbb{E}_{\nu}[g\mid \mathscr{G}_n](u,v) = \sum_{Q\in \mathscr{D}_n} \bigg(\frac{1}{m(Q)} \int_{Q} g(u,t) \; dm(t) \bigg) \cdot I_{Q}(v)
\end{equation}
Suppose that at stage $n$, the agent's world $(u,v)$ is such that its second coordinate $v$ located in the interval $Q_n(v)$ from $\mathscr{D}_n$. Intuitively this means that the agent's evidence at this stage of inquiry is the line $\{u\}\times Q_n(v)$. Then (\ref{eqn:refinedpartition:cexp}) says that the agent's best estimate as to the value of a random variable $g$ at this world and stage is obtained by defining the one-place random variable $f(v)=g(u,v)$ and then doing a one-dimensional analogue of the update in Example~\ref{ex:refinedpartition}. For instance, if $C$ is the displayed closed triangle and we consider the usc function $g=I_C$, then $\mathbb{E}_{\nu}[g\mid \mathscr{G}_n](u,v)$ is obtained by calculating the length of the line $C\cap (\{u\}\times Q_n(v))$, and then by multiplying by a factor of $2^{n-1}$ which is responsive to smaller partitions of the $y$-axis involving less likely events. We illustrate this with respect to the marked point $(u,v)=(\frac{1}{3}, \frac{2}{3})$ in the below diagram, where the line $C\cap (\{u\}\times Q_n(v))$ is indicated with a heavier dark line:
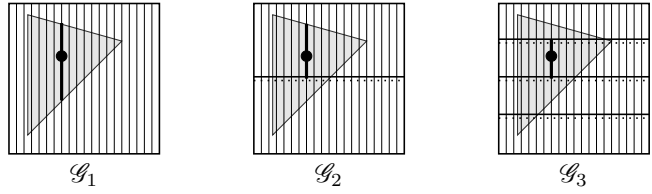
\vspace{2mm}
\begin{center}
\begin{tikzpicture}
\draw[fill=gray!20, draw=black!80] (.25, .25)--(.25, 1.85)--(1.5,1.5)--cycle;
\filldraw[black] (.70,1.30) circle (2pt);
\draw[very thick] (.70,.71) -- (.70,1.74);
\draw[very thin] (.10,0) -- (.10,2);
\draw[very thin] (.20,0) -- (.20,2);
\draw[very thin] (.30,0) -- (.30,2);
\draw[very thin] (.40,0) -- (.40,2);
\draw[very thin] (.50,0) -- (.50,2);
\draw[very thin] (.60,0) -- (.60,2);
\draw[very thin] (.70,0) -- (.70,2);
\draw[very thin] (.80,0) -- (.80,2);
\draw[very thin] (.90,0) -- (.90,2);
\draw[very thin] (1.00,0) -- (1.00,2);
\draw[very thin] (1.10,0) -- (1.10,2);
\draw[very thin] (1.20,0) -- (1.20,2);
\draw[very thin] (1.30,0) -- (1.30,2);
\draw[very thin] (1.40,0) -- (1.40,2);
\draw[very thin] (1.50,0) -- (1.50,2);
\draw[very thin] (1.60,0) -- (1.60,2);
\draw[very thin] (1.70,0) -- (1.70,2);
\draw[very thin] (1.80,0) -- (1.80,2);
\draw[very thin] (1.90,0) -- (1.90,2);
\draw[semithick] (0,0) rectangle (2,2);
\node[below] at (current bounding box.south) {$\mathscr{G}_1$};
\end{tikzpicture}
\hspace{10mm}
\begin{tikzpicture}
\draw[fill=gray!20, draw=black!80] (.25, .25)--(.25, 1.85)--(1.5,1.5)--cycle;
\filldraw[black] (.70,1.30) circle (2pt);
\draw[very thick] (.70,1) -- (.70,1.73);
\draw[very thin] (.10,0) -- (.10,2);
\draw[very thin] (.20,0) -- (.20,2);
\draw[very thin] (.30,0) -- (.30,2);
\draw[very thin] (.40,0) -- (.40,2);
\draw[very thin] (.50,0) -- (.50,2);
\draw[very thin] (.60,0) -- (.60,2);
\draw[very thin] (.70,0) -- (.70,2);
\draw[very thin] (.80,0) -- (.80,2);
\draw[very thin] (.90,0) -- (.90,2);
\draw[very thin] (1.00,0) -- (1.00,2);
\draw[very thin] (1.10,0) -- (1.10,2);
\draw[very thin] (1.20,0) -- (1.20,2);
\draw[very thin] (1.30,0) -- (1.30,2);
\draw[very thin] (1.40,0) -- (1.40,2);
\draw[very thin] (1.50,0) -- (1.50,2);
\draw[very thin] (1.60,0) -- (1.60,2);
\draw[very thin] (1.70,0) -- (1.70,2);
\draw[very thin] (1.80,0) -- (1.80,2);
\draw[very thin] (1.90,0) -- (1.90,2);
\draw[semithick] (0,0) rectangle (2,2);
\draw[semithick, dotted] (0,.975) -- (2,.975);
\draw[semithick] (0,1.025) -- (2,1.025);
\node[below] at (current bounding box.south) {$\mathscr{G}_2$};
\end{tikzpicture}
\hspace{10mm}
\begin{tikzpicture}
\draw[fill=gray!20, draw=black!80] (.25, .25)--(.25, 1.85)--(1.5,1.5)--cycle;
\filldraw[black] (.70,1.30) circle (2pt);
\draw[very thick] (.70,1) -- (.70,1.52);
\draw[very thin] (.10,0) -- (.10,2);
\draw[very thin] (.20,0) -- (.20,2);
\draw[very thin] (.30,0) -- (.30,2);
\draw[very thin] (.40,0) -- (.40,2);
\draw[very thin] (.50,0) -- (.50,2);
\draw[very thin] (.60,0) -- (.60,2);
\draw[very thin] (.70,0) -- (.70,2);
\draw[very thin] (.80,0) -- (.80,2);
\draw[very thin] (.90,0) -- (.90,2);
\draw[very thin] (1.00,0) -- (1.00,2);
\draw[very thin] (1.10,0) -- (1.10,2);
\draw[very thin] (1.20,0) -- (1.20,2);
\draw[very thin] (1.30,0) -- (1.30,2);
\draw[very thin] (1.40,0) -- (1.40,2);
\draw[very thin] (1.50,0) -- (1.50,2);
\draw[very thin] (1.60,0) -- (1.60,2);
\draw[very thin] (1.70,0) -- (1.70,2);
\draw[very thin] (1.80,0) -- (1.80,2);
\draw[very thin] (1.90,0) -- (1.90,2);
\draw[semithick] (0,0) rectangle (2,2);
\draw[semithick] (0,.525) -- (2,.525);
\draw[semithick] (0,1.025) -- (2,1.025);
\draw[semithick] (0,1.525) -- (2,1.525);
\draw[semithick, dotted] (0,.475) -- (2,.475);
\draw[semithick, dotted] (0,.975) -- (2,.975);
\draw[semithick, dotted] (0,1.475) -- (2,1.475);

\node[below] at (current bounding box.south) {$\mathscr{G}_3$};
\end{tikzpicture}
\end{center}

In contrast to the previous Example~\ref{ex:refinedpartition}, in this example many of the events in $\mathscr{G}_n$ have measure zero according to the prior probability measure $\nu$. Like in the previous Example~\ref{ex:refinedpartition}, we have natural pointwise failures of $\lim_n \mathbb{E}_{\nu}[f\mid \mathscr{F}_n]=f$ here as well: the rightmost vertex of the triangle displays the same kind of failure as in the previous example. And like in that case, the interpretation suggested by Theorems~\ref{thm:newnewlevy},\ref{thm:drlevy}, \ref{thm:doobmax}  is that the vertex is insufficiently random.

\end{ex}

\section{Examples of effective disintegrations}\label{sec:app:effective:disintegrations}

In this section, we describe several examples of effective disintegrations (cf. Definition~\ref{defn:eff:disin}). We focus for the most part on effectivizing the two paradigmatic Examples~\ref{ex:refinedpartition}-\ref{ex:refinedlines} from the previous appendix, but we also include a countable product (Example~\ref{ex:linescountableproduct2}). In a sequel to this paper, we look also at Bayesian parameter spaces and sample spaces.

One example like Example~\ref{ex:refinedpartition} is already widely-used 
in algorithmic randomness, although it is not usually thematized as such:
\begin{ex}\label{ex:refinedpartitionconcrete}\label{ex:refinedpartition:alt} (The canonical concrete refined partition disintegrations).

Suppose that $T\subseteq \mathbb{N}^{<\mathbb{N}}$ is a computable tree with no dead ends. Let $X=[T]$, the paths through $T$, which is a computable Polish space, and suppose $\nu$ in $\mathcal{P}(X)$ is computable with full support.

Suppose that $\mathscr{F}_n$ is the effective refined partition generated by the length $n$ strings in $T$. That is, $\mathscr{F}_n$ is generated by the sets $[\sigma]$ of paths in $T$ through the length~$n$ strings~$\sigma$.

Let $\rho^{(n)}:X\rightarrow \mathcal{P}(X)$ by $\rho_{\omega}= \nu(\cdot \mid [\omega\upharpoonright n])$.

Then $\rho^{(n)}$ is a Martin-L\"of disintegration of $\mathscr{F}_n$ with respect to $\nu$ and one has the following expression for the version of conditional expectation, which is defined for all $f$ in $\mathbb{L}_1(\nu)$ and all $\omega$ in $X$: 
\begin{equation}\label{eqn:firstconexpformula}
\mathbb{E}_{\nu}[f\mid \mathscr{F}_n](\omega) = \frac{1}{\nu([\omega\upharpoonright n])} \int_{[\omega\upharpoonright n]} f \; d\nu
\end{equation}
\end{ex}
Since we want to generalise this in what follows, we defer the verification that it is a Martin-L\"of disintegration.

The following is an effective disintegration like Example~\ref{ex:refinedlines}:
\begin{ex}\label{ex:refinedlines:alt} 
(The canonical concrete refined lines disintegrations).

Suppose that $S,T\subseteq \mathbb{N}^{<\mathbb{N}}$ are computable trees with no dead ends. Let $Y=[S]$ and $Z=[T]$, the paths through $S,T$ respectively, which are computable Polish spaces, and suppose $\lambda$ in $\mathcal{P}(Y)$ and $\mu$ in $\mathcal{P}(Z)$ are computable with full support. 

Let $\mathscr{F}_n$ be the $\sigma$-algebra on $Y\times Z$ generated by sets of the form $U\times [\tau]$, where $U$ ranges over c.e. opens from a $\lambda$-computable basis on $Y$, and where $\tau$ ranges over length $n$ strings in $T$.

Then the map $\rho: Y\times Z\rightarrow \mathcal{M}^+(Y\times Z)$ given by $\rho^{(n)}_{(\omega,\omega^{\prime})} = \delta_{\omega}\otimes \mu(\cdot \mid [\omega^{\prime}\upharpoonright n])$ is a Martin-L\"of disintegration of $\mathscr{F}_n$ with respect to $\lambda \otimes \mu$, and one has the following expression for the version of conditional expectation, which for each $f$ in $\mathbb{L}_1(\lambda\otimes\mu)$ is defined for $(\lambda \otimes \mu)$-a.s. many $(\omega,\omega^{\prime})$ in $Y\times Z$:
\begin{equation}\label{eqn:secondconexpformula}
\mathbb{E}_{\lambda\otimes\mu} [f \mid \mathscr{F}_n](\omega,\omega^{\prime})= \frac{1}{\mu([\omega^{\prime}\upharpoonright n])} \int_{[\omega^{\prime}\upharpoonright n]} f(\omega,\theta) \; d\mu(\theta)
\end{equation}
\end{ex}
Note that $\omega$ is free under the integral sign. Since there are no continuity assumptions on $f$ (it is merely an element of $\mathbb{L}_1(\nu)$) the value of the conditional expectation in (\ref{eqn:secondconexpformula}) can apriori change drastically with small changes of $\omega$. By contrast, in (\ref{eqn:firstconexpformula}), $\omega$'s contribution is restricted to its first $n$ bits.

In what follows, we want to generalise these two examples to a broader class of computable Polish spaces and verify that they are indeed Martin-L\"of disintegrations. We begin by generalising the way in which the previous examples involve partitions. We define a special case of Definition~\ref{defn:core}(\ref{defn:core:4}):
\begin{defn}
A sub-$\sigma$-algebra $\mathscr{F}$ of the Borel sets on $X$ is a \emph{$\nu$-effective partition} if it is generated by a computable sequence of events  $\{A_i: i\in I\}$ from the algebra $\mathscr{A}$ generated by $\nu$-computable basis $\mathscr{B}$ such that the events $\{A_i: i\in I\}$ are a partition of $X$.

Given such a partition with its computable index set $I$, we define the c.e. set $I^+=\{i\in I: \nu(A_i)>0\}$.

Further, a \emph{$\nu$-effective softening} of $\mathscr{F}$ is a pairwise disjoint computable sequence of c.e. opens $\{U_i: i\in I\}$ such that $U_i=A_i$ on $\mathsf{KR}^{\nu}$.
\end{defn}

\begin{prop}\label{prop:technical}
Every effective partition has an effective softening. 
\end{prop}
\begin{proof}
Since the computable sequence $A_i$ comes from the algebra generated by a $\nu$-computable basis, by Proposition~\ref{prop:krplusbasis}, there is a computable sequence of c.e. opens $V_i$ and effectively closed $C_i\supseteq V_i$ with $\nu(C_i)=\nu(A_i)$ and $V_i=A_i$ on  $\mathsf{KR}^{\nu}$. Then define recursively the sequence of c.e. opens by $U_0=V_0$ and $U_{n+1}=V_{n+1}\setminus \bigcup_{m\leq n} C_m$.
\end{proof}

Softenings of full partitions are a canonical way to obtain almost-full effective filtrations (cf. Definition~\ref{defn:core}(\ref{defn:core:5})):
\begin{prop}
Suppose that $\mathscr{F}_n$ is a full $\nu$-effective partition equipped with effective softenings which generate $\sigma$-algebras $\mathscr{G}_n$. Then $\mathscr{G}_n$ is an almost-full $\nu$-effective partition.
\end{prop}
\begin{proof}
Uniformly from an index for a c.e. open $U$, we can compute an index for a sequence $A_{m_i}$ from the sequences which generate the filtration $\mathscr{F}_0, \mathscr{F}_1, \ldots$ such that $U=\bigcup_i A_{m_i}$. Each $A_{m_i}$ is equal on $\mathsf{KR}^{\nu}$ to $U_{m_i}$, where the latter comes from the softening. Then $U$ is equal on $\mathsf{KR}^{\nu}$ to $\bigcup_i U_{m_i}$.
\end{proof}

Here is how to organise a suitably generalised version of a single stage of the filtration of Example~\ref{ex:refinedpartition}:
\begin{prop}\label{prop:partitionsaredisintegrations}
Let $\nu$ be a computable point of $\mathcal{P}(X)$. Let $\mathscr{F}$ be an effective partition $\{A_i: i\in I\}$ of $X$ with effective softening $\{U_i: i\in I\}$.

Let $\rho:X\rightarrow \mathcal{M}^+(X)$ by $\rho_x=\sum_{i\in I^+} \nu(\cdot \mid U_i)\cdot I_{U_i}(x)$. 

Then $\rho$ is a Martin-L\"of disintegration of $\mathscr{F}$ with respect to $\nu$ and one has the following expression for the version of conditional expectation, which is defined for all $f$ in $\mathbb{L}_1(\nu)$ and all $x$ in $X$: 
\begin{equation}\label{eqn:partitionex}
\mathbb{E}_{\nu}[f\mid \mathscr{F}](x) = \sum_{i\in I^+} \bigg(\frac{1}{\nu(U_i)} \int_{U_i} f \; d\nu\bigg) \cdot I_{U_i}(x)
\end{equation}
\end{prop}
\begin{proof}
By the definition of effective softening, the sets $U_i, U_j$ for distinct $i,j$ in $I^+$ have empty intersection. Hence, the map $\rho$ has codomain $\mathcal{M}^+(X)$. In particular, if $x$ is in $U_i$ for $i$ in $I^+$, then $\rho_x=\nu(\cdot \mid U_i)$, which is in $\mathcal{P}(X)$ and hence in $\mathcal{M}^+(X)$. But if $x$ not in any $U_i$ for $i$ in $I^+$, then $\rho_x=0$, which is a point of $\mathcal{M}^+(X)$.\footnote{If one does not introduce softenings, then this part of the argument breaks down and $\rho_x$ need not be a finite measure.}

By the definition in~(\ref{eqn:disintegrate}) and since $\frac{d\nu(\cdot \mid U_i)}{d\nu}=\frac{1}{\nu(U_i)}\cdot I_{U_i}$ for $i$ in $I^+$, one has the following for all $x$ in $X$:
\begin{equation*}
\mathbb{E}_{\nu}[f\mid \mathscr{F}](x) = \sum_{i\in I^+} \big(\int_X f(v) \; d\nu(\cdot \mid U_i)(v)\big) \cdot I_{U_i}(x)  =  \sum_{i\in I^+} \bigg(\frac{1}{\nu(U_i)} \int_{U_i} f(v) \; d\nu(v)\bigg) \cdot I_{U_i}(x)
\end{equation*}
If $j$ in $I^+$, then when we integrate over $x$ in $A_j$ with respect to $\nu$ we then get
\begin{equation*}
\int_{A_j} \mathbb{E}_{\nu}[f\mid \mathscr{F}](x) \; d\nu(x) = \int_{A_j} \frac{1}{\nu(U_j)} \int_{U_j} f(v) \; d\nu(v) \; d\nu(x) =\int_{A_j} f(v) \; d\nu(v)
\end{equation*}
If $j$ not in $I^+$ then the event $A_j$ is $\nu$-null and so trivially we get:
\begin{equation*}
\int_{A_j} \mathbb{E}_{\nu}[f\mid \mathscr{F}](x) \; d\nu(x) =  \int_{A_j} f(v) \; d\nu(v)
\end{equation*}
Since elements $A_j$ generate $\mathscr{F}$, this shows that (\ref{eqn:partitionex}) is a version of the conditional expectation of $f$ with respect to $\mathscr{F}$. Further, it is totally defined for all $f$ in $\mathbb{L}_1(\nu)$ and all $x$ in $X$.

Since $\mathscr{F}$ is an effective partition, one has that $[x]_{\mathscr{F}} = A_i$ for $x$ in $A_i$. Further, for $x$ in $\mathsf{KR}^{\nu}\cap A_i$, we have that $i$ in $I^+$. Hence for $x$ in $\mathsf{KR}^{\nu}\cap A_i$ we have $\rho_x([x]_{\mathscr{F}}\cap \mathsf{KR}^{\nu}) =\rho_x(A_i\cap \mathsf{KR}^{\nu})=
\nu(A_i\cap \mathsf{KR}^{\nu}\mid U_i) = \nu(A_i\mid U_i)=1$. The same argument works for $\mathsf{SR}^{\nu}$ and $\mathsf{MLR}^{\nu}$.

Suppose that $U$ is c.e. open and $q\geq 0$ is rational. Since the $U_i$ are pairwise disjoint, one has that $\rho_x(U)>q$ iff there is $i$ in $I^+$ such that $x$ is in $U_i$ and $\nu(U\mid U_i)>q$, which is a c.e. open condition in variable $x$.
\end{proof}

Here is how to organise a suitably generalised version of the initial step of the filtration of Example~\ref{ex:refinedlines}, that is, where the partition on the second component consists just of a single set (like $\mathscr{G}_1$ in Example~\ref{ex:refinedlines}).
\begin{prop}\label{prop:linesdisintegration} 
Suppose that $Y,Z$ are computable Polish spaces. Suppose that $\lambda$ is a computable point of $\mathcal{P}(Y)$ and $\mu$ is a computable point of $\mathcal{P}(Z)$. Let $\mathscr{F}$ be the $\lambda\otimes\mu$-effective $\sigma$-algebra on $Y\times Z$ generated by sets of the form $U\times Z$, where $U$ ranges over c.e. opens from a $\lambda$-computable basis on $Y$. Then $\rho: Y\times Z\rightarrow \mathcal{P}(Y\times Z)$ given by $\rho_{(u,v)} = \delta_u \otimes \mu$ is a Martin-L\"of disintegration of $\mathscr{F}$ with respect to $\lambda \otimes \mu$, and one has the following expression for the version of conditional expectation, which for each $f$ in $\mathbb{L}_1(\lambda\otimes\mu)$ is defined for $(\lambda \otimes \mu)$-a.s. many $(u,v)$ in $Y\times Z$:
\begin{equation}\label{eqn:linesdisintegration}
\mathbb{E}_{\lambda\otimes\mu} [f \mid \mathscr{F}](u,v)=\int_Z f(u,t) \; d\mu(t)
\end{equation}
\end{prop}
\begin{proof}
By the definition in (\ref{eqn:disintegrate}) and by Fubini-Tonelli, one has the following for $f$ in $\mathbb{L}_1(\lambda\otimes \mu)$, and by the same theorem it is defined for $\lambda$-a.s. many $u$ in $Y$, and hence for $(\lambda \otimes \mu)$-a.s. many $(u,v)$ in $Y\times Z$:
\begin{align*}
 \mathbb{E}_{\lambda\otimes\mu} [f \mid \mathscr{F}](u,v)=\int_{Y\times Z} f(s,t) \; d\rho_{(u,v)}(s,t) =\int_Z \int_Y f(s,t) \; d\delta_u(s) \; d\mu(t) = \int_Z f(u,t) \; d\mu(t)
\end{align*}
Since $v$ from $Z$ does not appear free in this last term, when we integrate with respect to $v$ in $Z$ we get:
\begin{equation*}
\int_Z \mathbb{E}_{\lambda\otimes\mu} [f \mid \mathscr{F}](u,v)\; d\mu(v) = \int_Z f(u,t) \; d\mu(t)
\end{equation*}
Hence for Borel subsets $B$ of $Y$ we have by Fubini-Tonelli that:
\begin{align*}
 \int_{B\times Z} \mathbb{E}_{\lambda\otimes\mu} [f \mid \mathscr{F}](u,v) \; d(\lambda\otimes \mu)(u,v) 
& =  \int_B\int_Z f(u,t) \; d\mu(t)\; d\lambda(u) \\
& = \int_{B\times Z} f(u,t) \; d(\lambda\otimes \mu)(u,t)
\end{align*}
This shows that (\ref{eqn:linesdisintegration}) is a version of the condition expectation of $f$ with respect to~$\mathscr{F}$.

Note that $[(u,v)]_{\mathscr{F}}= \{u\}\times Z$, for any $(u,v)\in Y\times Z$. Further, recall that $\mathsf{KR}^{\lambda\otimes \mu}(Y\times Z)=\mathsf{KR}^{\lambda}(Y)\times \mathsf{KR}^{\mu}(Z)$.\footnote{One can easily check this by hand. It also follows from the fact that Kurtz randomness is preserved both ways under computable continuous open maps.} Hence for $(u,v)$ in $\mathsf{KR}^{\lambda\otimes \mu}(Y\times Z)$ one has the identity:
\begin{equation*}
[(u,v)]_{\mathscr{F}} \cap \mathsf{KR}^{\lambda\otimes \mu}(Y\times Z) = (\{u\}\cap \mathsf{KR}^{\lambda}(Y))\times (Z\cap  \mathsf{KR}^{\mu}(Z)) = \{u\}\times \mathsf{KR}^{\mu}(Z)
\end{equation*}
From this we get $\rho_{(u,v)}([(u,v)]_{\mathscr{F}}\cap \mathsf{KR}^{\lambda\otimes \mu}(Y\times Z)) = \delta_u(\{u\}) \cdot \mu(\mathsf{KR}^{\mu}(Z))=1$.

In this next paragraph, we use some notation familiar from Fubini-Tonelli, namely if $A\subseteq Y\times Z$ and $s$ in $Y$, then $A_s$ is defined to be $\{t\in Z: (s,t)\in A\}$.

For Schnorr disintegrations, suppose that $(u,v)$ is in $\mathsf{SR}^{\lambda \otimes \mu}(Y\times Z)$, so that $u$ is in $\mathsf{SR}^{\lambda}(Y)$. Since $[(u,v)]_{\mathscr{F}}= \{u\}\times Z$, we want to show that $(\delta_u\otimes \mu)(A)=1$, where $A$ is the event $(\{u\}\times Z)\cap \mathsf{SR}^{\lambda \otimes \mu}(Y\times Z)$. We have $A_u=\{t\in Z: (u,t)\in \mathsf{SR}^{\lambda \otimes \mu}(Y\times Z)\}$. By choosing a Turing degree ${\bf a}$ which computes a fast Cauchy sequence for $u$, we have by van Lambalgen's Theorem that $A_u\supseteq \mathsf{SR}^{\mu, {\bf a}}(Z)$ and so $\mu(A_u)=1$.\footnote{This ``hard'' direction of van Lambalgen's Theorem works in arbitrary computable Polish spaces with computable measures, basically because it is Fubini-Tonelli type argument. It similarly works for $\mathsf{MLR}^{\nu}$. For the setting of Cantor space with uniform measure, see discussion in \cite[pp. 257-258, 357]{Downey2010aa}.} Then by Fubini-Tonelli $\rho_{(u,v)}(A)=(\delta_u\otimes \mu)(A)=\int_Y \mu(A_s)\; d\delta_u(s) = \mu(A_u) =1$, which is what we wanted to show. The argument for Martin-L\"of disintegrations is similar.

Suppose that $W\subseteq Y\times Z$ is c.e. open. Then we can write $W=\bigcup_i U_i\times V_i$ where $U_i\subseteq Y, V_i\subseteq Z$ are computable sequences of c.e. opens with $U_i\subseteq U_{i+1}$ and $V_i\subseteq V_{i+1}$. Then for rational $q\geq 0$, we have $\rho_{u,v}(W)>q$ iff there is $i\geq 0$ with $\rho_{(u,v)}(U_i\times V_i)>q$, which happens iff there is $i\geq 0$ with $\delta_u(U_i)\cdot \mu(V_i)>q$, which happens iff there is $i\geq 0$ with $u$ in $U_i$ and $\mu(V_i)>q$. This is a c.e. open condition in variables $(u,v)$.
\end{proof}

Finally, we can combine partitions and lines as follows, which gives a suitably generalised version of an individual step in the filtration from Example~\ref{ex:refinedlines} (like $\mathscr{G}_2$ or $\mathscr{G}_3$ in that example).
\begin{prop}\label{prop:combinepartitionslines1}
Suppose that $Y,Z$ are computable Polish spaces. Suppose that $\lambda$ is a computable point of $\mathcal{P}(Y)$ and $\mu$ is a computable point of $\mathcal{P}(Z)$. 

Suppose $\{C_i: i\in I\}$ is an effective partition of $Z$ with effective softening $\{V_i: i\in I\}$.

Let $\mathscr{F}$ be the $\lambda\otimes\mu$-effective $\sigma$-algebra on $Y\times Z$ generated by sets of the form $U\times C_i$, where $U$ ranges over c.e. opens from a $\lambda$-computable basis on $Y$. Then the map $\rho: Y\times Z\rightarrow \mathcal{M}^+(Y\times Z)$ given by $\rho_{(u,v)} = \sum_{i\in I^+} \big(\delta_u\otimes \mu(\cdot \mid V_i)\big)\cdot I_{V_i}(v)$ is a Martin-L\"of disintegration of $\mathscr{F}$ with respect to $\lambda \otimes \mu$, and one has the following expression for the version of conditional expectation:
\begin{equation}\label{eqn:combinepartitionslines1}
\mathbb{E}_{\lambda\otimes\mu} [f \mid \mathscr{F}](u,v)= \sum_{i\in I^+} \bigg(\frac{1}{\mu(V_i)} \int_{V_i} f(u,t) \; d\mu(t)\bigg) \cdot I_{V_i}(v)
\end{equation}
\end{prop}
\begin{proof}
This proof is just a combination of the proofs of Proposition~\ref{prop:partitionsaredisintegrations} and Proposition~\ref{prop:linesdisintegration}.
\end{proof}

Another variant on Example~\ref{ex:refinedlines} is 
\begin{prop}\label{ex:linescountableproduct} 
Let $X=\prod X_i$, where $X_i$ is a computable sequence of computable Polish spaces. Let a computable point $\nu$ of $\mathcal{P}(X)$ be given by $\nu=\bigotimes_i \nu_i$, where $\nu_i$ is a computable sequence in $\mathcal{P}(X_i)$. Let $n\geq 1$. Let $\mathscr{F}_n$ be the $\sigma$-algebra on $X$ generated by sets of the form $\prod_{i\leq n} V_i\times \prod_{i>n} X_i$, where $V_i$ for $i\leq n$ ranges over c.e. opens from a $\nu_i$-computable basis on $X_i$. For $x$ in $X$, write its coordinates as $x=(x_1, x_2, \ldots)$. Then $\rho^{(n)}: X\rightarrow \mathcal{P}(X)$ given by $\rho_x^{(n)} = (\otimes_{i\leq n} \delta_{x_i}) \otimes (\otimes_{i>n} \nu_i)$ is a Martin-L\"of disintegration of $\mathscr{F}_n$ with respect to $\nu$, and one has the following expression for the version of conditional expectation, which for each $f$ in $\mathbb{L}_1(\nu)$ is defined for $\nu$-a.s. many $x$ from $X$, and where $x=(x_1, x_2, \ldots)$ and $t=(t_{n+1}, t_{n+2}, \ldots)$
\begin{equation*}
\mathbb{E}_{\nu} [f \mid \mathscr{F}_n](x)=\int_{\prod_{i>n} X_i} f(x_1, \ldots, x_n, \overline{t}) \; d(\otimes_{i>n} \nu_i)(\overline{t})
\end{equation*}
\end{prop}
\begin{proof}
Simply apply Proposition~\ref{prop:linesdisintegration} to $Y\times Z$, where $Y=\prod_{i\leq n} X_i$ and $Z=\prod_{i>n} X_i$ and $\lambda = \otimes_{i\leq n} \nu_i$ and $\mu=\otimes_{i>n} \nu_i$.
\end{proof}

The simplest kind of an effective filtration, which occurs in both Examples~\ref{ex:refinedpartition:alt}-\ref{ex:refinedlines:alt} is the following:
\begin{defn}
Suppose that $X$ is a computable Polish space. Suppose that $T\subseteq \mathbb{N}^{<\mathbb{N}}$ is a computable tree with no dead ends. Let $I_n=\{\sigma\in T: \left|\sigma\right|=n\}$. Suppose that $\mathscr{F}_n$ is an effective partition $\{A_{\sigma}: \sigma\in I_n\}$, uniformly in $n\geq 0$. If the partitions refine one another, in that $A_{\sigma}=\bigcup_{\sigma^{\frown}(j)\in T} A_{\sigma^{\frown}(j)}$ for all $\sigma$ in $T$, then the $\mathscr{F}_n$ is an effective filtration, which we call an \emph{effective refined partition}.
\end{defn}

The following isolates the natural sufficient condition for an effective refined partition to be full, and this condition is obviously met in Example~\ref{ex:refinedpartition:alt}:
\begin{prop}\label{prop:effectshrinking}
Suppose that the effective refined partition $\mathscr{F}_n$ satisfies the following properties:
\begin{itemize}[leftmargin=*]
\item \emph{Effectively Shrinking}: There is a computable function $\ell:\mathbb{Q}^{>0}\rightarrow \mathbb{N}$ such that for all rational $\epsilon>0$ one has that $\mathrm{diam}(A_{\sigma})<\epsilon$ for all $\sigma$ in $T$ with $\left|\sigma\right|\geq \ell(\epsilon)$.
\item \emph{Effectively non-empty}: There is a uniformly computable sequence of points $x_{\sigma}$ in $A_{\sigma}$.
\end{itemize}
Then the effective refined partition $\mathscr{F}_n$ is full.
\end{prop}
\begin{proof}
(Sketch) The two conditions imply that there is a well-defined computable continuous surjection $\pi:[T]\rightarrow X$ given by $\pi(\omega)=x$ iff $\{x\}=\bigcap_n A_{\omega\upharpoonright n}$. For fullness, suppose that $U\subseteq X$ is c.e. open. Then $\pi^{-1}(U)$ is c.e. open. Then there is c.e. set $S\subseteq T$ such that $\pi^{-1}(U)=\bigcup_{\sigma\in S} [T]\cap [\sigma]$. Then one can check that $U=\bigcup_{\sigma\in S} A_{\sigma}$.
\end{proof}

We can also obtain full effective filtrations by combining lines with full effective partitions, as in \ref{ex:refinedlines:alt}:
\begin{ex}\label{ex:combinepartitionslines2}

Suppose that $Y,Z$ are computable Polish spaces. Suppose that $\lambda$ is a computable point of $\mathcal{P}(Y)$ and suppose that $\mu$ is a computable point of $\mathcal{P}(Z)$. 

Suppose that $\mathscr{F}_n$ is a full effective partition of $Z$ (resp. almost-full effective partition of $Z$).

Let $\mathscr{G}_n$ be the effective $\sigma$-algebra $\{U\times A: U\subseteq Y \mbox{ c.e. open } \; \& \; A\in \mathscr{F}_n\}$.

Then $\mathscr{G}_n$ is a full effective filtration of $Y\times Z$ (resp. almost-full effective filtration of $Y\times Z$).
\end{ex}

Another example of a full effective filtration is related to countable products:
\begin{ex}\label{ex:linescountableproduct2}
The effective $\sigma$-algebras $\mathscr{F}_n$ from the countable products Example~\ref{ex:linescountableproduct} is a full effective filtration. This is because every c.e. open $W\subseteq X=\prod_i X_i$ can be uniformly written as $\bigcup_{\sigma\in J} W_{\sigma}$, for some c.e. index set $J$, where $W_{\sigma}=\prod_{i<\left|\sigma\right|} V_{\sigma(i)}\times \prod_{i\geq \left|\sigma\right|} X_i$, where $V_{\sigma(i)}$ is uniformly c.e. open in $X_i$ for $i\leq n$.
\end{ex}
Of course, this example is the same as that of full effective partitions when $X_i$ is uniformly countable. However, this example goes beyond that of full effective partitions when the $X_i$ are uncountable.

\newpage

\bibliographystyle{amsplain}
\bibliography{00-HWZB-project-01.bib}

\end{document}